\documentclass{amsart}

\usepackage[T1]{fontenc}
\usepackage[latin1]{inputenc}
\usepackage{amsmath,amsthm,amstext,amsfonts,amssymb,amscd,amsxtra}
\usepackage{bm}
\usepackage{enumitem}
\usepackage{eucal}
\usepackage{mathrsfs}
\usepackage{stmaryrd}
\usepackage[all]{xy}
\usepackage{url}
\usepackage{xcolor}

\allowdisplaybreaks

\theoremstyle{plain}
\newtheorem{theorem}{Theorem}[section]

\newtheorem*{theoremB}{Theorem A}
\newtheorem*{corollaryC}{Theorem B}

\newtheorem{lemma}[theorem]{Lemma}
\newtheorem{proposition}[theorem]{Proposition}
\newtheorem{proposition-definition}[theorem]{Proposition-Definition}
\newtheorem{corollary}[theorem]{Corollary}
\newtheorem{claim}[theorem]{Claim}

\theoremstyle{definition}
\newtheorem{definition}{Definition}[section]

\theoremstyle{remark}
\newtheorem{remark}[theorem]{Remark}

\newcommand{\NN}{\mathbb{N}}
\newcommand{\ZZ}{\mathbb{Z}}
\newcommand{\QQ}{\mathbb{Q}}
\newcommand{\RR}{\mathbb{R}}
\newcommand{\CC}{\mathbb{C}}

\newcommand{\KK}{\mathbb{K}}

\makeatletter
 
 \@addtoreset{equation}{section}
\makeatother

\makeatletter
 \newcommand{\esssup}{\mathop{\operator@font ess.sup}\displaylimits}
 \newcommand{\essinf}{\mathop{\operator@font ess.inf}\displaylimits}
\makeatother

\makeatletter
 \def\BIG#1{%
  {\hbox{$\left#1\vbox to20.5\p@{}\right.\n@space$}}}
 \def\BIGl{\mathopen\BIG}
 
 \def\BIGr{\mathclose\BIG}
\makeatother

\makeatletter 
 \let\@@pmod\pmod 
 \DeclareRobustCommand{\pmod}{\@ifstar\@pmods\@@pmod} 
 \def\@pmods#1{\mkern4mu({\operator@font mod}\mkern 6mu#1)} 
\makeatother

\renewcommand{\leq}{\leqslant}
\renewcommand{\geq}{\geqslant}
\renewcommand{\mod}{\mathop{\mathrm{mod}}}

\DeclareMathOperator{\Ker}{Ker}

\DeclareMathOperator{\aHz}{\widehat{\Gamma}^{\rm ss}}

\DeclareMathOperator{\aHzf}{\widehat{\Gamma}^{\rm f}}

\DeclareMathOperator{\trdeg}{tr.\! deg}
\DeclareMathOperator{\depth}{depth}
\DeclareMathOperator{\DV}{\mathfrak{V}}
\DeclareMathOperator{\ratrk}{rat.\!rk}
\DeclareMathOperator{\codim}{codim}

\DeclareMathOperator{\Spec}{Spec}

\DeclareMathOperator{\Supp}{Supp}
\DeclareMathOperator{\CSupp}{Supp_C}
\DeclareMathOperator{\WSupp}{Supp_W}

\DeclareMathOperator{\Rat}{Rat}
\DeclareMathOperator{\Exc}{Ex}

\DeclareMathOperator{\SBs}{\mathbf{B}}
\DeclareMathOperator{\aSBs}{\widehat{\mathbf{B}}^{\rm ss}}
\DeclareMathOperator{\Bsp}{\mathbf{B}_+}

\DeclareMathOperator{\Div}{Div}
\DeclareMathOperator{\WDiv}{WDiv}

\DeclareMathOperator{\Bl}{Bl}

\DeclareMathOperator{\aTheta}{\widehat{\Theta}}

\DeclareMathOperator{\aBigCone}{\widehat{Big}}
\DeclareMathOperator{\aBBig}{\widehat{\mathbb{B}ig}}
\DeclareMathOperator{\aBWBig}{\widehat{W\mathbb{B}ig}}
\DeclareMathOperator{\aBDBig}{\widehat{D\mathbb{B}ig}}
\DeclareMathOperator{\aBVDiv}{\widehat{B\mathbb{D}iv}}
\DeclareMathOperator{\aBDDiv}{\widehat{D\mathbb{D}iv}}
\DeclareMathOperator{\aBVBig}{\widehat{B\mathbb{B}ig}}
\DeclareMathOperator{\VDiv}{BC}
\DeclareMathOperator{\aNef}{\widehat{Nef}}
\DeclareMathOperator{\aDiv}{\widehat{Div}}
\DeclareMathOperator{\aInt}{\widehat{Int}}
\DeclareMathOperator{\aBDiv}{\widehat{\mathbb{D}iv}}
\DeclareMathOperator{\aBWDiv}{\widehat{W\mathbb{D}iv}}

\DeclareMathOperator{\aBsp}{\widehat{\mathbf{B}}_+}
\DeclareMathOperator{\Weil}{W}

\DeclareMathOperator{\adeg}{\widehat{deg}}

\DeclareMathOperator{\ord}{ord}

\DeclareMathOperator{\vol}{vol}
\DeclareMathOperator{\avol}{\widehat{vol}}

\DeclareMathOperator{\rk}{rk}

\DeclareMathOperator{\sgn}{sgn}

\newcommand{\lex}{{\rm lex}}

\newcommand{\pr}{{\rm pr}}

\newcommand{\aDelta}{\mbox{$\widehat{\Delta}$}}

\newcommand{\sbullet}{{\scriptscriptstyle\bullet}}

\def\aSBss#1{\widehat{\mathbf{B}}^{#1}}
\def\aHzq#1{\widehat{\Gamma}^{\rm ss}_{#1}}
\def\aHzsq#1#2{\widehat{\Gamma}^{#1}_{#2}}
\def\aHzsmq#1{\widehat{\Gamma}^{\rm s}_{#1}}

\def\aSpan#1#2{\langle#2\rangle_{#1}}

\title[Adelic Cartier divisors with base conditions]{Adelic Cartier divisors with base conditions and the Bonnesen--Diskant--type inequalities}
\author{Hideaki Ikoma}
\thanks{This research is supported by JSPS KAKENHI 25$\cdot$1895.}
\address{Graduate School of Mathematical Sciences, The University of Tokyo, Tokyo, 153-8914, Japan}
\email{ikoma@ms.u-tokyo.ac.jp}
\subjclass{Primary 14G40; Secondary 11G50}
\keywords{Arakelov theory, adelic divisors, base conditions, arithmetic volumes}
\begin{document}

\begin{abstract}
In this paper, we introduce positivity notions for pairs of adelic $\RR$-Cartier divisors and $\RR$-base conditions, and study fundamental properties of the arithmetic volumes defined for such pairs.
We show that the G\^ateaux derivatives of the arithmetic volume function at big pairs along the directions of adelic $\RR$-Cartier divisors are given by suitable arithmetic positive intersection numbers.
As a corollary, we obtain an Arakelov theoretic analogue of the Bonnesen--Diskant inequality in convex geometry.
\end{abstract}

\maketitle
\tableofcontents

\section{Introduction}\label{sec:Intro}

Let $X$ be a normal projective variety that is geometrically irreducible over a number field $K$, and let $\Rat(X)$ be the field of rational functions on $X$.
We freely use the definition and basic properties of the adelic $\RR$-Cartier divisors, and refer to \cite{MoriwakiAdelic} for details (see also Notation and terminology~4).
To an adelic $\RR$-Cartier divisor $\overline{D}$ on $X$, we assign a finite set of all the \emph{strictly small} sections of $\overline{D}$,
\[
 \aHz(\overline{D}):=\left\{\phi\in\Rat(X)^{\times}\,:\,\overline{D}+\widehat{(\phi)}>0\right\}\cup\{0\},
\]
and define the \emph{arithmetic volume} of $\overline{D}$ as
\[
 \avol(\overline{D}):=\limsup_{\substack{m\in\NN, \\ m\to+\infty}}\frac{\log\sharp\aHz(m\overline{D})}{m^{\dim X+1}/(\dim X+1)!}.
\]
The arithmetic Siu inequality of Yuan \cite{Yuan07} is essentially equivalent to the fact that the G\^ateaux derivatives of the arithmetic volume function at big adelic $\RR$-Cartier divisors are given by the arithmetic positive intersection numbers (see \cite{Chen11,IkomaCon}).
It also implies the equidistribution theorem of algebraic points with small heights, and has fruitful applications to arithmetic dynamical systems.

The purpose of this paper is to introduce the notion of pairs of adelic $\RR$-Cartier divisors and $\RR$-base conditions, and to study their positivity properties.
We show that the above-mentioned result on the differentiability of the arithmetic volume function can be naturally generalized to the arithmetic volume function defined for such pairs.

Let $\DV(\Rat(X))$ be the set of all the (nontrivial) normalized discrete valuations of $\Rat(X)$.
An \emph{$\RR$-base condition} on $X$ is defined as a finite formal sum
\[
 \mathcal{V}=\sum_{\nu\in\DV(\Rat(X))}\nu(\mathcal{V})[\nu]
\]
with coefficients $\nu(\mathcal{V})$ in $\RR$.
We denote by $\aBVDiv_{\RR,\RR}(X)$ the $\RR$-vector space of all the pairs of adelic $\RR$-Cartier divisors and $\RR$-base conditions.
As in the case of adelic $\RR$-Cartier divisors, we can assign to such a pair $(\overline{D};\mathcal{V})\in\aBVDiv_{\RR,\RR}(X)$ a finite set of all the strictly small sections of $\overline{D}$ vanishing along the positive part of $\mathcal{V}$; namely
\[
 \aHz(\overline{D};\mathcal{V}):=\left\{\phi\in\Rat(X)^{\times}\,:\,\overline{D}+\widehat{(\phi)}>0,\,\nu_X(D+(\phi))\geq \nu(\mathcal{V})\right\}\cup\left\{0\right\}
\]
(see sections~\ref{subsec:Prelim} and \ref{subsec:Base_Cond} for detail).
We then define the arithmetic volume of $(\overline{D};\mathcal{V})$ as
\[
 \avol(\overline{D};\mathcal{V}):=\limsup_{\substack{m\in\NN, \\ m\to+\infty}}\frac{\log\sharp\aHz(m\overline{D};m\mathcal{V})}{m^{\dim X+1}/(\dim X+1)!}.
\]
We say that a pair $(\overline{D};\mathcal{V})$ is \emph{big} if there exists a \emph{weakly ample} adelic $\RR$-Cartier divisor $\overline{A}$ on $X$ (see Notation and terminology~5) such that $\overline{D}-\overline{A}$ is strictly effective and $\nu_X(D-A)\geq\nu(\mathcal{V})$ for every $\nu\in\DV(\Rat(X))$.
Our main theorem is then stated as follows.

\begin{theoremB}[Theorem~\ref{thm:diff_along_arith}]
Let $X$ be a normal projective variety over a number field, let $(\overline{D};\mathcal{V})\in\aBVDiv_{\RR,\RR}(X)$, and let $\overline{D}'$ be an adelic $\RR$-Cartier divisor on $X$.
If $(\overline{D};\mathcal{V})$ is big, then the G\^ateaux derivative of the arithmetic volume function at $(\overline{D};\mathcal{V})$ along $\overline{D}'$ is given by the formula
\[
 \lim_{r\to 0}\frac{\avol(\overline{D}+r\overline{D}';\mathcal{V})-\avol(\overline{D};\mathcal{V})}{r}=(\dim X+1)\cdot\langle(\overline{D};\mathcal{V})^{\cdot\dim X}\rangle\cdot\overline{D}'.
\]
\end{theoremB}

The right-hand side of the formula is the arithmetic positive intersection number defined for pairs (see section~\ref{subsec:aPosIntNum} for detail).
An \emph{approximation} of a big pair $(\overline{D};\mathcal{V})$ is a couple $(\mu:X'\to X,\overline{M})$ consisting of a modification $\mu:X'\to X$ and a nef and big adelic $\RR$-Cartier divisor $\overline{M}$ on $X'$ such that $(\mu^*\overline{D}-\overline{M};\mathcal{V}^{\mu})$ is pseudo-effective (see (\ref{eqn:pull-back_base_cond}) and Definition~\ref{defn:positivity}).
We denote by $\aTheta(\overline{D};\mathcal{V})$ the set of all the approximations of $(\overline{D};\mathcal{V})$.
For a nef and big adelic $\RR$-Cartier divisor $\overline{N}$, we define
\[
 \langle(\overline{D};\mathcal{V})^{\cdot\dim X}\rangle\cdot\overline{N}:=\sup_{(\mu,\overline{M})\in\aTheta(\overline{D};\mathcal{V})}\adeg\left(\overline{M}^{\cdot\dim X}\cdot\mu^*\overline{N}\right),
\]
which we can extend by linearity and continuity to
\[
 \langle(\overline{D};\mathcal{V})^{\cdot\dim X}\rangle\cdot:\aDiv_{\RR}(X)\to\RR
\]
(see Definition~\ref{defn:aPosInt}).

It is known that, if $X$ is a smooth curve and $(\overline{D};\mathcal{V})$ is big, then the ordered set
\[
 \Upsilon(\overline{D};\mathcal{V}):=\left\{\overline{P}\,:\,\text{$\overline{P}$ is nef and $(\overline{D}-\overline{P};\mathcal{V})\geq 0$}\right\}
\]
admits a unique maximal element $\overline{P}(\overline{D};\mathcal{V})$ (see \cite{MoriwakiAdelic}).
In this case, an arithmetic positive intersection number is given by
\[
 \langle(\overline{D};\mathcal{V})\rangle\cdot\overline{D}'=\adeg\left(\overline{P}(\overline{D};\mathcal{V})\cdot\overline{D}'\right)
\]
for every adelic $\RR$-Cartier divisor $\overline{D}'$.

In the context of convex geometry, T. Bonnesen gave a systematic proof to the classical isoperimetric inequality in dimension two by showing a stronger inequality called Bonnesen's inequality (see \cite{Bonnesen29}).
The method was generalized to the case of arbitrary dimensions by V. I. Diskant (see \cite{TeissierBonn, Diskant}).
Analogous inequalities in the context of algebraic geometry were established by Boucksom--Favre--Jonsson (see \cite{Bou_Fav_Mat06,Cutkosky13}), and those in the context of Arakelov geometry are in \cite{IkomaCon}.
These inequalities are important in studying the properties of the volume functions and the Zariski decompositions of divisors.
As a corollary of Theorem~A, we can generalize the Bonnesen--Diskant--type inequalities (\cite[Theorem~7.1 and Corollary~7.3]{IkomaCon}) to the case of pairs.

\begin{corollaryC}[Theorem~\ref{thm:Diskant}]
Let $X$ be a normal projective variety over a number field, and let $(\overline{D}_1;\mathcal{V}_1),(\overline{D}_2,\mathcal{V}_2)\in\aBVDiv_{\RR,\RR}(X)$ be big pairs.
We set
\[
 s_i:=\langle(\overline{D}_1;\mathcal{V}_1)^{\cdot i}\cdot(\overline{D}_2;\mathcal{V}_2)^{\cdot(\dim X+1-i)}\rangle
\]
for $i=0,\dots,\dim X+1$,
\[
 r=r((\overline{D}_1;\mathcal{V}_1),(\overline{D};\mathcal{V}_2)):=\inf_{(\mu,\overline{M})\in\aTheta(\overline{D}_2;\mathcal{V}_2)}\sup\left\{t\in\RR\,:\,(\mu^*\overline{D}_1-t\overline{M};\mathcal{V}_1^{\mu})\succeq 0\right\},
\]
and
\[
 R=R((\overline{D}_1;\mathcal{V}_1),(\overline{D}_2;\mathcal{V}_2)):=\frac{1}{r((\overline{D}_2;\mathcal{V}_2),(\overline{D}_1;\mathcal{V}_1))}.
\]
One then has
\begin{enumerate}
\item (an arithmetic Diskant inequality)
\[
 0\leq\left(s_{\dim X}^{\frac{1}{\dim X}}-rs_0^{\frac{1}{\dim X}}\right)^{\dim X+1}\leq s_{\dim X}^{1+\frac{1}{\dim X}}-s_{\dim X+1}\cdot s_0^{\frac{1}{\dim X}},
\]
\item
\begin{multline*}
 \frac{s_{\dim X}^{\frac{1}{\dim X}}-\left(s_{\dim X}^{1+\frac{1}{\dim X}}-s_{\dim X+1}\cdot s_0^{\frac{1}{\dim X}}\right)^{\frac{1}{\dim X+1}}}{s_0^{\frac{1}{\dim X}}}\leq r \\
 \leq\frac{s_{\dim X+1}}{s_{\dim X}}\leq\dots\leq\frac{s_1}{s_0}\\
 \leq R\leq\frac{s_{\dim X+1}^{\frac{1}{\dim X}}}{s_1^{\frac{1}{\dim X}}-\left(s_1^{1+\frac{1}{\dim X}}-s_0\cdot s_{\dim X+1}^{\frac{1}{\dim X}}\right)^{\frac{1}{\dim X+1}}},
\end{multline*}
and
\item (an arithmetic Bonnesen inequality)
\[
 \left(\frac{s_0}{2}(R-r)\right)^2\leq s_1^2-s_0s_2
\]
if $X$ has dimension one.
\end{enumerate}
In particular, if the big pairs $(\overline{D}_1;\mathcal{V}_1),(\overline{D}_2,\mathcal{V}_2)\in\aBVDiv_{\RR,\RR}(X)$ satisfy
\[
 \avol(\overline{D}_1+\overline{D}_2;\mathcal{V}_1+\mathcal{V}_2)^{\frac{1}{\dim X+1}}=\avol(\overline{D}_1;\mathcal{V}_1)^{\frac{1}{\dim X+1}}+\avol(\overline{D}_2;\mathcal{V}_2)^{\frac{1}{\dim X+1}},
\]
then $s_{\dim X}^{\dim X+1}=s_{\dim X+1}^{\dim X}\cdot s_0$, $s_1^{\dim X+1}=s_0^{\dim X}\cdot s_{\dim X}$, and
\[
 \left(\frac{s_{\dim X}}{s_0}\right)^{\frac{1}{\dim X}}=r=\frac{s_{\dim X+1}}{s_{\dim X}}=\dots=\frac{s_1}{s_0}=R=\left(\frac{s_{\dim X+1}}{s_1}\right)^{\frac{1}{\dim X}}.
\]
\end{corollaryC}

The structure of this paper is as follows.
After defining of the notation and terminology we use in this paper in section~\ref{subsec:notation_and_terminology}, we introduce, in section~\ref{sec:adelic_with_base}, the notion of pairs of adelic $\RR$-Cartier divisors and $\RR$-base conditions and their positivity properties.
Here we would like to treat the positivity and the $\nu$-positivity, simultaneously, because the arguments are almost parallel.
The latter will be used elsewhere.
The main purpose of section~\ref{sec:adelic_with_base} is Theorem~\ref{thm:abbig_cone} asserting the openness of the big cone of pairs.

Sections~\ref{subsec:ample} and \ref{subsec:arith_base_loci} are devoted to studying fundamental properties of the arithmetically ample adelic $\RR$-Cartier divisors.
We give the definitions of arithmetic volumes and arithmetic base loci associated to pairs in section~\ref{subsec:pairs} and in section~\ref{subsec:arith_base_loci}, respectively.

In section~\ref{sec:Approx_pairs}, we show several preliminary results that will be used to show the main theorems.
In particular, we introduce the arithmetic positive intersection numbers for pairs in section~\ref{subsec:aPosIntNum}.
Finally, we give proofs of Theorems~\ref{thm:diff_along_arith} and \ref{thm:Diskant} in section~\ref{sec:Diff}.

\subsection{Notation and terminology}\label{subsec:notation_and_terminology}

\paragraph{1.}
Let $R$ be a ring, let $M$ be an $R$-module, and let $\Gamma$ be a subset of $M$.
We denote by $\langle\Gamma\rangle_R$ the $R$-submodule of $M$ spanned by $\Gamma$.

\paragraph{2.}
Let $X$ be a projective variety over a field $k$ of characteristic zero.
We denote the field of rational functions on $X$ by $\Rat(X)$.
Let $\KK$ be either a blank, $\QQ$ or $\RR$.
The $\KK$-module of all the $\KK$-Cartier divisors (respectively, $\KK$-Weil divisors) on $X$ is denoted by $\Div_{\KK}(X)$ (respectively, by $\WDiv_{\KK}(X)$).

Let $D$ be an $\RR$-Cartier divisor on $X$; namely, $D$ can be written as an $\RR$-linear combination
\[
 D=\sum_{i=1}^la_iD_i
\]
with $D_i\in\Div(X)$ and $a_i\in\RR$.
A \emph{local equation} defining $D$ around a point $x\in X$ is
\begin{equation}
 f_x:=f_1^{\otimes a_i}\otimes\dots\otimes f_l^{\otimes a_l}\in\Rat(X)^{\times}\otimes_{\ZZ}\RR,
\end{equation}
where $f_i$ is a local equation defining $D_i$ around $x$ for each $i$.

The \emph{(Cartier) support} of $D\in\Div_{\RR}(X)$ is defined as
\begin{equation}
 \CSupp(D):=\left\{x\in X\,:\,f_x\notin\mathcal{O}_{X,x}^{\ast}\otimes_{\ZZ}\RR\right\},
\end{equation}
which we know is a proper Zariski closed subset of $X$ (see \cite[Proposition~1.1.1]{MoriwakiAdelic}).

Let $\mu:Y\to X$ be a morphism of $k$-varieties.
If $\mu(Y)$ is not contained in $\CSupp(D)$, then we can define the pull-back $\mu^*D\in\Div_{\RR}(Y)$.

The \emph{(Weil) support} of $D\in\WDiv_{\RR}(X)$ is defined as
\begin{equation}
 \WSupp(D):=\bigcup_{\substack{\text{$Z$: prime Weil divisor}, \\ \ord_Z(D)\neq 0}}Z.
\end{equation}
We know that, if $X$ is regular, then
\[
 \CSupp(D)=\WSupp(D)
\]
(see for example \cite[Proposition~1.1.3]{MoriwakiAdelic}).
Actually, this equation is valid as soon as $X$ is normal (see Lemma~\ref{lem:valuation_of_divisors}(2)), so that we can write
\begin{equation}
 \Supp(D):=\CSupp(D)=\WSupp(D).
\end{equation}

Suppose that $X$ is normal.
Let $D\in\WDiv_{\RR}(X)$, and let $\KK$ be either a blank, $\QQ$, or $\RR$.
We set
\begin{equation}
H^0_{\KK}(D):=\left\{\phi\in\Rat(X)\otimes_{\ZZ}\KK\,:\,D+(\phi)\geq 0\right\}\cup\{0\}.
\end{equation}

\paragraph{3.}
Let $\mu:X'\to X$ be a morphism of projective varieties over a field.
The \emph{exceptional locus} of $\mu$ is defined as the minimal Zariski closed subset of $X'$ such that the restriction
\[
 \mu:X'\setminus\Exc(\mu)\to X
\]
is an immersion (see \cite[(3.6)]{IkomaRem}).
If $X$ is normal, then, by Zariski's main theorem \cite[Proposition~(4.4.1)]{EGAIII_1}, one has
\begin{equation}
 \Exc(\mu)=\left\{x'\in X'\,:\,\dim_{x'}(\mu^{-1}(\mu(x')))\geq 1\right\}=\bigcup_{\substack{Z\subset X', \\ \dim\mu(Z)<\dim Z}}Z.
\end{equation}

\paragraph{4.}
Let $K$ be a number field, and let $O_K$ be the ring of integers of $K$.
The set of all the finite places of $K$ is denoted by $M_K$.
For each $v\in M_K$, let $K_v$ be the $v$-adic completion of $K$, let $O_{K_v}$ be the ring of integers of $K_v$, and let $\widetilde{K}_v$ be the residue field of $K_v$.

Let $X$ be a projective $K$-variety.
For each $v\in M_K$, we denote by $(X_v^{\rm an},\pi_v:X_v^{\rm an}\to X_{K_v})$ the Berkovich analytic space associated to $X_{K_v}:=X\times_{\Spec(K)}\Spec(K_v)$ and, for $v=\infty$, denote by $(X_{\infty}^{\rm an},\pi_v:X_{\infty}^{\rm an}\to X_{\CC})$ the complex analytic space associated to $X_{\CC}:=X\times_{\Spec(\QQ)}\Spec(\CC)$.

Let $D$ be an $\RR$-Cartier divisor on $X$, and let $v\in M_K\cup\{\infty\}$.
A \emph{$D$-Green function (of continuous type)} on $X_v^{\rm an}$ is a continuous map
\[
 g_v^{\overline{D}}:(X\setminus\CSupp(D))_v^{\rm an}\to\RR
\]
such that, for each $x\in X_v^{\rm an}$, $g_v^{\overline{D}}+\log|f_x|^2$ extends to a continuous function around $x$, where $f_x$ is a local equation defining $D$ around $\pi_v(x)$ (see Notation and terminology~2).

Let $U$ be a nonempty open subscheme of $\Spec(O_K)$.
A \emph{$U$-model} of $(X,D)$ is a couple $(\mathscr{X}_U,\mathscr{D}_U)$ such that $\mathscr{X}_U$ is a reduced, irreducible, projective, and flat $U$-scheme with a fixed $K$-isomorphism from $X$ onto the generic fiber $\mathscr{X}_U\times_U\Spec(K)$, and such that $\mathscr{D}_U$ is an $\RR$-Cartier divisor on $\mathscr{X}_U$ satisfying $\mathscr{D}_U|_X=D$.

Given a $U$-model $(\mathscr{X}_U,\mathscr{D}_U)$ of $(X,D)$ and a $v\in M_K\cap U$, we define the $D$-Green function $g_v^{(\mathscr{X}_U,\mathscr{D}_U)}$ associated to $(\mathscr{X}_U,\mathscr{D}_U)$ as
\begin{equation}
 g_v^{(\mathscr{X}_U,\mathscr{D}_U)}(x):=-\log |f_x'|^2,
\end{equation}
where $f_x'$ is a local equation defining $\mathscr{D}_U$ around $r^{\mathscr{X}_U}_v(x)$ and $r^{\mathscr{X}_U}_v:X_v^{\rm an}\to\mathscr{X}_U\times_U\Spec(\widetilde{K}_v)$ denotes the reduction map over $v$ (see \cite[section~1.2]{MoriwakiAdelic}).

Let $\KK$ be either a blank, $\QQ$, or $\RR$.
An \emph{adelic $\KK$-Cartier divisor} on $X$ is a couple
\[
 \overline{D}=\left(D,\sum_{v\in M_K\cup\{\infty\}}g_v^{\overline{D}}[v]\right)
\]
having the following properties.
\begin{itemize}
\item $D$ is a $\KK$-Cartier divisor on $X$,
\item For each $v\in M_K$, $g_v^{\overline{D}}$ is a $D$-Green function on $X_v^{\rm an}$.
\item For $v=\infty$, $g_{\infty}^{\overline{D}}$ is a $D$-Green function on $X_{\infty}^{\rm an}$ that is invariant under the complex conjugation.
\item There exists a nonempty open subset $U$ of $\Spec(O_K)$ and a $U$-model $(\mathscr{X}_U,\mathscr{D}_U)$ of $(X,D)$ such that $g_v^{\overline{D}}=g_v^{(\mathscr{X}_U,\mathscr{D}_U)}$ for every $v\in M_K\cap U$.
\end{itemize}
We call the $U$-model $(\mathscr{X}_U,\mathscr{D}_U)$ appearing in the above definition a \emph{$U$-model of definition for $\overline{D}$}.
We denote by $\aDiv_{\KK}(X)$ the $\KK$-module of all the adelic $\KK$-Cartier divisors on $X$.

For each $v\in M_K$, $C^0_v(X)$ denotes the $\RR$-vector space of all the $\RR$-valued continuous functions on $X_v^{\rm an}$, and, for $v=\infty$, $C^0_{\infty}(X)$ denotes the $\RR$-vector space of all the $\RR$-valued continuous functions on $X_{\infty}^{\rm an}$ that are invariant under the complex conjugation.

\begin{description}
\item[($\RR$-linear equivalence)] For $\overline{D}_1,\overline{D}_2\in\aDiv_{\RR}(X)$, we write $\overline{D}_1\sim_{\RR}\overline{D}_2$ if there exists a $\phi\in\Rat(X)^{\times}\otimes_{\ZZ}\RR$ such that $\overline{D}_1-\overline{D}_2=\widehat{(\phi)}$.
\item[(effective)] Note that, if $D\geq 0$, then $\essinf_{x\in X_v^{\rm an}}g_v^{\overline{D}}(x)>-\infty$ for every $v\in M_K\cup\{\infty\}$ and $\essinf_{x\in X_v^{\rm an}}g_v^{\overline{D}}(x)\geq 0$ for all but finitely many $v\in M_K$.
We say that $\overline{D}$ is \emph{effective} if
\[
D\geq 0\quad\text{and}\quad g_v^{\overline{D}}\geq 0\quad\text{for every $v\in M_K\cup\{\infty\}$}.
\]
Moreover, we say that $\overline{D}$ is \emph{strictly effective} if 
\[
\text{$\overline{D}$ is effective}\quad\text{and}\quad\essinf_{x\in X_{\infty}^{\rm an}}g_{\infty}^{\overline{D}}(x)>0.
\]
\end{description}
We write
\begin{equation}
 \overline{D}_1\leq\overline{D}_2\quad \text{(respectively, $\overline{D}_1<\overline{D}_2$)}
\end{equation}
if $\overline{D}_2-\overline{D}_1$ is effective (respectively, strictly effective).

Let $\KK$ be either a blank, $\QQ$, or $\RR$.
We set
\begin{align}
 &\aHzsq{\rm f}{\KK}(\overline{D}) \\
 &\, :=\left\{\phi\in\Rat(X)\otimes_{\ZZ}\KK\,:\,\text{$D+(\phi)\geq 0$ and $g_v^{\overline{D}+\widehat{(\phi)}}\geq 0$, $\forall v\in M_K$}\right\}\cup\{0\}, \nonumber\\
 &\aHzq{\KK}(\overline{D}):=\left\{\phi\in\Rat(X)\otimes_{\ZZ}\KK\,:\,\overline{D}+\widehat{(\phi)}>0\right\}\cup\{0\},
\end{align}
and
\begin{equation}
 \aHzsmq{\KK}(\overline{D}):=\left\{\phi\in\Rat(X)\otimes_{\ZZ}\KK\,:\,\overline{D}+\widehat{(\phi)}\geq 0\right\}\cup\{0\}.
\end{equation}
We then define the \emph{arithmetic volume} of $\overline{D}$ as
\begin{equation}
 \avol(\overline{D}):=\limsup_{\substack{m\in\NN, \\ m\to+\infty}}\frac{\log\sharp\aHz(m\overline{D})}{m^{\dim X+1}/(\dim X+1)!}.
\end{equation}

\paragraph{5.}
Let $\overline{A}=\left(A,\sum_{v\in M_K\cup\{\infty\}}g_v^{\overline{A}}[v]\right)$ be an adelic $\RR$-Cartier divisor on $X$.
\begin{description}
\item[(big)] We say that $\overline{A}$ is \emph{big} if $\avol(\overline{A})>0$.
The cone of all the big adelic $\RR$-Cartier divisors on $X$ is denoted by $\aBigCone_{\RR}(X)$.
\item[(pseudo-effective)] We say that $\overline{A}$ is \emph{pseudo-effective} if $\avol(\overline{A}+\overline{B})>0$ for every $\overline{B}\in\aBigCone_{\RR}(X)$.
We write $\overline{D}_1\preceq \overline{D}_2$ if $\overline{D}_2-\overline{D}_1$ is pseudo-effective.
\item[(nef)] We say that $\overline{A}$ is \emph{relatively nef} if $A$ is nef and $g_v^{\overline{A}}$ is semipositive for every $v\in M_K\cup\{\infty\}$ (see \cite[section~4.4]{MoriwakiAdelic} for the notion of semipositivity).
We say that $\overline{A}$ is \emph{nef} if $\overline{A}$ is relatively nef and $\inf_{x\in X(\overline{K})}h_{\overline{A}}(x)\geq 0$, where
\[
 h_{\overline{A}}(x):=\frac{1}{[\kappa(x):K]}\adeg\left(\overline{A}|_x\right)
\]
is the \emph{height} of $x\in X(\overline{K})$ with respect to $\overline{A}$, $\kappa(x)$ is the residue field of the image of $x$, and
\[
 \adeg\left(\overline{A}|_x\right):=\frac{1}{2}\sum_{v\in M_K}\sum_{\substack{w\in M_{\kappa(x)}, \\ w|v}}[\kappa(x)_w:K_v]g_v^{\overline{A}}(x^w)+\frac{1}{2}\sum_{\sigma:\kappa(x)\to\CC}g_{\infty}^{\overline{A}}(x^{\sigma})
\]
(see \cite[sections~2.4 and 4.2]{MoriwakiAdelic} for detail).
We denote by $\aNef_{\RR}(X)$ the cone of all the nef adelic $\RR$-Cartier divisors on $X$.
Note that a relatively nef adelic $\RR$-Cartier divisor $\overline{A}$ is nef if and only if
\[
 \adeg\left((\overline{A}|_Y)^{\cdot(\dim Y+1)}\right)\geq 0
\]
for every closed subvariety $Y$ of $X$.
\item[(integrable)] We say that $\overline{A}$ is \emph{integrable} if $\overline{A}$ can be written as a difference of two nef adelic $\RR$-Cartier divisors.
We denote by $\aInt_{\RR}(X)$ the $\RR$-vector space of all the integrable adelic $\RR$-Cartier divisors on $X$.
\item[(w-ample)] We say that $\overline{A}$ is \emph{weakly ample} or \emph{w-ample} for short if $\overline{A}$ is a positive $\RR$-linear combination $\sum_{i=1}^la_i\overline{A}_i$ of adelic Cartier divisors $\overline{A}_i$ such that each $A_i$ is ample and such that $H^0(mA_i)=\langle\aHz(m\overline{A}_i)\rangle_{K}$ for every $m\gg 1$.
Note that this definition does not depend on the choice of $K$ (see \cite[Theorem~4.3]{IkomaRem}).
\item[(ample)] $\overline{A}$ is said to be \emph{ample} (in the sense of Zhang) if $A$ is relatively nef and
\[
 \adeg\left((\overline{A}|_Y)^{\cdot(\dim Y+1)}\right)>0
\]
for every closed subvariety $Y$ of $X$.
\end{description}

Let $\mathscr{X}$ be a normal projective arithmetic variety over $\Spec(O_K)$ such that $\mathscr{X}_K$ is $K$-isomorphic to $X$.
To an arithmetic $\RR$-Cartier divisor $\overline{\mathscr{D}}=(\mathscr{D},g^{\overline{\mathscr{D}}})$ on $\mathscr{X}$ (see \cite[section~5]{MoriwakiZar}), we can associate an adelic $\RR$-Cartier divisor
\begin{equation}
 \overline{\mathscr{D}}^{\rm ad}:=\left(\mathscr{D}_K,\sum_{v\in M_K}g_v^{(\mathscr{X},\mathscr{D})}[v]+g^{\overline{\mathscr{D}}}[\infty]\right)
\end{equation}
on $X$.
We say that $\overline{\mathscr{D}}$ is \emph{w-ample} (respectively, \emph{ample}, etc.) if so is $\overline{\mathscr{D}}^{\rm ad}$.

\paragraph{6.}
By \cite[section 4.5]{MoriwakiAdelic} and the same arguments as in \cite[Lemma~2.5]{IkomaCon}, we can uniquely extend the arithmetic intersection numbers of $C^{\infty}$-Hermitian line bundles to a multilinear map
\begin{align}
 \adeg:\aInt_{\RR}(X)^{\times\dim X}\times\aDiv_{\RR}(X) &\to\RR, \\
 (\overline{D}_1,\dots,\overline{D}_{\dim X+1}) &\mapsto\adeg\left(\overline{D}_1\cdots\overline{D}_{\dim X+1}\right), \nonumber
\end{align}
in such a way that
\begin{enumerate}
\item[(i)] the restriction $\adeg:\aInt_{\RR}(X)^{\times (\dim X+1)}\to\RR$ is symmetric,
\item[(ii)] $\adeg\left(\overline{N}^{\cdot(\dim X+1)}\right)=\avol\left(\overline{N}\right)$ for every $\overline{N}\in\aNef_{\RR}(X)$, and
\item[(iii)] $\adeg\left(\overline{D}_1\cdots\overline{D}_{\dim X+1}\right)\geq 0$ for every $\overline{D}_1,\dots,\overline{D}_{\dim X}\in\aNef_{\RR}(X)$ and $\overline{D}_{\dim X+1}\succeq 0$.
\end{enumerate}

\section{Adelic Cartier divisors with base conditions}\label{sec:adelic_with_base}

\subsection{Preliminaries on the valuations}\label{subsec:Prelim}

In this subsection, we recall several basic facts on the valuations.

\begin{definition}
Let $(\Lambda,\leq)$ be a totally ordered $\ZZ$-module of rank $r$.
By \cite[Chap.\ VI, \S 10, no.\ 2, Proposition~4]{BourbakiCA85}, $(\Lambda,\leq)$ is isomorphic to $(\ZZ^r,\leq_{\lex})$ if and only if the height of $\Lambda$ equals $r$, where $\leq_{\lex}$ denotes the lexicographical order.

Let $F\supset k$ be a field extension and let
\begin{equation}
 \nu:F^{\times}\to\Lambda
\end{equation}
be a \emph{valuation of $F/k$ with values in $\Lambda$}; namely, $\nu$ satisfies
\begin{enumerate}
\item[(i)] $\nu(a)=0$ for $a\in k^{\times}$,
\item[(ii)] $\nu(\phi\psi)=\nu(\phi)+\nu(\psi)$ for $\phi,\psi\in F^{\times}$, and
\item[(iii)] $\nu(\phi+\psi)\geq\min\{\nu(\phi),\nu(\psi)\}$ for $\phi,\psi\in F^{\times}$ with $\phi+\psi\neq 0$.
\end{enumerate}
The \emph{valuation ring} of $\nu$ is
\begin{equation}
 O_{\nu}:=\{\phi\in F^{\times}\,:\,\nu(\phi)\geq 0\}\cup\{0\}
\end{equation}
and the \emph{maximal ideal} of $\nu$ is
\begin{equation}
 \mathfrak{m}_{\nu}:=\{\phi\in F^{\times}\,:\,\nu(\phi)>0\}\cup\{0\}.
\end{equation}
We put $O_{\nu}^{\ast}:=O_{\nu}\setminus\mathfrak{m}_{\nu}=\{\phi\in F^{\times}\,:\,\nu(\phi)=0\}$, and put $k_{\nu}:=O_{\nu}/\mathfrak{m}_{\nu}$.
The \emph{value group} of $\nu$ is defined as
\begin{equation}
 \Lambda_{\nu}:=\nu(F^{\times})=F^{\times}/O_{\nu}^{\ast}
\end{equation}
endowed with the order $\leq$, and the \emph{rational rank} of $\nu$ is defined as $\ratrk(\nu):=\rk_{\QQ}\Lambda_{\nu}\otimes_{\ZZ}\QQ$.

Two valuations $\nu_1:F^{\times}\to\Lambda_{\nu_1}$ and $\nu_2:F^{\times}\to\Lambda_{\nu_2}$ are said to be \emph{equivalent} if the following equivalent conditions are satisfied.
\begin{itemize}
\item There exists an order-preserving isomorphism $\iota:\Lambda_{\nu_1}\to\Lambda_{\nu_2}$ such that $\nu_2=\iota\circ\nu_1$.
\item $O_{\nu_1}=O_{\nu_2}$ in $F$.
\end{itemize}
A non-trivial valuation $\nu:F^{\times}\to\Lambda_{\nu}$ is said to be \emph{discrete} if the value group $(\Lambda_{\nu},\leq)$ is isomorphic to $(\ZZ,\leq)$.
We denote by $\DV(F)=\DV(F/k)$ the set of all the equivalence classes of the discrete valuations of $F/k$.
Given any $\nu\in\DV(F)$, there exists a unique valuation $\nu'$ of $F$ such that $\nu'$ is equivalent to $\nu$ and such that $\nu'$ has value group $(\ZZ,\leq)$.
Hence, in the following, we always assume that the value group of a $\nu\in\DV(F)$ is \emph{normalized} to $(\ZZ,\leq)$.
\end{definition}

\begin{definition}
Let $X$ be a projective variety over a field $k$, let $\Rat(X)$ be the field of rational functions on $X$, and let $\nu:\Rat(X)^{\times}\to\Lambda_{\nu}$ be a valuation of $\Rat(X)/k$.
The \emph{center of $\nu$ on $X$} is a point $c_X(\nu)\in X$ such that
\[
\mathcal{O}_{X,c_X(\nu)}\subset O_{\nu}\quad\text{and}\quad\mathfrak{m}_{c_X(\nu)}=\mathfrak{m}_{\nu}\cap\mathcal{O}_{X,c_X(\nu)}.
\]
By the valuative criterion of properness, there exists a unique center on $X$ for each $\nu\in\DV(\Rat(X))$.

Let $D$ be an effective Cartier divisor on $X$.
If $f,g$ are two local equations defining $D$ around $c_X(\nu)$, then $f/g\in\mathcal{O}_{X,c_X(\nu)}^{\ast}$; thus $\nu(f)=\nu(g)$.
Therefore, we can define
\begin{equation}
 \nu_X(D):=\nu(f)\in\Lambda_{\nu}.
\end{equation}
Note that $\nu_X(0)$ is defined as $\nu(1)=0$.
Since
\[
 \nu_X(D+D')=\nu_X(D)+\nu(D')
\]
for two effective Cartier divisors $D,D'$ on $X$, we can uniquely extend the map $\nu_X:D\mapsto\nu_X(D)$ to an $\RR$-linear map
\begin{equation}\label{eqn:defn_valuation_R-div}
 \nu_X:\Div_{\RR}(X)\to\Lambda_{\nu}\otimes_{\ZZ}\RR
\end{equation}
by linearity.
\end{definition}

\begin{remark}\label{rem:center_and_valuation}
Let $\pi:X'\to X$ be a birational projective morphism.
Then
\[
 \nu_{X'}(\pi^*D)=\nu_X(D).
\]
In fact, we have $\pi(c_{X'}(\nu))=c_X(\nu)$.
So, if $f$ is a local equation defining $D$ around $c_X(\nu)$, then $\pi^*f$ is a local equation defining $\pi^*D$ around $c_{X'}(\nu)$.
Hence $\nu_{X'}(\pi^*D)=\nu(\pi^*f)=\nu(f)=\nu_{X}(D)$.
\end{remark}

\begin{lemma}\label{lem:Q-lin_indep}
Let $X$ be a projective variety over a field and let $D$ be an $\RR$-Cartier divisor on $X$.
\begin{enumerate}
\item There exist Cartier divisors $D_1,\dots,D_l$ and $a_1,\dots,a_l\in\RR$ such that $a_1,\dots,a_l$ are $\QQ$-linearly independent and $D=\sum_{i=1}^la_iD_i$.
Moreover, in this case, one has
\[
 \CSupp(D)=\bigcup_{i=1}^l\CSupp(D_i)\quad\text{and}\quad\WSupp(D)=\bigcup_{i=1}^l\WSupp(D_i).
\]
\item Let $\phi\in\Rat(X)^{\times}\otimes_{\ZZ}\RR$.
There exist $\phi_1,\dots,\phi_l\in\Rat(X)^{\times}$ and $a_1,\dots,a_l\in\RR$ such that $a_1,\dots,a_l$ are $\QQ$-linearly independent and $(\phi)=\sum_{i=1}^la_i(\phi_i)$.
Moreover, in this case, one has
\[
 \CSupp((\phi))=\bigcup_{i=1}^l\CSupp((\phi))\quad\text{and}\quad\WSupp((\phi))=\bigcup_{i=1}^l\WSupp((\phi_i)).
\]
\end{enumerate}
\end{lemma}

\begin{proof}
(1): Choose an expression $D=\sum_{i=1}^la_iD_i$ such that $a_i\in\RR$, $D_i\in\Div_{\QQ}(X)$, and $l$ is minimal among such expressions.
If $a_i$ are $\QQ$-linearly dependent, one can find, after renumbering $a_1,\dots,a_l$, an expression
\[
 \sum_{i=1}^lr_ia_i=0
\]
such that $r_1,\dots,r_l\in\ZZ$ and $r_l\neq 0$.
Then
\[
 D=\sum_{i=1}^{l-1}a_iD_i-\frac{1}{r_l}\left(\sum_{i=1}^{l-1}r_ia_i\right)D_l=\sum_{i=1}^{l-1}\frac{a_i}{r_l}(r_lD_i-r_iD_l),
\]
which contradicts the minimality of $l$.

For the first equality, the inclusion $\subset$ is clear.
Suppose that $x\notin\CSupp(D)$; hence $f_1^{\otimes a_1}\otimes\dots\otimes f_l^{\otimes a_l}\in\mathcal{O}_{X,x}^{\ast}\otimes_{\ZZ}\RR$, where $f_i$ is a local equation defining $D_i$ around $x$.
By applying \cite[Lemma~1.3.1]{MoriwakiAdelic}, one has $f_i\in\mathcal{O}_{X,x}^{\ast}\otimes_{\ZZ}\QQ$ for every $i$.
Thus $x\notin\CSupp(D_i)$ for every $i$.

For the second, the inclusion $\subset$ is clear.
Let $V$ be the $\QQ$-subspace of $\WDiv_{\QQ}(X)$ generated by the irreducible components of $\WSupp(D)$.
Since $a_1D_1+\dots+a_lD_l\in V\otimes_{\QQ}\RR$, one has $D_i\in V$ for every $i$ by \cite[Lemma~1.3.1]{MoriwakiAdelic}.
Hence $\WSupp(D_i)\subset\WSupp(D)$ for every $i$.

The same arguments as above also show the assertion (2).
\end{proof}

\begin{lemma}\label{lem:valuation_of_divisors}
Let $D\in\Div_{\RR}(X)$, and suppose that $X$ is normal.
\begin{enumerate}
\item For any $\nu\in\DV(\Rat(X))$, if $D\geq 0$, then $\nu_X(D)\geq 0$.
\item One has
\[
 \CSupp(D)=\WSupp(D)=\bigcup_{\substack{\nu\in\DV(\Rat(X)), \\ \nu_X(D)\neq 0}}\overline{\{c_X(\nu)\}}.
\]
\end{enumerate}
\end{lemma}

\begin{proof}
(1): Let $\pi:X'\to X$ be a resolution of singularities, and write
\[
 \pi^*D=a_1D_1'+\dots+a_lD_l'
\]
with $a_i\geq 0$ and prime divisors $D_i'$ on $X'$.
Let $f_i$ be a local equation defining $D_i'$ around $c_{X'}(\nu)$.
Since $f_i\in\mathcal{O}_{X',c_{X'}(\nu)}\setminus\{0\}$, one has $\nu(f_i)\geq 0$ and
\[
 \nu_X(D)=\nu_{X'}(\pi^*D)=\sum_{i=1}^ka_i\nu(f_i)\geq 0.
\]

(2): Let $\nu\in\DV(\Rat(X))$ such that $\nu_X(D)\neq 0$.
If $c_X(\nu)\notin\CSupp(D)$, then there is a local equation $f\in\mathcal{O}_{X,c_X(\nu)}^{\ast}\otimes_{\ZZ}\RR$ defining $D$ around $c_X(\nu)$.
So $\nu_X(D)=\nu(f)=0$, and it is a contradiction.
Therefore, one obtains the inclusions
\[
 \CSupp(D)\supset\bigcup_{\substack{\nu\in\DV(\Rat(X)), \\ \nu_X(D)\neq 0}}\overline{\{c_X(\nu)\}}\supset\WSupp(D).
\]
We are going to show $\CSupp(D)=\WSupp(D)$.
Let $D=\sum_{i=1}^la_iD_i$ be an expression as in Lemma~\ref{lem:Q-lin_indep}(1).
Since
\[
 \CSupp(D)=\bigcup_{i=1}^l\CSupp(D_i)\quad\text{and}\quad\WSupp(D)=\bigcup_{i=1}^l\WSupp(D_i),
\]
we can assume $D\in\Div(X)$.
In this case, $\CSupp(D)$ is nothing but the usual (Cartier) support $\left\{x\in X\,:\,f\notin\mathcal{O}_{X,c_X(\nu)}^{\ast}\right\}$ (see \cite[Proposition~1.1.1]{MoriwakiAdelic}).

We endow $\CSupp(D)$ with the reduced induced scheme structure, and let $x$ be a maximal point of $\CSupp(D)$.
By \cite[Corollaire~(21.1.9)]{EGAIV_4}, we have $\depth(\mathcal{O}_{X,x})=1$; hence $\dim(\mathcal{O}_{X,x})=1$ since $X$ is normal.
Therefore, $\CSupp(D)$ is a Zariski closed subset of pure codimension one in $X$.

Let $\pi:X'\to X$ be a resolution of singularities of $X$.
There exists an open subset $U$ of $X$ such that $\codim(X\setminus U,X)\geq 2$ and $\pi:\pi^{-1}(U)\to U$ is an isomorphism.
If $x\in U$ is a maximal point of $\CSupp(D)$, then $\mathcal{O}_{X',\pi^{-1}(x)}=\mathcal{O}_{X,x}$, and $\pi^{-1}(x)$ belongs to $\CSupp(\pi^*D)=\WSupp(\pi^*D)$.
Clearly, $\pi(\WSupp(\pi^*D))=\WSupp(D)$; hence $x\in\WSupp(D)$.
\end{proof}

\begin{lemma}\label{lem:abhyankar}
Let $F\supset k$ be a field extension.
\begin{enumerate}
\item Let $\Lambda,\Lambda'$ be totally ordered $\ZZ$-modules.
Let $\ell:\Lambda\to\Lambda'$ be an order-preserving homomorphism; namely, $\ell$ is a homomorphism of $\ZZ$-modules such that $\lambda\geq 0$ implies $\ell(\lambda)\geq 0$ for every $\lambda\in\Lambda$.
If $\nu:F^{\times}\to\Lambda$ is a valuation of $F$, then so is $\ell\circ\nu:F^{\times}\to\Lambda'$.
\item Suppose that $F$ is finitely generated over $k$.
Let $X$ be a projective $k$-variety with $\Rat(X)=F$ and let $r:=\trdeg_kF$.
For a $\nu\in\DV(F)$, the following are equivalent.
\begin{enumerate}
\item $\trdeg_kk_{\nu}=r-1$.
\item There exists a valuation $\bm{\nu}:F^{\times}\to\ZZ^r$ of $F$ such that $\bm{\nu}$ has value group $(\ZZ^r,\leq_{\lex})$ and $\nu=\pr_1\circ\bm{\nu}$.
\item There exist a valuation $\bm{\nu}:F^{\times}\to\Lambda_{\nu}$ of $F$ and an order-preserving surjective linear form $\ell:\Lambda_{\nu}\to\ZZ$ such that $\rk_{\QQ}\Lambda_{\nu}\otimes_{\ZZ}\QQ=r$ and $\nu=\ell\circ\bm{\nu}$.
\item There exist a birational projective morphism $X'\to X$ and a prime Weil divisor $Y'$ of $X'$ such that $X'$ is normal and $\nu$ is equivalent to $\ord_{Y'}$.
\end{enumerate}
\end{enumerate}
\end{lemma}

\begin{definition}\label{defn:divisorial}
We call a valuation $\nu:\Rat(X)^{\times}\to\Lambda_{\nu}$ \emph{divisorial} if $\nu$ is discrete and satisfies the equivalent conditions in Lemma~\ref{lem:abhyankar}(2).
\end{definition}

\begin{proof}[Proof of Lemma~\ref{lem:abhyankar}]
The assertion (1) is obvious.

(2): Since $\nu$ is a discrete valuation, $\mathfrak{m}_{\nu}$ is a principal ideal with a generator $\varpi$.
The implication (b) $\Rightarrow$ (c) is clear.

(a) $\Rightarrow$ (b): The valuation $\nu$ satisfies $\ratrk(\nu)=1$ and
\[
 \ratrk(\nu)+\trdeg_kk_{\nu}=\trdeg_kF.
\]
Hence, by \cite[Chap.\ VI, \S 10, no.\ 3, Corollaire~1]{BourbakiCA85}, $k_{\nu}$ is a finitely generated field extension of $k$ with $\trdeg_kk_{\nu}=r-1$.
There exists a valuation $\overline{\nu}:k_{\nu}^{\times}\to\ZZ^{r-1}$ having value group $(\ZZ^{r-1},\leq_{\lex})$.
We define $\bm{\nu}:F^{\times}\to\ZZ^r$ by
\[
 \phi\mapsto (\nu(\phi),\overline{\nu}(\phi\cdot\varpi^{-\nu(\phi)}\mod\mathfrak{m}_{\nu})).
\]
One can verify that $\bm{\nu}$ is a valuation of $F$.
Given any $(n_1,\dots,n_r)\in\ZZ^r$, there exists a $\psi\in O_{\nu}^{\ast}$ such that $\overline{\nu}(\psi\mod\mathfrak{m}_{\nu})=(n_2,\dots,n_r)$.
Hence, $\bm{\nu}(\psi\cdot\varpi^{n_1})=(n_1,\dots,n_r)$ and the value group of $\bm{\nu}$ is $(\ZZ^r,\leq_{\lex})$.

(c) $\Rightarrow$ (a): We know that $\mathfrak{m}_{\nu}$ is a prime ideal of $O_{\bm{\nu}}$ and $\overline{V}:=O_{\bm{\nu}}/\mathfrak{m}_{\nu}$ is a valuation ring for $k_{\nu}$ (see \cite[Chap.\ VI, \S 4, no.\ 1, Propositions~1 et 2]{BourbakiCA85}).
So we have a homomorphism of semigroups
\[
 \overline{V}\setminus\{0\}=O_{\bm{\nu}}\setminus\mathfrak{m}_{\nu}\xrightarrow{\bm{\nu}}\Ker(\ell).
\]
We can uniquely extend this to a homomorphism $\overline{\nu}:k_{\nu}^{\times}\to\Ker(\ell)$ of abelian groups.
We are going to show that $\overline{\nu}$ is a valuation of $k_{\nu}$ with $\ratrk(\overline{\nu})=r-1$.
The conditions (i) and (ii) are obvious.
For $\phi,\psi\in O_{\bm{\nu}}\setminus\mathfrak{m}_{\nu}$ with $\phi+\psi\not\in\mathfrak{m}_{\nu}$, we have
\begin{align*}
 \overline{\nu}(\phi+\psi\mod\mathfrak{m}_{\nu})&=\bm{\nu}(\phi+\psi) \\
 &\geq\min\{\bm{\nu}(\phi),\bm{\nu}(\psi)\}=\min\{\overline{\nu}(\phi\mod\mathfrak{m}_{\nu}),\overline{\nu}(\psi\mod\mathfrak{m}_{\nu})\}.
\end{align*}
So the condition (iii) holds in general.
Take an $e_1\in\Lambda_{\nu}$ with $\ell(e_1)=1$.
Then $\ell:\Lambda_{\nu}\to\ZZ$ splits and $\Lambda_{\nu}=\ZZ e_1\otimes\Ker(\ell)$ as ordered $\ZZ$-modules, where the right-hand side is endowed with the lexicographical order.
Let $\lambda\in\Ker(\ell)$.
Either $\lambda$ or $-\lambda$ is non-negative, so we can assume $\lambda\geq 0$.
There exists a $\phi\in O_{\bm{\nu}}$ such that $\bm{\nu}(\phi)=\lambda\geq 0$; thus $\overline{\nu}$ is surjective.
Since $\ratrk(\overline{\nu})\leq\trdeg_kk_{\nu}\leq r-1$, we have $\trdeg_kk_{\nu}=r-1$.

(a) $\Rightarrow$ (d): By \cite[Proposition~2.3]{Vaquie06}, there exist a birational projective morphism $X'\to X$ and a point $\xi'\in X'$ of codimension one such that $X'$ is normal and $\nu$ has center $\xi'$ on $X'$.
Since $\mathcal{O}_{X',\xi'}$ is a discrete valuation ring dominated by $O_{\nu}$, we have $\mathcal{O}_{X',\xi'}=O_{\nu}$ and $\nu$ is equivalent to $\ord_{\xi'}$.

(d) $\Rightarrow$ (a): If $\nu$ is equivalent to $\ord_{Y'}$, then $k_{\nu}=\Rat(Y')$ has transcendence degree $r-1$ over $k$.
\end{proof}

\subsection{Base conditions}\label{subsec:Base_Cond}

A purpose of this subsection is to introduce the notion of pairs of adelic $\RR$-Cartier divisors and $\RR$-base conditions (see Definition~\ref{defn:Adelic_with_Base_Conditions4}).

\begin{definition}\label{defn:Base_Cond}
Let $K$ be a number field, let $X$ be a normal projective $K$-variety, and let $\DV(\Rat(X))$ be the set of all the (non-trivial) normalized discrete valuations of $\Rat(X)/K$.
Let $\KK$ be either $\RR$, $\QQ$, or $\ZZ$.
A \emph{$\KK$-base condition} is defined as a finite formal sum
\begin{equation}
 \mathcal{V}:=\sum_{\nu\in\DV(\Rat(X))}a_{\nu}[\nu],
\end{equation}
where $a_{\nu}\in\KK$ and $\nu\in\DV(\Rat(X))$.
We denote by $\VDiv_{\KK}(X)$ the $\KK$-module of all the $\KK$-base conditions on $X$.
The \emph{order} of an $\RR$-base condition $\mathcal{V}$ along $\nu$ is defined as
\begin{equation}\label{eqn:defn_nu_base_cond}
 \nu(\mathcal{V}):=a_{\nu}.
\end{equation}

We say that $\mathcal{V}$ is \emph{effective} if $\nu(\mathcal{V})\geq 0$ for every $\nu\in\DV(\Rat(X))$ and denote it by $\mathcal{V}\geq 0$.
Put
\begin{equation}
 \mathcal{V}_+:=\sum_{\nu(\mathcal{V})\geq 0}\nu(\mathcal{V})[\nu]\quad\text{and}\quad \mathcal{V}_-:=\mathcal{V}_+-\mathcal{V}.
\end{equation}
We say that $\mathcal{V}$ is \emph{divisorial} if $\nu(\mathcal{V})\neq 0$ implies that $\nu$ is divisorial (see Definition~\ref{defn:divisorial}).

We define the \emph{support} on $X$ of an $\RR$-base condition $\mathcal{V}$ as
\begin{equation}\label{eqn:support_of_base_cond}
 \Supp_X(\mathcal{V}):=\bigcup_{\nu(\mathcal{V})\neq 0}\overline{\{c_X(\nu)\}},
\end{equation}
which is a Zariski closed subset of $X$.
\end{definition}

We can naturally regard an $\RR$-Cartier divisor as an $\RR$-Weil divisor.
To an $\RR$-Weil divisor
\[
 \Xi=\sum_{\text{$Z$: prime Weil divisor}}a_ZZ,
\]
we can naturally associate an $\RR$-base condition
\begin{equation}\label{eqn:associated_base_cond}
 [\Xi]:=\sum_{\text{$Z$: prime Weil divisor}}a_Z[\ord_Z].
\end{equation}

\begin{remark}
Let $Z$ be a prime Weil divisor on $X$, and let $D\in\Div_{\RR}(X)$.
One then has
\[
 \ord_Z([D])=\ord_Z(D)=\ord_{Z,X}(D)
\]
(see (\ref{eqn:defn_valuation_R-div}), (\ref{eqn:defn_nu_base_cond}), and (\ref{eqn:associated_base_cond})).
In fact, it suffices to show the equality for $D\in\Div(X)$.
Let $f$ be a local equation defining $D$ around the generic point of $Z$.
Then $\ord_{Z,X}(D)=\ord_Z(f)=\ord_Z(D)=\ord_Z([D])$.
\end{remark}

\begin{definition}\label{defn:Adelic_with_Base_Conditions4}
Let $\KK$ and $\KK'$ be either $\RR$, $\QQ$, or $\ZZ$.
We set
\begin{align}
 \aBVDiv_{\KK,\KK'}(X) &:=\aDiv_{\KK}(X)\times\VDiv_{\KK'}(X), \\
 \aBWDiv_{\KK,\KK'}(X) &:=\aDiv_{\KK}(X)\times\WDiv_{\KK'}(X), \\
 \aBDiv_{\KK,\KK'}(X) &:=\aDiv_{\KK}(X)\times\Div_{\KK'}(X),
\end{align}
and
\begin{equation}
 \aBDDiv_{\KK,\KK'}(X):=\left\{(\overline{D};\mathcal{V})\in\aBVDiv_{\KK,\KK'}(X)\,:\,\text{$\mathcal{V}$ is divisorial}\right\}
\end{equation}
(see Notation and terminology~2 and Definitions~\ref{defn:divisorial} and \ref{defn:Base_Cond}).
We always identify a pair $(\overline{D};0)$ with the adelic $\RR$-Cartier divisor $\overline{D}$.
In particular, we have the natural inclusions of the five types of base conditions;
\[
 \aDiv_{\KK}(X)\subset\aBDiv_{\KK,\KK'}(X)\subset\aBWDiv_{\KK,\KK'}(X)\subset\aBDDiv_{\KK,\KK'}(X)\subset\aBVDiv_{\KK,\KK'}(X).
\]

Let $(\overline{D};\mathcal{V})\in\aBVDiv_{\RR,\RR}(X)$.
\begin{description}
\item[(effective)] We say that $(\overline{D};\mathcal{V})$ is \emph{effective} (respectively, \emph{strictly effective}) if $\overline{D}\geq 0$ (respectively, $\overline{D}>0$; see Notation and terminology~4 for definition of the inequality signs) and $\nu_X(D)\geq\nu(\mathcal{V})$ for every $\nu\in\DV(\Rat(X))$.
For two pairs $(\overline{D}_1;\mathcal{V}_1),(\overline{D}_2;\mathcal{V}_2)$ on $X$, we write
\begin{equation}
 (\overline{D}_1;\mathcal{V}_1)\leq (\overline{D}_2;\mathcal{V}_2)\quad \text{(respectively, $(\overline{D}_1;\mathcal{V}_1)<(\overline{D}_2;\mathcal{V}_2)$)}
\end{equation}
if $(\overline{D}_2-\overline{D}_1;\mathcal{V}_2-\mathcal{V}_1)$ is effective (respectively, strictly effective).
\item[($\nu_0$-effective)] Let $\nu_0\in\DV(\Rat(X))$.
We say that $(\overline{D};\mathcal{V})$ is \emph{$\nu_0$-effective} (respectively, \emph{strictly $\nu_0$-effective}) if $(\overline{D};\mathcal{V})\geq 0$ (respectively, $(\overline{D};\mathcal{V})>0$) and $\nu_{0,X}(D)=\nu_0(\mathcal{V})$.
We write
\begin{equation}\label{eqn:nu0-effective}
 (\overline{D}_1;\mathcal{V}_1)\leq_{\nu_0}(\overline{D}_2;\mathcal{V}_2)\quad \text{(respectively, $(\overline{D}_1;\mathcal{V}_1)<_{\nu_0}(\overline{D}_2;\mathcal{V}_2)$)}
\end{equation}
if $(\overline{D}_2-\overline{D}_1;\mathcal{V}_2-\mathcal{V}_1)$ is $\nu_0$-effective (respectively, strictly $\nu_0$-effective).
Obviously, if $(\overline{D};\mathcal{V})\geq_{\nu_0}0$, then $\nu_0(\mathcal{V})=\nu_{0,X}(D)\geq 0$.
\end{description}

Let $\KK$ be either a blank, $\QQ$, or $\RR$.
Given a pair $(\overline{D};\mathcal{V})\in\aBVDiv_{\RR,\RR}(X)$, we set
\begin{equation}
 \aHzq{\KK}(\overline{D};\mathcal{V}) :=\left\{\phi\in\Rat(X)^{\times}\otimes_{\ZZ}\KK\,:\,(\overline{D}+\widehat{(\phi)};\mathcal{V})>0\right\}\cup\{0\}
\end{equation}
and
\begin{equation}
 \aHzsmq{\KK}(\overline{D};\mathcal{V}):=\left\{\phi\in\Rat(X)^{\times}\otimes_{\ZZ}\KK\,:\,(\overline{D}+\widehat{(\phi)};\mathcal{V})\geq 0\right\}\cup\{0\}
\end{equation}
(see Definition~\ref{defn:Adelic_with_Base_Conditions4}).
If $\Xi$ is an effective $\RR$-Weil divisor on $X$, then
\[
 \aHzsq{?}{\KK}(\overline{D};[\Xi])=\aHzsq{?}{\KK}(\overline{D})\cap H^0_{\KK}(D-\Xi)
\]
for $?=\text{s or ss}$ and $\KK=\text{a blank}$, $\QQ$, or $\RR$.

It follows from definition that
\begin{equation}\label{eqn:general_rule1}
 (\overline{E};[E])\geq_{\nu} 0
\end{equation}
for every effective $\overline{E}\in\aDiv_{\RR}(X)$ and $\nu\in\DV(\Rat(X))$.
Moreover,
\begin{equation}\label{eqn:general_rule2}
 (0;-\mathcal{V})\geq 0\quad\text{(respectively, $(0;-\mathcal{V})\geq_{\nu}0$)}
\end{equation}
for every effective $\mathcal{V}\in\VDiv_{\RR}(X)$ (respectively, effective $\mathcal{V}\in\VDiv_{\RR}(X)$ with $\nu(\mathcal{V})=0$).
\end{definition}

\begin{remark}\label{rem:general_rule1}
\begin{enumerate}
\item By Lemma~\ref{lem:valuation_of_divisors}(1), it follows that $\aHzsq{?}{\KK}(\overline{D};\mathcal{V})=\aHzsq{?}{\KK}(\overline{D};\mathcal{V}_+)$ for $?=\text{s or ss}$ and $\KK=\text{a blank}$, $\QQ$, or $\RR$.
\item Let $\KK$ be either a blank, $\QQ$, or $\RR$.
Let $(\overline{D};\mathcal{V}),(\overline{D}';\mathcal{V}')\in\aBVDiv_{\RR,\RR}(X)$.
If $\phi\in\aHzq{\KK}(\overline{D};\mathcal{V})$ and $\phi'\in\aHzsmq{\KK}(\overline{D}';\mathcal{V}')$, then $\phi\cdot\phi'\in\aHzq{\KK}(\overline{D}+\overline{D}';\mathcal{V}+\mathcal{V}')$, where $\mathcal{V}$ and $\mathcal{V}'$ may not be effective.
\end{enumerate}
\end{remark}

\begin{lemma}\label{lem:atarimae}
Let $S$ be a projective scheme over a Noetherian ring and let $\mathcal{A}$ be an ample invertible sheaf on $S$.
Let $T$ be a closed subscheme of $S$ and let $x\in S$ be a point not contained in $T$.
There exist an $m\geq 1$ and an $s\in H^0(\mathcal{A}^{\otimes m})$ such that $s$ vanishes along $T$ and $s(x)\neq 0$.
\end{lemma}

\begin{proof}
Let $\mathcal{I}$ be the ideal sheaf defining $T$.
For an $m\geq 1$, $\mathcal{I}\otimes_{\mathcal{O}_S}\mathcal{A}^{\otimes m}$ is generated by its global sections.
Since $x\notin T$, $(\mathcal{I}\otimes_{\mathcal{O}_S}\mathcal{A}^{\otimes m})_x$ is isomorphic to $\mathcal{O}_{S,x}$ as $\mathcal{O}_{S,x}$-modules.
Thus, one finds an $s'\in H^0(\mathcal{I}\otimes_{\mathcal{O}_S}\mathcal{A}^{\otimes m})$ such that $s'(x)\neq 0$.
The image of $s'$ via $H^0(\mathcal{I}\otimes_{\mathcal{O}_S}\mathcal{A}^{\otimes m})\to H^0(\mathcal{A}^{\otimes m})$ has the required properties.
\end{proof}

The following lemma, which we will use in the proof of Theorem~\ref{thm:abbig_cone}, gives a sufficient condition to generalize the relation (\ref{eqn:general_rule1}).

\begin{lemma}\label{lem:vanish_base_cond}
Let $\mathcal{V}\in\VDiv_{\RR}(X)$ and let $\nu_0\in\DV(\Rat(X))$.
\begin{enumerate}
\item There exists an adelic Cartier divisor $\overline{A}$ such that $(\overline{A};\mathcal{V})>0$.
\item For $\nu\in\DV(\Rat(X))$, the following are equivalent.
\begin{enumerate}
\item There exists an adelic $\QQ$-Cartier divisor $\overline{A}$ such that $(\overline{A};[\nu])>_{\nu_0}0$.
\item Either $\nu=\nu_0$ or $c_X(\nu_0)\notin\overline{\{c_X(\nu)\}}$.
\end{enumerate}
\item If
\[
 \nu_0(\mathcal{V})\geq 0\quad\text{and}\quad c_X(\nu_0)\notin\Supp_X(\mathcal{V}_+-\nu_0(\mathcal{V})[\nu_0]),
\]
then there exists an adelic $\RR$-Cartier divisor $\overline{A}$ such that $(\overline{A};\mathcal{V})>_{\nu_0}0$.
\end{enumerate}
\end{lemma}

\begin{proof}
We omit the proof of the assertion (1).

(2)(a) $\Rightarrow$ (b): Assume $\nu\neq\nu_0$ in $\DV(\Rat(X))$.
Since $(\overline{A};[\nu])>_{\nu_0}0$, $A\geq 0$ and $\nu_{0,X}(A)=0$; thus $c_X(\nu_0)\notin\Supp(A)$.
On the other hand, since $\nu_X(A)=1$, one has $c_X(\nu)\in\Supp(A)$.
Hence, $c_X(\nu_0)\notin\overline{\{c_X(\nu)\}}$.

(b) $\Rightarrow$ (a): Choose a strictly effective adelic Cartier divisor $\overline{A}_{\nu_0}'$ such that $A_{\nu_0}'$ passes through $c_X(\nu_0)$, and set $\overline{A}_{\nu_0}:=(1/\nu_{0,X}(A_{\nu_0}'))\overline{A}_{\nu_0}'$.
Then
\[
 (\overline{A}_{\nu_0};[\nu_0])>_{\nu_0}0.
\]

Let $\nu\in\DV(\Rat(X))$ such that $c_X(\nu_0)\notin\overline{\{c_X(\nu)\}}$.
By Lemma~\ref{lem:atarimae}, there exists an effective Cartier divisor $A_{\nu}'$ on $X$ such that
\[
 c_X(\nu)\in\Supp(A_{\nu}')\quad\text{and}\quad c_X(\nu_0)\notin\Supp(A_{\nu}').
\]
We endow $A_{\nu}:=(1/\nu_X(A_{\nu}'))A_{\nu}'$ with $A_{\nu}$-Green functions such that $\overline{A}_{\nu}>0$.
We then have
\[
 (\overline{A}_{\nu};[\nu])>_{\nu_0}0.
\]

(3): Since $\nu_0(\mathcal{V}_-)=0$, we have
\[
 (\overline{D};\mathcal{V})\geq_{\nu_0}(\overline{D};\mathcal{V}_+)
\]
for every adelic $\RR$-Cartier divisor $\overline{D}$ (see (\ref{eqn:general_rule2})), so that we can assume $\mathcal{V}\geq 0$.
We fix $\overline{A}_{\nu_0}$ and $\overline{A}_{\nu}$ as above for $\nu\in\DV(\Rat(X))$ with $c_X(\nu_0)\notin\overline{\{c_X(\nu)\}}$, and set
\[
 \overline{A}:=\nu_0(\mathcal{V})\overline{A}_{\nu_0}+\sum_{c_X(\nu_0)\notin\overline{\{c_X(\nu)\}}}\nu(\mathcal{V})\overline{A}_{\nu}.
\]
Then $(\overline{A};\mathcal{V})>_{\nu_0}0$.
\end{proof}

\begin{definition}
Let $X$ be a normal projective $K$-variety and let
\[
 \overline{D}=\left(D,\sum_{v\in M_K\cup\{\infty\}}g_v^{\overline{D}}\right)
\]
be an adelic $\RR$-Cartier divisor on $X$.
By Lemma~\ref{lem:Q-lin_indep}(1), one can write
\[
 D=\sum_{i=1}^la_iD_i
\]
with Cartier divisors $D_1,\dots,D_l$ and $a_i\in\RR$ such that $\Supp(D)=\bigcup_{i=1}^l\Supp(D_i)$.
Let $\iota:Y\to X$ be a $K$-morphism of normal projective $K$-varieties.
If $\iota(Y)$ is not contained in $\Supp(D)$, then $\iota(Y)$ is not contained in $\Supp(D_i)$ for every $i$.
One can define the \emph{pull-back} of $\overline{D}$ via $\iota$ by
\[
 \iota^*\overline{D}:=\left(\sum_{i=1}^la_i\iota^*D_i,\sum_{v\in M_K\cup\{\infty\}}g_v^{\overline{D}}\circ\iota_v^{\rm an}[v]\right),
\]
which one can see is an adelic $\RR$-Cartier divisor on $Y$ (see \cite[Proposition~2.1.4]{MoriwakiAdelic}).

Let $\overline{D}'$ be another adelic $\RR$-Cartier divisor on $X$ such that $\overline{D}'\sim_{\RR}\overline{D}$ and $\iota(Y)$ is not contained in $\Supp(D')$.
By using Lemma~\ref{lem:Q-lin_indep}(1), one can find $\phi_1,\dots,\phi_k\in\Rat(X)^{\times}$ and $r_1,\dots,r_k\in\RR$ such that
\[
 \overline{D}'=\overline{D}+r_1\widehat{(\phi_1)}+\dots+r_k\widehat{(\phi_k)}
\]
and $\iota(Y)$ is not contained in $\Supp((\phi_i))$ for every $i$.
So $\iota^*\overline{D}'\sim_{\RR}\iota^*\overline{D}$.
\end{definition}

The functoriality of the pairs can now be described as follows.
Let $\mu:X'\to X$ be a birational morphism of normal projective varieties.
We can consider a \emph{pull-back} of $(\overline{D};\mathcal{V})\in\aBVDiv_{\RR,\RR}(X)$ defined by
\begin{equation}
 \mu_*^{-1}:\aBVDiv_{\RR,\RR}(X)\to\aBVDiv_{\RR,\RR}(X'),\quad (\overline{D};\mathcal{V})\mapsto (\mu^*\overline{D};\mathcal{V}^{\mu}),
\end{equation}
where we set
\begin{equation}\label{eqn:pull-back_base_cond}
 \mathcal{V}^{\mu}:=\sum_{\nu\in\DV(\Rat(X))}\nu(\mathcal{V})[\nu\circ{\mu^*}^{-1}:\Rat(X')^{\times}\to\ZZ].
\end{equation}
Note that $\mu$ is isomorphic over an open subscheme $U$ of $X$ with $\codim(X\setminus U,X)\geq 2$.
If $Z$ is a prime Weil divisor on $X$, then
\begin{equation}\label{eqn:strict_trans}
 \ord_Z=\ord_{Z'}:\Rat(X)^{\times}\to\ZZ
\end{equation}
holds for the strict transform $Z'$ of $Z$ via $\mu$.

Let $\nu_1,\dots,\nu_l$ be divisorial valuations of $\Rat(X)$ and let $a_,\dots,a_l$ be real numbers.
By applying Lemma~\ref{lem:abhyankar}(2) to $\nu_1,\dots,\nu_l$ successively, one can find a birational projective morphism $\mu:X'\to X$ such that $X'$ is smooth and prime divisors $Y_1,\dots,Y_l$ on $X'$ such that $\nu_i$ is equivalent to $\ord_{Y_i}$ for each $i$.
One then has
\begin{equation}
 \mu_*^{-1}\left(\overline{D};\sum_{i=1}^la_i[\nu_i]\right)=\left(\mu^*\overline{D};\sum_{i=1}^la_i[Y_i]\right)\in\aBDiv_{\RR,\RR}(X').
\end{equation}

If $(\overline{D};[E])\in\aBDiv_{\RR,\RR}(X)$, one can consider another \emph{pull-back} of $(\overline{D};[E])$ defined as
\begin{equation}
 \mu^*:\aBDiv_{\RR,\RR}(X)\to\aBDiv_{\RR,\RR}(X'),\quad (\overline{D};[E])\mapsto (\mu^*\overline{D};[\mu^*E]).
\end{equation}

\begin{lemma}\label{lem:birat_small_sections}
Let $\mu:X'\to X$ be a birational morphism of normal projective varieties and let $\nu\in\DV(\Rat(X))$.
\begin{enumerate}
\item Let $(\overline{D};\mathcal{V})\in\aBVDiv_{\RR,\RR}(X)$.
For $\KK=\RR$, $\QQ$, and a blank and $?=\text{ss}$ and s, one has
\[
 \aHzsq{?}{\KK}(\overline{D};\mathcal{V})\overset{\mu^*}{=}\aHzsq{?}{\KK}(\mu^*\overline{D};\mathcal{V}^{\mu}).
\]
In particular, $(\overline{D};\mathcal{V})\in\aBVDiv_{\RR,\RR}(X)$ is effective (respectively, strictly effective, $\nu$-effective, strictly $\nu$-effective) if and only if so is $\mu_*^{-1}(\overline{D};\mathcal{V})$.
\item Let $(\overline{D};[E])\in\aBDiv_{\RR,\RR}(X)$.
For $\KK=\RR$, $\QQ$, and a blank and $?=\text{ss}$ and s, one has
\[
 \aHzsq{?}{\KK}(\overline{D};[E])\overset{\mu^*}{=}\aHzsq{?}{\KK}(\mu^*\overline{D};[\mu^*E]).
\]
In particular, $(\overline{D};[E])\in\aBDiv_{\RR,\RR}(X)$ is effective (respectively, strictly effective, $\nu$-effective, strictly $\nu$-effective) if and only if so is $\mu^*(\overline{D};[E])$.
\end{enumerate}
\end{lemma}

\begin{proof}
If $\phi\in\aHzsq{?}{\KK}(\overline{D};\mathcal{V})$, then obviously $\mu^*\phi\in\aHzsq{?}{\KK}(\mu^*\overline{D};\mathcal{V}^{\mu})$.
Suppose that $\phi'\in\aHzsq{?}{\KK}(\mu^*\overline{D};\mathcal{V}^{\mu})\setminus\{0\}$ and set $\phi:={\mu^*}^{-1}(\phi')$.
Since $X$ is normal, one has $\phi\in\aHzsq{?}{\KK}(\overline{D};\mathcal{V})$.

Note that a $\phi\in\aHzsq{?}{\KK}(\overline{D};\mathcal{V})\setminus\{0\}$ satisfies $\nu_X(D+(\phi))=\nu(\mathcal{V})$ if and only if $\nu_{X'}(\mu^*D+(\mu^*\phi))=\nu(\mathcal{V}^{\mu})$ (see Remark~\ref{rem:center_and_valuation}).

Similar arguments also imply the assertion (2).
\end{proof}

\subsection{Arithmetic ampleness}\label{subsec:ample}

This subsection and the next are devoted to showing several fundamental properties of ``arithmetically ample'' adelic $\RR$-Cartier divisors.
As in Notation and terminology~5, we consider two notions of ``arithmetic ampleness'', which we call the ``weak ampleness'' (see Lemma~\ref{lem:w-ample}) and the ``ampleness (in the sense of Zhang)'' (see Theorem~\ref{thm:ample}).

\begin{lemma}\label{lem:w-ample}
Let $X$ be a normal projective $K$-variety and let $\overline{A}\in\aDiv_{\RR}(X)$.
\begin{enumerate}
\item If $\overline{A}\in\aDiv_{\RR}(X)$ is w-ample, then so is $\overline{A}+\widehat{(\phi)}$ for every $\phi\in\Rat(X)^{\times}\otimes_{\ZZ}\RR$.
\item\label{enum:w-ample_restriction} Let $Y$ be a closed subvariety of $X$.
If $\overline{A}$ is a w-ample adelic $\RR$-Cartier divisor on $X$ such that $Y\not\subset \Supp(A)$, then the restriction $\overline{A}|_Y$ is again w-ample.
\item\label{enum:w-ample_Y-big} Let $\nu\in\DV(\Rat(X))$.
If $\overline{A}\in\aDiv_{\RR}(X)$ is w-ample, then there exists a $\phi\in\aHzq{\RR}(\overline{A})$ such that $\overline{A}+\widehat{(\phi)}>_{\nu}0$.
\item\label{enum:w-ample_is_open1} Let $\overline{D}_1,\dots,\overline{D}_m\in\aDiv_{\RR}(X)$, $v_1,\dots,v_l\in M_K\cup\{\infty\}$, and $\varphi_1\in C^0_{v_1}(X),\dots,\varphi_l\in C^0_{v_l}(X)$.
If $\overline{A}\in\aDiv_{\RR}(X)$ is w-ample, then there exists an $\varepsilon>0$ such that
\[
 \overline{A}+\sum_{i=1}^m\varepsilon_i\overline{D}_i+\sum_{k=1}^l(0,\varphi_k[v_k])
\]
is also w-ample for every $\varepsilon_i,\varphi_k$ with $|\varepsilon_i|\leq\varepsilon$ and $\|\varphi_k\|_{\sup}\leq\varepsilon$.
\item\label{enum:w-ampleQ} For any w-ample adelic $\RR$-Cartier divisor $\overline{A}$ on $X$, there exists a w-ample adelic $\QQ$-Cartier divisor $\overline{A}'$ such that $\overline{A}>\overline{A}'$.
\end{enumerate}
\end{lemma}

\begin{proof}
(1): By definition (see Notation and terminology~5), we can write
\[
 \overline{A}=\sum_{k=1}^la_k\overline{A}_k
\]
with $l\geq 1$, $a_k>0$, and adelic Cartier divisors $\overline{A}_k$ on $X$ such that, for each $k$, $A_k$ is ample and $H^0(mA_k)=\aSpan{K}{\aHz(m\overline{A}_k)}$ for every $m\gg 1$.
We write $\phi=\phi_1^{\otimes e_1}\otimes\dots\otimes\phi_r^{\otimes e_r}$ with $\phi_1,\dots,\phi_r\in\Rat(X)$ and positive numbers $e_1,\dots,e_r$.
Then
\begin{align*}
 \overline{A}+\widehat{(\phi)} &=\sum_{k=1}^la_k\overline{A}_k+\sum_{j=1}^re_j\widehat{(\phi_j)} \\
 &=\sum_{j=1}^re_j\left(b_j\overline{A}_1+\widehat{(\phi_j)}\right)+\left(a_1-\sum_{j=1}^re_jb_j\right)\overline{A}_1+\sum_{k=2}^la_k\overline{A}_k
\end{align*}
is w-ample for every positive rational numbers $b_1,\dots,b_r$ with $\sum_{j=1}^re_jb_j\leq a_1$.

(2): Assume that $\overline{A}$ is an adelic Cartier divisor on $X$ such that $A$ is ample, such that $H^0(mA)=\aSpan{K}{\aHz(m\overline{A})}$ for every $m\gg 1$, and such that $Y\not\subset\Supp(A)$.
For each $m\geq 1$, we have a diagram
\[
\xymatrix{\aSpan{K}{\aHz(m\overline{A}|_Y)} \ar[r] & H^0(mA|_Y) \\ \aSpan{K}{\aHz(m\overline{A})} \ar[r] \ar[u] & H^0(mA). \ar[u]
}
\]
For every $m\gg 1$, $\aSpan{K}{\aHz(m\overline{A})}=H^0(mA)$ and $H^0(mA)\to H^0(mA|_Y)$ is surjective, so that we can obtain $\aSpan{K}{\aHz(m\overline{A}|_Y)}=H^0(mA|_Y)$ for every $m\gg 1$.

In general, a w-ample adelic $\RR$-Cartier divisor $\overline{A}$ is a positive $\RR$-linear combination $\sum_{k=1}^la_k\overline{A}_k$ such that $A_k$ is an ample Cartier divisor on $X$ and such that $H^0(mA_k)=\aSpan{K}{\aHz(m\overline{A}_k)}$ for every $m\gg 1$.
For each $k$, we take a $\phi_k\in\Rat(X)^{\times}$ such that $Y\not\subset\Supp(A_k+(\phi_k))$, and set $\overline{A}':=\overline{A}+\sum_{k=1}^la_k\widehat{(\phi_k)}$.
By the above arguments,
\[
 \overline{A}'|_Y=\sum_{k=1}^la_k(\overline{A}_k+\widehat{(\phi_k)})|_Y
\]
is a w-ample adelic $\RR$-Cartier divisor on $Y$.
By Lemma~\ref{lem:Q-lin_indep}(2), there exist $\psi_1,\dots\psi_r\in\Rat(X)^{\times}$ and $b_1,\dots,b_r\in\RR$ such that $\overline{A}-\overline{A}'=\sum_{i=1}^r b_i\widehat{(\psi_i)}$, and such that $Y\not\subset\Supp((\psi_i))$ for every $i$.
Hence
\[
 \overline{A}|_Y=\overline{A}'|_Y+\sum_{i=1}^r b_i\widehat{(\psi_i|_Y)}
\]
is w-ample by the assertion (1).

(\ref{enum:w-ample_Y-big}): There exists a $\phi\in\aHzq{\RR}(\overline{A})$ such that $c_{X}(\nu)\notin\Supp(A+(\phi))$; thus $\nu_X(A+(\phi))=0$.

(\ref{enum:w-ample_is_open1}): Write $\overline{A}=\sum_{k=1}^la_k\overline{A}_k$ as above.
Without loss of generality, one can assume that $\overline{D}_i\in\aDiv(X)$ for every $i$ and that $l=0$.
By \cite[Proposition~5.3(5)]{IkomaRem}, one finds a positive rational number $\varepsilon'>0$ such that, for every $i$, $A_1\pm\varepsilon'D_i$ is ample and $H^0(m(A_1\pm\varepsilon'D_i))=\aSpan{K}{\aHz(m(\overline{A}_1\pm\varepsilon'\overline{D}_i))}$ for every sufficiently divisible $m$.
Then
\[
 \overline{A}+\sum_{i=1}^m\varepsilon_i\overline{D}_i=\sum_{i=1}^m\frac{|\varepsilon_i|}{\varepsilon'}\left(\overline{A}_1+\sgn(\varepsilon_i)\varepsilon'\overline{D}_i\right)+\left(a_1-\sum_{i=1}^m\frac{|\varepsilon_i|}{\varepsilon'}\right)\overline{A}_1+\sum_{k=2}^la_k\overline{A}_k
\]
is w-ample for every real numbers $\varepsilon_i$ with $\sum_{i=1}^m|\varepsilon_i|\leq\varepsilon' a_1$.

The assertion (\ref{enum:w-ampleQ}) results from definition and the assertion (\ref{enum:w-ample_is_open1}) above.
\end{proof}

\begin{theorem}\label{thm:ample}
Let $\pi:X\to\Spec(K)$ be a normal projective $K$-variety.
\begin{enumerate}
\item\label{enum:rel_nef_and_w-ample} If $\overline{A}\in\aDiv_{\RR}(X)$ is relatively nef and w-ample, then $\overline{A}$ is ample.
\item If $\overline{A}\in\aDiv_{\RR}(X)$ is ample and $\overline{N}\in\aDiv_{\RR}(X)$ is nef, then $\overline{A}+\overline{N}$ is also ample.
\item For an $\overline{A}\in\aDiv_{\RR}(X)$, the following are equivalent.
\begin{enumerate}
\item $\overline{A}$ is ample.
\item $A$ is ample, $\overline{A}$ is relatively nef, and $\inf_{x\in X(\overline{K})}h_{\overline{A}}(x)>0$.
\item $A$ is ample and there exists an $\varepsilon>0$ such that $\overline{A}-\pi^*\overline{N}$ is nef for every $\overline{N}\in\aDiv_{\RR}(\Spec(K))$ with $0<\adeg\left(\overline{N}\right)\leq\varepsilon$.
\item $A$ is ample and $\overline{A}-\pi^*\overline{N}$ is nef for an $\overline{N}\in\aDiv_{\RR}(\Spec(K))$ with $\overline{N}>0$.
\end{enumerate}
\item Let $\overline{A}$ be a relatively nef adelic $\QQ$-Cartier divisor on $X$.
Then $\overline{A}$ is ample if and only if $\overline{A}$ is w-ample.
\item Let $\overline{A}_1,\dots,\overline{A}_l$ be relatively nef and w-ample adelic $\QQ$-Cartier divisors on $X$.
Then $\overline{N}+\alpha_1\overline{A}_1+\dots+\alpha_l\overline{A}_l$ is w-ample for every positive real numbers $\alpha_1,\dots,\alpha_l$ and for every nef adelic $\RR$-Cartier divisor $\overline{N}$ such that $N$ is contained in the rational $\RR$-subspace spanned by $A_1,\dots,A_l$.
\end{enumerate}
\end{theorem}

\begin{proof}
(1): Let $x\in X(\overline{K})$.
There is a $\phi\in\Rat(X)^{\times}\otimes_{\ZZ}\RR$ such that $\overline{A}+\widehat{(\phi)}$ can be written as $\sum_{k=1}^la_k\overline{A}_k$ with positive real numbers $a_1,\dots,a_l$ and adelic Cartier divisors $\overline{A}_1,\dots,\overline{A}_l$ such that $A_k$ are ample, such that, for each $k$, $H^0(mA_k)=\aSpan{K}{\aHz(m\overline{A}_k)}$ for every $m\gg 1$, and such that $x\notin\bigcup_k\Supp(A_k)$.
Thus
\[
 h_{\overline{A}}(x)=h_{\overline{A}+\widehat{(\phi)}}(x)=\sum_{k=1}^la_kh_{\overline{A}_k}(x)>0,
\]
and $\overline{A}$ is nef.

By Lemma~\ref{lem:w-ample}(\ref{enum:w-ample_restriction}), given any closed subvariety $Y$ of $X$, $\overline{A}|_Y$ is nef and w-ample.
Hence $\adeg\left((\overline{A}|_Y)^{\cdot(\dim Y+1)}\right)=\avol(\overline{A}|_Y)>0$.

(2): For every closed subvariety $Y$ of $X$, we have
\begin{align*}
 &\adeg\left((\overline{A}|_Y+\overline{N}|_Y)^{\cdot(\dim Y+1)}\right) \\
 &\qquad\qquad =\sum_{i=0}^{\dim Y+1}\binom{\dim Y+1}{i}\adeg\left((\overline{A}|_Y)^{\cdot(\dim Y+1-i)}\cdot(\overline{N}|_Y)^{\cdot i}\right).
\end{align*}
Since $\adeg\left((\overline{A}|_Y)^{\cdot(\dim Y+1-i)}\cdot(\overline{N}|_Y)^{\cdot i}\right)\geq 0$ for every $i$ and $\adeg\left((\overline{A})^{\cdot(\dim Y+1)}\right)>0$, we conclude that $\adeg\left((\overline{A}|_Y+\overline{N}|_Y)^{\cdot(\dim Y+1)}\right)>0$.

(3): The implication (c) $\Rightarrow$ (d) are obvious.

(a) $\Rightarrow$ (b): Obviously, $\overline{A}$ is nef.
So, for every closed subvariety $Y$ of $X$, we have $\avol(\overline{A}|_Y)=\adeg\left((\overline{A}|_Y)^{\cdot(\dim Y+1)}\right)>0$.
This implies that $A|_Y$ is big for every closed variety $Y$ of $X$.
Hence, by the Nakai--Moishezon criterion, $A$ is ample.

We are going to show $\inf_{x\in X(\overline{K})}h_{\overline{A}}(x)>0$ by induction on dimension (see \cite[Proof of Lemma~1.3]{Yuan_Zhang13}).
We can assume that $\dim X$ is positive.
Since $\avol(\overline{A})=\adeg\left(\overline{A}^{\cdot(\dim X+1)}\right)>0$, $\overline{A}$ is big.
Thus there is a $\phi\in\Rat(X)^{\times}\otimes_{\ZZ}\RR$ such that $\overline{A}+\widehat{(\phi)}>0$.

For every $x\in X(\overline{K})$ with $x\notin\Supp(A+(\phi))$,
\[
 \adeg\left(\overline{A}|_x\right)=\adeg\left((\overline{A}+\widehat{(\phi)})|_x\right)\geq\frac{[\kappa(x):\QQ]}{2}\essinf_{p\in X_{\infty}^{\rm an}}g_{\infty}^{\overline{A}+\widehat{(\phi)}}(p)>0,
\]
where $\kappa(x)$ denotes the residue field of the image of $x$.
Set $Z:=\Supp(A+(\phi))$ endowed with the reduced induced scheme structure.
By the induction hypothesis,
\[
 \inf_{x\in X(\overline{K})}h_{\overline{A}}(x)\geq\min\left\{\frac{[K:\QQ]}{2}\essinf_{p\in X_{\infty}^{\rm an}}g_{\infty}^{\overline{A}+\widehat{(\phi)}}(p),\inf_{x\in Z(\overline{K})}h_{\overline{A}}(x)\right\}>0.
\]

(b) $\Rightarrow$ (c): Set $\varepsilon:=\inf_{x\in X(\overline{K})}h_{\overline{A}}(x)>0$, and let $\overline{N}\in\aDiv_{\RR}(\Spec(K))$ such that $\adeg\left(\overline{N}\right)\leq\varepsilon$.
For every $x\in X(\overline{K})$, we have
\[
 h_{\overline{A}-\pi^*\overline{N}}(x)\geq h_{\overline{A}}(x)-\varepsilon\geq 0.
\]
So $\overline{A}-\pi^*\overline{N}$ is nef.

(d) $\Rightarrow$ (a): Both $\pi^*\overline{N}$ and $\overline{A}$ are nef.
For every closed subvariety $Y$, we have
\begin{align*}
 &\adeg\left((\overline{A}|_Y)^{\cdot(\dim Y+1)}\right) \\
 &\qquad =\sum_{i=0}^{\dim Y+1}\binom{\dim Y+1}{i}\adeg\left((\overline{A}|_Y-\pi^*\overline{N}|_Y)^{\cdot(\dim Y+1-i)}\cdot(\pi^*\overline{N}|_Y)^{\cdot i}\right) \\
 &\qquad \geq(\dim Y+1)\adeg\left((\overline{A}|_Y-\pi^*\overline{N}|_Y)^{\cdot\dim Y}\cdot(\pi^*\overline{N}|_Y)\right) \\
 &\qquad =(\dim Y+1)\deg\left((A|_Y)^{\cdot\dim Y}\right)\cdot\adeg\left(\overline{N}\right)>0.
\end{align*}

(4): The ``if'' part is nothing but the assertion (1), so we are going to show the ``only if'' part.
Let $U$ be an open subset of $\Spec(O_K)$ over which a model of definition for $\overline{A}$ exists.
By the assertion (3) above, one can find an $\varepsilon>0$ such that
\[
 \overline{A}-\pi^*\left(0,\sum_{v\notin U}\varepsilon[v]\right)
\]
is ample.
By definition, there exists an $O_K$-model $(\mathscr{X}_{\varepsilon},\mathscr{A}_{\varepsilon})$ of $(X,A)$ such that $(\mathscr{X}_{\varepsilon},\mathscr{A}_{\varepsilon})|_U$ is a model of definition for $\overline{A}$, $\mathscr{A}_{\varepsilon}$ is a relatively nef $\QQ$-Cartier divisor on $\mathscr{X}_{\varepsilon}$, and
\begin{equation}
 \overline{A}\geq\left(\mathscr{A}_{\varepsilon},g_{\infty}^{\overline{A}}\right)\geq\overline{A}-\pi^*\left(0,\sum_{v\notin U}\varepsilon[v]\right).
\end{equation}
Therefore, $\left(\mathscr{A}_{\varepsilon},g_{\infty}^{\overline{A}}\right)$ is also ample.
By the arithmetic Nakai--Moishezon criterion \cite[Theorems~(3.5) and (4.2)]{ZhangVar}, $\left(\mathscr{A}_{\varepsilon},g_{\infty}^{\overline{A}}\right)$ is w-ample, and so is $\overline{A}$.

(5): There exist $\beta_1,\dots,\beta_l$ such that $0<\beta_i<\alpha_i$ for every $i$ and $N+\beta_1A_1+\dots+\beta_lA_l$ is rational.
By the assertions (1), (2), and (4) above, $\overline{N}+\beta_1\overline{A}_1+\dots+\beta_l\overline{A}_l$ is w-ample, and so is
\[
 \overline{N}+\sum_{i=1}^l\alpha_i\overline{A}_i=\left(\overline{N}+\sum_{j=1}^l\beta_j\overline{A}_j\right)+\sum_{k=1}^l(\alpha_k-\beta_k)\overline{A}_k.
\]
\end{proof}

\begin{remark}
In \cite[Remark~3.20]{BMPS12}, Burgos Gil, Moriwaki, Philippon, and Sombra proposed a question whether an ample adelic $\RR$-Cartier divisor $\overline{D}$ on $X$ is w-ample or not.
This question is known to have positive answer in the following cases.
\begin{enumerate}
\item $\overline{D}$ is an ample adelic $\QQ$-Cartier divisor (see Theorem~\ref{thm:ample}(4)).
\item $X$ has dimension one (see Corollary~\ref{cor:aNM_for_curves} below).
\item $\overline{D}$ is a toric metrized $\RR$-Cartier divisor on a projective toric variety $X$ (see \cite[Corollary~6.3(2)]{BMPS12}).
\end{enumerate}
\end{remark}

\subsection{Arithmetic base loci}\label{subsec:arith_base_loci}

\begin{definition}\label{defn:base_locus}
Let $X$ be a normal projective variety over a field.
Recall that the \emph{augmented stable base locus} of an $\RR$-Weil divisor $D$ is defined as
\begin{equation}
 \Bsp(D):=\bigcap_{\text{$A$: ample}}\SBs(D-A),
\end{equation}
where $\SBs(D-A)$ denotes the real stable base locus of $D-A$ and the intersection is taken over all the ample $\RR$-Cartier divisors $A$ on $X$. (see \cite[section~3.5]{BCHM10} and \cite[section~1]{Ein_Laz_Mus_Nak_Pop06} for detail).

Suppose that $X$ is defined over $K$, and let $\mathcal{V}\in\VDiv_{\RR}(X)$.
The $\RR$-linear map $[\cdot]:\WDiv_{\RR}(X)\to\VDiv_{\RR}(X)$ admits a natural retraction $\Weil_X:\VDiv_{\RR}(X)\to\WDiv_{\RR}(X)$ defined by
\begin{equation}
 \mathcal{V}\mapsto\sum_{\dim(\mathcal{O}_{X,c_X(\nu)})=1}\nu(\mathcal{V})\overline{\{c_X(\nu)\}}.
\end{equation}

Let $?=\text{ss}$ or s.
In view of Lemma~\ref{lem:valuation_of_divisors}(2), we define the \emph{real stable base locus} of a pair $(\overline{D};\mathcal{V})\in\aBVDiv_{\RR,\RR}(X)$ as
\begin{equation}
 \aSBss{?}(\overline{D};\mathcal{V}):=\bigcap_{\phi\in\aHzsq{?}{\RR}(\overline{D};\mathcal{V})}\Supp\left(D+(\phi)-\Weil_X(\mathcal{V}_+)\right),
\end{equation}
and the \emph{augmented stable base locus} of $(\overline{D};[\Xi])$ as
\begin{equation}
 \aBsp(\overline{D};\mathcal{V}):=\bigcap_{\text{$\overline{A}$: w-ample}}\aSBs(\overline{D}-\overline{A};\mathcal{V}),
\end{equation}
where the intersection is taken over all the w-ample adelic $\RR$-Cartier divisors $\overline{A}$ on $X$ (see Notation and terminology~5).

It follows from definition that all of these are Zariski closed subsets of $X$, that
\begin{equation}
 \aSBss{?}(\overline{D};\mathcal{V})=\aSBss{?}(\overline{D};\Weil_X(\mathcal{V}_+))\quad\text{and}\quad\aBsp(\overline{D};\mathcal{V})=\aBsp(\overline{D};\Weil_X(\mathcal{V}_+)),
\end{equation}
and that
\begin{equation}\label{eqn:asbs_basic}
 \aSBs(\overline{D}_1+\overline{D}_2;\mathcal{V}_1+\mathcal{V}_2)\subset\aSBs(\overline{D}_1;\mathcal{V}_1)\cup\aSBss{\rm s}(\overline{D}_2;\mathcal{V}_2)
\end{equation}
holds for every $(\overline{D}_1;\mathcal{V}_1),(\overline{D}_2;\mathcal{V}_2)\in\aBVDiv_{\RR,\RR}(X)$ with $\mathcal{V}_1\geq 0$ and $\mathcal{V}_2\geq 0$.

If $\overline{A}\in\aDiv_{\RR}(X)$ is w-ample, then $\aBsp(\overline{A})=\emptyset$ (see also Proposition~\ref{prop:w-ample}(1)) and, if $\overline{E}\in\aDiv_{\RR}(X)$ is effective, then $\aSBss{\rm s}(\overline{E};[E])=\emptyset$.
\end{definition}

\begin{proposition}\label{prop:real_rat}
If $(\overline{D};[\Xi])\in\aBWDiv_{\QQ,\QQ}(X)$, then
\[
 \aSBs(\overline{D};[\Xi])=\bigcap_{\phi\in\aHzsq{?}{\QQ}(\overline{D};[\Xi])}\Supp(D+(\phi)-\Xi_+).
\]
\end{proposition}

\begin{proof}
The inclusion $\subset$ is clear.
We can assume $(\overline{D};[\Xi])\in\aBWDiv(X)$ and $\Xi\geq 0$.
Suppose that $x\notin\aSBs(\overline{D};[\Xi])$, so that there are $\phi_1,\dots,\phi_r\in\Rat(X)^{\times}$ and $e_1,\dots,e_r\in\RR$ such that
\[
 \overline{D}+\sum_{i=1}^re_i\widehat{(\phi_i)}>0,\quad D+\sum_{i=1}^re_i(\phi_i)\geq\Xi,\quad\text{and}\quad x\notin\Supp\left(D+\sum_{i=1}^re_i(\phi_i)-\Xi\right).
\]
If $(e_1,\dots,e_r)\in\QQ^r$, then we have nothing to show, so, by the same arguments as in Lemma~\ref{lem:Q-lin_indep}(2), we may assume that $e_1,\dots,e_r$ are $\QQ$-linearly independent.

We denote by $\overline{D}_a$ the adelic $\RR$-Cartier divisor $\overline{D}+\sum_{i=1}^ra_i\widehat{(\phi_i)}$ for $a=(a_1,\dots,a_r)\in\RR^r$.
Let $V$ be the rational $\RR$-subspace of $\WDiv_{\RR}(X)$ generated by the components of $D$, $\Xi$, and $(\phi_i)$'s.
Let $W$ be the rational $\RR$-subspace of $V$ generated by $(\phi_i)$'s.
Then
\begin{equation}
 P:=\left\{D'\in V\,:\,\text{$D'-D\in W$, $D'\geq\Xi$, and $x\notin\Supp\left(D'-\Xi\right)$}\right\}
\end{equation}
is a convex rational polytope containing $D_e$ for $e:=(e_1,\dots,e_r)$ in its relative interior.
By the following claim and $\essinf_{x\in X_{\infty}^{\rm an}}g_{\infty}^{\overline{D}_e}(x)>0$, one finds a rational point $f=(f_1,\dots,f_r)$ such that $D_f\in P$ and $\essinf_{x\in X_{\infty}^{\rm an}}g_{\infty}^{\overline{D}_f}(x)>0$.

\begin{claim}
The function
\[
 P\to\RR, \quad D'\mapsto\essinf_{x\in X_{\infty}^{\rm an}}g_{\infty}^{\overline{D}'}(x),
\]
is continuous over the relative interior of $P$.
\end{claim}

\begin{proof}
For $D',D''\in P$ and $0\leq\lambda\leq 1$, we have
\[
 \essinf_{x\in X_{\infty}^{\rm an}}g_{\infty}^{(1-\lambda)\overline{D}'+\lambda\overline{D}''}(x)\leq(1-\lambda)\cdot\essinf_{x\in X_{\infty}^{\rm an}}g_{\infty}^{\overline{D}'}(x)+\lambda\cdot\essinf_{x\in X_{\infty}^{\rm an}}g_{\infty}^{\overline{D}''}(x).
\]
So, by \cite[Theorem~6.3.4]{DudleyBook}, the function is continuous over the relative interior of $P$.
\end{proof}

Let $m\geq 1$ be an integer such that $mD_f\in\Div(X)$.
Since $\aHzf(m\overline{D}_f)$ is a full-rank lattice in $H^0(mD_f)$ (see Notation and terminology~4), there exists an integer $p\geq 1$ such that $p\cdot 1\in\aHzf(m\overline{D}_f)$.
So,
\[
 m\overline{D}_f+\widehat{(p)}>0,\quad mD_f+(p)\geq m\Xi,\quad\text{and}\quad x\notin\Supp\left(mD_f+(p)-m\Xi\right).
\]
\end{proof}

\begin{lemma}\label{lem:absp}
Let $(\overline{D};[\Xi])\in\aBWDiv_{\RR,\RR}(X)$.
For any w-ample adelic $\RR$-Cartier divisors $\overline{A}_1,\dots\overline{A}_l$, there exists an $\alpha>0$ such that
\[
 \aBsp(\overline{D};[\Xi])=\aSBs\left(\overline{D}-\sum_{k=1}^l\alpha_k\overline{A}_k;[\Xi]\right)
\]
for every $\alpha_k$ with $0<\alpha_k\leq\alpha$.
\end{lemma}

\begin{proof}
Since $X$ is a Noetherian topological space, one finds w-ample adelic $\RR$-Cartier divisors $\overline{B}_1,\dots,\overline{B}_m$ such that
\begin{equation}\label{eqn:w-ample_absp}
 \aBsp(\overline{D};[\Xi])=\bigcap_{j=1}^m\aSBs(\overline{D}-\overline{B}_j;[\Xi]).
\end{equation}
By Lemma~\ref{lem:w-ample}(\ref{enum:w-ample_is_open1}), there exists an $\alpha>0$ such that
\[
 \overline{B}_j-\sum_{k=1}^l\alpha_k\overline{A}_k
\]
are w-ample for all $j$ and all $\alpha_k$ with $0<\alpha_k\leq\alpha$.
So, by (\ref{eqn:asbs_basic}) and (\ref{eqn:w-ample_absp}),
\begin{align*}
 \aBsp(\overline{D};[\Xi]) &=\bigcap_{j=1}^m\aSBs\left(\overline{D}-\sum_{k=1}^l\alpha_k\overline{A}_k-\left(\overline{B}_j-\sum_{k=1}^l\alpha_k\overline{A}_k\right);[\Xi]\right) \\
 &\supset\aSBs\left(\overline{D}-\sum_{k=1}^l\alpha_k\overline{A}_k;[\Xi]\right) \\
 &\supset\aBsp(\overline{D};[\Xi])
\end{align*}
for every $\alpha_k$ with $0<\alpha_k\leq\alpha$.
This completes the proof.
\end{proof}

The following is the main purpose of this subsection.

\begin{proposition}\label{prop:w-ample}
Let $\overline{A}\in\aDiv_{\RR}(X)$.
\begin{enumerate}
\item $\overline{A}$ is w-ample if and only if $\aBsp(\overline{A})=\emptyset$.
\item If $\overline{A}\in\aDiv(X)$, then $\overline{A}$ is w-ample if and only if $A$ is ample and $H^0(mA)=\aSpan{K}{\aHz(m\overline{A})}$ for every $m\gg 1$.
\item\label{enum:w-ample_is_open2} Let $\overline{D}_1,\dots,\overline{D}_m\in\aDiv_{\RR}(X)$, $v_1,\dots,v_l\in M_K\cup\{\infty\}$, and $\varphi_1\in C^0_{v_1}(X),\dots,\varphi_l\in C^0_{v_l}(X)$.
Let $E_1,\dots,E_n\in\Div_{\RR}(X)$ be effective $\RR$-Cartier divisors on $X$.
If $\overline{A}$ is w-ample, then there exists an $\varepsilon>0$ such that
\[
 \aBsp\left(\overline{A}+\sum_{i=1}^m\varepsilon_i\overline{D}_i+\sum_{k=1}^l(0,\varphi_k[v_k]);\sum_{j=1}^n\delta_j[E_j]\right)=\emptyset
\]
for every $\varepsilon_i$, $\delta_j$, and $\varphi_k$ with $|\varepsilon_i|\leq\varepsilon$, $0\leq\delta_j\leq\varepsilon$, and $\|\varphi_k\|_{\sup}\leq\varepsilon$, respectively.
\end{enumerate}
\end{proposition}

\begin{proof}
(1): If $\overline{A}$ is w-ample, then $\aBsp(\overline{A})=\aSBs(\overline{A}-\overline{A})=\emptyset$.
So we show the converse.
By \cite[Lemma~5.2(1)]{IkomaRem}, $\overline{A}$ is an $\RR$-linear combination of adelic Cartier divisors $\overline{B}$ on $X$ such that $B$ is ample and $H^0(mB)=\aSpan{K}{\aHz(m\overline{B})}$ for every $m\gg 1$.
By Lemma~\ref{lem:absp}, one finds a w-ample adelic $\RR$-Cartier divisor $\overline{A}'$ such that $\overline{A}-\overline{A}'$ is rational and such that $\aSBs\left(\overline{A}-\overline{A}'\right)=\emptyset$.

Write $\overline{A}'=\sum_{k=1}^la_k\overline{A}_k'$ with positive real numbers $a_1,\dots,a_l$ and adelic Cartier divisors $\overline{A}_1',\dots,\overline{A}_l'$ such that, for each $k$, $A_k'$ is ample and $H^0(mA_k')=\aSpan{K}{\aHz(m\overline{A}_k')}$ for every $m\gg 1$.
By Proposition~\ref{prop:real_rat} and \cite[Lemma~5.2(2)]{IkomaRem},
\[
 \overline{A}=(\overline{A}-\overline{A}')+\sum_{k=1}^la_k\overline{A}_k'=\left(b_1\overline{A}_1'+(\overline{A}-\overline{A}')\right)+(a_1-b_1)\overline{A}_1'+\sum_{k=2}^la_k\overline{A}_k'
\]
is w-ample, where $b_1$ is a rational number with $0<b_1<a_1$.

(2): The ``if'' part is clear by definition.
Assume that $\overline{A}\in\aDiv(X)$ is w-ample (see Notation and terminology~5); namely, $A$ is ample and $\aBsp(\overline{A})=\emptyset$.
Hence, by Proposition~\ref{prop:real_rat}, and \cite[Proposition~5.3(3)]{IkomaRem}, we have $H^0(mA)=\aSpan{K}{\aHz(m\overline{A})}$ for every $m\gg 1$.

(3): We endow each $E_j$ with $E_j$-Green functions such that $\overline{E}_j\geq 0$.
By the relation (\ref{eqn:asbs_basic}), we have
\begin{align*}
 &\aBsp\left(\overline{A}+\sum_{i=1}^m\varepsilon_i\overline{D}_i+\sum_{k=1}^l(0,\varphi_k[v_k]);\sum_{j=1}^n\delta_j[E_j]\right)\\
 \subset &\aBsp\left(\overline{A}+\sum_{i=1}^m\varepsilon_i\overline{D}_i-\sum_{k=1}^l\|\varphi_k\|_{\sup}(0,[v_k])-\sum_{j=1}^n\delta_j\overline{E}_j\right)\cup\aSBss{\rm s}\left(\sum_{j=1}^n\delta_j\overline{E}_j;\sum_{j=1}^n\delta_j[E_j]\right) \\
 = &\aBsp\left(\overline{A}+\sum_{i=1}^m\varepsilon_i\overline{D}_i-\sum_{k=1}^l\|\varphi_k\|_{\sup}(0,[v_k])-\sum_{j=1}^n\delta_j\overline{E}_j\right).
\end{align*}
So the assertion results from Lemma~\ref{lem:w-ample}(\ref{enum:w-ample_is_open1}).
\end{proof}

\begin{lemma}\label{lem:aaugbs}
Let $\mu:X'\to X$ be a birational morphism of normal projective $K$-varieties, and let $\Exc(\mu)$ be the exceptional locus.
\begin{enumerate}
\item For every $(\overline{D};[\Xi])\in\aBWDiv_{\RR,\RR}(X)$, one has
\[
 \aBsp(\overline{D};[\Xi])=\aBsp(\overline{D};[\Xi]_+)=\Bsp(D-\Xi_+)\cup\aSBs(\overline{D};[\Xi]).
\]
\item For every $D\in\Div_{\RR}(X)$, one has
\[
 \Bsp(\mu^*D)=\mu^{-1}\Bsp(D)\cup\Exc(\mu).
\]
\item For every $(\overline{D};[E])\in\aBDiv_{\RR,\RR}(X)$, one has
\[
 \aBsp(\mu^*\overline{D};[\mu^*E])=\mu^{-1}\aBsp(\overline{D};[E])\cup\Exc(\mu).
\]
\end{enumerate}
\end{lemma}

\begin{proof}
(1): We assume $\Xi\geq 0$ and show the second equality.
The inclusion $\supset$ is clear by definition.
If $x\notin\Bsp(D-\Xi)\cup\aSBs(\overline{D};[\Xi])$, then there are an w-ample adelic $\RR$-Cartier divisor $\overline{A}$ on $X$, a $\phi\in H^0_{\RR}(D-A-\Xi)\setminus\{0\}$, and a $\psi\in\aHzq{\RR}(\overline{D};[\Xi])\setminus\{0\}$ such that
\[
 x\notin\Supp(D-A+(\phi)-\Xi)\cup\Supp(D+(\psi)-\Xi).
\]
Since $D-A+(\phi)\geq 0$, the Green functions $g_v^{\overline{D}-\overline{A}}-\log|\phi|_v^2$ are bounded from below for all $v\in M_K\cup\{\infty\}$, and are non-negative for all but finitely many $v\in M_K\cup\{\infty\}$.
So one finds a sufficiently small rational number $\lambda$ and a $p\in K^{\times}\otimes_{\ZZ}\RR$ such that
\[
 \overline{D}-\lambda\overline{A}+\left(\lambda\widehat{(\phi)}+(1-\lambda)\widehat{(\psi)}+\widehat{(p)}\right)>0.
\]
Therefore, $x\notin\aSBs(\overline{D}-\lambda\overline{A};[\Xi])$.

The assertion (2) is nothing but \cite[Proposition~2.3]{Bou_Bro_Pac13} (which is valid over arbitrary fields).

(3): By the assertions (1) and (2),
\begin{align*}
 \aBsp(\mu^*\overline{D};[\mu^*E]) &=\Bsp(\mu^*(D-E))\cup\aSBs(\mu^*\overline{D};[\mu^*E]) \\
 &=\mu^{-1}\Bsp(D-E)\cup\Exc(\mu)\cup\mu^{-1}\aSBs(\overline{D};[E]) \\
 &=\mu^{-1}\left(\Bsp(D-E)\cup\aSBs(\overline{D};[E])\right)\cup\Exc(\mu) \\
 &=\mu^{-1}\aBsp(\overline{D};[E])\cup\Exc(\mu).
\end{align*}
\end{proof}

\subsection{Positivity of pairs}\label{subsec:pairs}

In this subsection, we introduce several positivity notions of pairs, and prove the openness of the big cones of pairs (see Theorem~\ref{thm:abbig_cone}).

\begin{definition}\label{defn:positivity}
Let $X$ be a normal projective $K$-variety, let $(\overline{D};\mathcal{V})\in\aBVDiv_{\RR,\RR}(X)$, and let $\KK$ be either $\RR$, $\QQ$, or $\ZZ$.
We define positivity notions for pairs as follows.
\begin{description}
\item[(big)] We say that $(\overline{D};\mathcal{V})$ is \emph{big} if there exists a w-ample adelic $\RR$-Cartier divisor $\overline{A}$ such that $(\overline{D}-\overline{A};\mathcal{V})>0$ (see Notation and terminology~5 and Definition~\ref{defn:Adelic_with_Base_Conditions4}).
We denote by
\[
 \aBVBig_{\KK,\KK'}(X)\supset\aBDBig_{\KK,\KK'}(X)\supset\aBWBig_{\KK,\KK'}(X)\supset\aBBig_{\KK,\KK'}(X)
\]
the cone of all the big pairs in
\[
 \aBVDiv_{\KK,\KK'}(X)\supset\aBDDiv_{\KK,\KK'}(X)\supset\aBWDiv_{\KK,\KK'}(X)\supset\aBDiv_{\KK,\KK'}(X),
\]
respectively (see Definition~\ref{defn:Adelic_with_Base_Conditions4}).
\item[(pseudo-effective)] We say that $(\overline{D};\mathcal{V})$ is \emph{pseudo-effective} if $(\overline{D}+\overline{A};\mathcal{V})$ is big for every big adelic $\RR$-Cartier divisor $\overline{A}$.
We write
\[
 (\overline{D}_1;\mathcal{V}_1)\preceq(\overline{D}_2;\mathcal{V}_2)
\]
if $(\overline{D}_2-\overline{D}_1;\mathcal{V}_2-\mathcal{V}_1)$ is pseudo-effective.
\end{description}
It is clear that the above positivity notions are compatible with addition: for example, if $(\overline{D}_1;\mathcal{V}_1)$ is big and $(\overline{D}_2;\mathcal{V}_2)\geq 0$, then $(\overline{D}_1+\overline{D}_2;\mathcal{V}_1+\mathcal{V}_2)$ is also big.

We define the \emph{arithmetic volume} of $(\overline{D};\mathcal{V})$ as
\begin{equation}\label{eqn:defavolbase}
 \avol(\overline{D};\mathcal{V}):=\limsup_{\substack{m\in\NN, \\ m\to+\infty}}\frac{\log\sharp\aHz(m\overline{D};m\mathcal{V})}{m^{\dim X+1}/(\dim X+1)!}.
\end{equation}
\end{definition}

\begin{definition}\label{defn:Y-positivity}
Let $X$ be a normal projective $K$-variety, let $\nu_0\in\DV(\Rat(X))$, and let $(\overline{D};\mathcal{V})\in\aBVDiv_{\RR,\RR}(X)$.
\begin{description}
\item[($\nu_0$-big)] A pair $(\overline{D};\mathcal{V})$ is called \emph{$\nu_0$-big} if there exists a w-ample adelic $\RR$-Cartier divisor $\overline{A}$ such that $(\overline{D}-\overline{A};\mathcal{V})>_{\nu_0}0$.
Let $\KK$, $\KK'$ be either $\ZZ$, $\QQ$, or $\RR$.
We denote by $\aBVBig_{\KK,\KK'}(X|\nu_0)$, etc., the cone of all the $\nu_0$-big pairs in $\aBVDiv_{\KK,\KK'}(X)$, etc..
\item[($\nu_0$-pseudo-effective)] We say that $(\overline{D};\mathcal{V})$ is \emph{$\nu_0$-pseudo-effective} if $(\overline{D}+\overline{A};\mathcal{V})$ is $\nu_0$-big for every $\nu_0$-big adelic $\RR$-Cartier divisor $\overline{A}$, and write
\[
 (\overline{D}_1;\mathcal{V}_1)\preceq_{\nu_0}(\overline{D}_2;\mathcal{V}_2)
\]
if $(\overline{D}_2-\overline{D}_1;\mathcal{V}_2-\mathcal{V}_1)$ is $\nu_0$-pseudo-effective.
\end{description}
\end{definition}

\begin{remark}\label{rem:ad_hoc}
\begin{enumerate}
\item The cones $\aBVBig_{\RR,\RR}(X|\nu)$, etc., are not open in $\aBVDiv_{\RR,\RR}(X)$, etc. (see also Lemma~\ref{lem:abbig_char}(\ref{enum:Y-bigness}) and Theorem~\ref{thm:abbig_cone}(2)).
For instance, even if $\overline{D}$ is $\nu$-big, $(\overline{D};-r[\nu])$ is not $\nu$-big for every $r>0$.
\item If $(\overline{D};\mathcal{V})\in\aBVBig_{\RR,\RR}(X)$, then $\aHzq{\QQ}(\overline{D};\mathcal{V})\neq\{0\}$.
In fact, by Lemma~\ref{lem:w-ample}(\ref{enum:w-ampleQ}), there exists a w-ample adelic $\QQ$-Cartier divisor $\overline{A}'$ such that $(\overline{D};\mathcal{V})>\overline{A}'$; hence $\aHzq{\QQ}(\overline{D};\mathcal{V})\supset\aHzq{\QQ}(\overline{A}')\neq\{0\}$.
\end{enumerate}
\end{remark}

\begin{lemma}\label{lem:abbig_char}
Let $(\overline{D};\mathcal{V})\in\aBVDiv_{\RR,\RR}(X)$, and let $\nu\in\DV(\Rat(X))$.
\begin{enumerate}
\item\label{enum:Y-bigness} The following are equivalent.
\begin{enumerate}
\item $(\overline{D};\mathcal{V})$ is $\nu$-big.
\item $\nu(\mathcal{V})\geq 0$ and $(\overline{D};\mathcal{V}_+)$ is $\nu$-big.
\end{enumerate}
Moreover, if $\mathcal{V}\in\WDiv_{\RR}(X)$ and $\dim\mathcal{O}_{X,c_X(\nu)}=1$, then the following is also equivalent.
\begin{enumerate}
\item[(c)] $\nu(\mathcal{V})\geq 0$ and $c_X(\nu)\notin\aBsp(\overline{D};\mathcal{V})$.
\end{enumerate}
\item The following are equivalent.
\begin{enumerate}
\item $(\overline{D};\mathcal{V})$ is pseudo-effective.
\item If $(\overline{D}';\mathcal{V}')$ is big, then so is $(\overline{D}+\overline{D}';\mathcal{V}+\mathcal{V}')$.
\end{enumerate}
\item The following are equivalent.
\begin{enumerate}
\item $(\overline{D};\mathcal{V})$ is $\nu$-pseudo-effective.
\item If $(\overline{D}';\mathcal{V}')$ is $\nu$-big, then so is $(\overline{D}+\overline{D}';\mathcal{V}+\mathcal{V}')$.
\end{enumerate}
\end{enumerate}
\end{lemma}

\begin{proof}
(\ref{enum:Y-bigness}): The implication (a) $\Rightarrow$ (b) results from a remark after (\ref{eqn:nu0-effective}), and the converse results from $(0;-\mathcal{V}_-)\geq_{\nu}0$.
The equivalence (b) $\Leftrightarrow$ (c) is obvious by definition of $\aBsp(\overline{D};\mathcal{V})$.

(3): The implication (b) $\Rightarrow$ (a) is clear.
Given a $\nu$-big pair $(\overline{D}';\mathcal{V}')$, one finds a w-ample adelic $\RR$-Cartier divisor $\overline{A}$ on $X$ such that $(\overline{D}'-2\overline{A};\mathcal{V}')>_{\nu}0$.
Since $(\overline{D};\mathcal{V})$ is $\nu$-pseudo-effective, $(\overline{D}+\overline{A};\mathcal{V})$ is $\nu$-big, so there is a $\phi\in\aHzq{\RR}(\overline{D}+\overline{A};\mathcal{V})\setminus\{0\}$ such that $(\overline{D}+\overline{A}+\widehat{(\phi)};\mathcal{V})>_{\nu}0$ (see Lemma~\ref{lem:w-ample}(\ref{enum:w-ample_Y-big})).
Hence $(\overline{D}+\overline{D}';\mathcal{V}+\mathcal{V}')>_{\nu}\overline{A}-\widehat{(\phi)}$.

By the same arguments, one can show the assertion (2).
\end{proof}

\begin{lemma}\label{lem:birat_abhyankar}
Let $\mu:X'\to X$ be a birational morphism of normal projective varieties.
\begin{enumerate}
\item $(\overline{D};\mathcal{V})\in\aBVDiv_{\RR,\RR}(X)$ is big (respectively, pseudo-effective) if and only if so is $\mu_*^{-1}(\overline{D};\mathcal{V})$.
\item If $\nu\in\DV(\Rat(X))$ and $\dim(\mathcal{O}_{X,c_X(\nu)})=1$, then $(\overline{D};\mathcal{V})\in\aBVDiv_{\RR,\RR}(X)$ is $\nu$-big ($\nu$-pseudo-effective) if and only if so is $\mu_*^{-1}(\overline{D};\mathcal{V})$.
\end{enumerate}
\end{lemma}

\begin{proof}
We show the assertion (2).

(2): Let $\overline{A}$ be a w-ample adelic $\RR$-Cartier divisor on $X$ such that $(\overline{D};\mathcal{V})>_{\nu}\overline{A}$.
Then $(\mu^*\overline{D};\mathcal{V}^{\mu})>_{\nu}\mu^*\overline{A}$ and $\mu^*\overline{A}$ is $\nu$-big, since $c_X(\nu)\notin\Exc(\mu)=\aBsp(\mu^*\overline{A})$ (see Lemmas~\ref{lem:aaugbs}(3) and \ref{lem:abbig_char}(1)).
Thus $(\mu^*\overline{D};\mathcal{V}^{\mu})$ is $\nu$-big.

Conversely, if $\overline{A}'$ is a w-ample adelic $\RR$-Cartier divisor on $X'$ such that $(\mu^*\overline{D};\mathcal{V}^{\mu})>\overline{A}'$, then there exists a w-ample adelic $\RR$-Cartier divisor $\overline{A}$ on $X$ such that $\overline{A}'>_{\nu}\mu^*\overline{A}$ (see Lemma~\ref{lem:w-ample}(\ref{enum:w-ample_Y-big})).
So $(\mu^*(\overline{D}-\overline{A});\mathcal{V}^{\mu})>_{\nu}0$ and $(\overline{D};\mathcal{V})$ is $\nu$-big (see Lemma~\ref{lem:birat_small_sections}(1)).

Next, if $(\overline{D};\mathcal{V})$ is $\nu$-pseudo-effective and $\overline{A}$ is $\nu$-big, then $(\overline{D}+\varepsilon\overline{A};\mathcal{V})$ is $\nu$-big for every $\varepsilon>0$.
So $\mu^*\overline{A}$ is $\nu$-big and $(\mu^*\overline{D}+\varepsilon\mu^*\overline{A};\mathcal{V}^{\mu})$ is $\nu$-big for every $\varepsilon>0$.

Conversely, if $\overline{A}$ is a w-ample adelic $\RR$-Cartier divisor on $X$, then $\mu^*\overline{A}$ is $\nu$-big.
So, $\mu_*^{-1}(\overline{D}+\varepsilon\overline{A};\mathcal{V})$ is $\nu$-big for every $\varepsilon>0$ and, so is $(\overline{D}+\varepsilon\overline{A};\mathcal{V})$.
\end{proof}

The main purpose of this subsection is the following.

\begin{theorem}\label{thm:abbig_cone}
Let $(\overline{D};\mathcal{V})\in\aBVDiv_{\RR,\RR}(X)$.
Let $\overline{D}_1,\dots,\overline{D}_m\in\aDiv_{\RR}(X)$, $\mathcal{V}_1,\dots,\mathcal{V}_n\in\VDiv_{\RR}(X)$, $v_1,\dots,v_l\in M_K\cup\{\infty\}$, and $\varphi_1\in C^0_{v_1}(X),\dots,\varphi_l\in C^0_{v_l}(X)$.
\begin{enumerate}
\item If $(\overline{D};\mathcal{V})\in\aBVBig_{\RR,\RR}(X)$, then there exists an $\varepsilon>0$ such that
\[
 \left(\overline{D}+\sum_{i=1}^m\varepsilon_i\overline{D}_i+\sum_{k=1}^l(0,\varphi_k[v_k]);\mathcal{V}+\sum_{j=1}^n\delta_j\mathcal{V}_j\right)\in\aBVBig_{\RR,\RR}(X)
\]
for every $\varepsilon_i$, $\delta_j$, and $\varphi_k$ with $|\varepsilon_i|\leq\varepsilon$, $|\delta_j|\leq\varepsilon$, and $\|\varphi_k\|_{\sup}\leq\varepsilon$, respectively.
In particular, $\aBVBig_{\RR,\RR}(X)$, etc., are open cones in $\aBVDiv_{\RR,\RR}(X)$, etc..
\item Let $\nu_0\in\DV(\Rat(X))$ such that
\[
 c_X(\nu_0)\notin\bigcup_{j=1}^n\Supp_X(\mathcal{V}_j).
\]
\begin{enumerate}
\item If $(\overline{D};\mathcal{V})\in\aBVBig_{\RR,\RR}(X|\nu_0)$, then there exists an $\varepsilon>0$ such that
\[
 \left(\overline{D}+\sum_{i=1}^m\varepsilon_i\overline{D}_i+\sum_{k=1}^l(0,\varphi_k[v_k]);\mathcal{V}+\sum_{j=1}^n\delta_j\mathcal{V}_j\right)\in\aBVBig_{\RR,\RR}(X|\nu_0)
\]
for every $\varepsilon_i$, $\delta_j$, and $\varphi_k$ with $|\varepsilon_i|\leq\varepsilon$, $|\delta_j|\leq\varepsilon$, and $\|\varphi_k\|_{\sup}\leq\varepsilon$, respectively.
\item Suppose that $c_X(\nu_0)\notin\Supp_X\left(\mathcal{V}-\nu_0(\mathcal{V})[\nu_0]\right)$.
If $(\overline{D};\mathcal{V})\in\aBVBig_{\RR,\RR}(X|\nu_0)$, then there exists an $\varepsilon>0$ such that
\[
 \left(\overline{D}+\sum_{i=1}^m\varepsilon_i\overline{D}_i+\sum_{k=1}^l(0,\varphi_k[v_k]);\mathcal{V}+\delta_0[\nu_0]+\sum_{j=1}^n\delta_j\mathcal{V}_j\right)\in\aBVBig_{\RR,\RR}(X|\nu_0)
\]
for every $\varepsilon_i$, $\delta_j$, and $\varphi_k$ such that $|\varepsilon_i|\leq\varepsilon$, $|\delta_j|\leq\varepsilon$ for $j=0,1,\dots,n$, $\delta_0\geq -\nu_0(\mathcal{V})$, and $\|\varphi_k\|_{\sup}\leq\varepsilon$, respectively.
\end{enumerate}
In particular, if $\dim\mathcal{O}_{X,c_X(\nu_0)}=1$, then $\aBWBig_{\RR,\RR}(X|\nu_0)$ is an open cone in
\[
 \left\{(\overline{D};[\Xi])\in\aBWDiv_{\RR,\RR}(X)\,:\,\nu_0([\Xi])\geq 0\right\}.
\]
\end{enumerate}
\end{theorem}

\begin{proof}
We show the assertion (2) only.
Similar arguments also implies the assertion (1).

(2)(a): We can assume that $\mathcal{V}_1,\dots,\mathcal{V}_n$ are all effective.
By Lemma~\ref{lem:vanish_base_cond}(3), there exists an adelic Cartier divisor $\overline{E}_j$ such that
\[
 (\overline{E}_j;\mathcal{V}_j)\geq_{\nu_0} 0.
\]
Therefore,
\begin{align*}
 &\left(\overline{D}+\sum_{i=1}^m\varepsilon_i\overline{D}_i+\sum_{k=1}^l(0,\varphi_k[v_k]);\mathcal{V}+\sum_{j=1}^n\delta_j\mathcal{V}_j\right) \\
 &\qquad\qquad\qquad \geq_{\nu_0}\left(\overline{D}+\sum_{i=1}^m\varepsilon_i\overline{D}_i-\sum_{\delta_j\geq 0}\delta_j\overline{E}_j-\sum_{k=1}^l\|\varphi_k\|_{\sup}(0,[v_k]);\mathcal{V}\right),
\end{align*}
so we can assume $l=n=0$.
We choose a w-ample adelic $\RR$-Cartier divisor $\overline{A}$ such that $(\overline{D};\mathcal{V})>_{\nu_0}\overline{A}$.
By Lemma~\ref{lem:w-ample}(\ref{enum:w-ample_is_open1}), there is an $\varepsilon>0$ such that $\overline{A}+\sum_{i=1}^m\varepsilon_i\overline{D}_i$ is w-ample for every $\varepsilon_i$ with $|\varepsilon_i|\leq\varepsilon$.
Since
\[
 \left(\overline{D}+\sum_{i=1}^m\varepsilon_i\overline{D}_i;\mathcal{V}\right)>_{\nu_0}\overline{A}+\sum_{i=1}^m\varepsilon_i\overline{D}_i
\]
and the right-hand side is w-ample, we have the required assertion.

(b): By Lemma~\ref{lem:vanish_base_cond}(3) and the same arguments as above, we have

\begin{claim}\label{clm:abbig_cone2-1}
There exists a $\gamma>0$ such that
\[
 \left(\overline{D}+\sum_{i=1}^m\varepsilon_i\overline{D}_i+\sum_{k=1}^l(0,\varphi_k[v_k]);\mathcal{V}+\delta_0[\nu_0]+\sum_{j=1}^n\delta_j\mathcal{V}_j\right)\in\aBVBig_{\RR,\RR}(X|\nu_0)
\]
for every $\varepsilon_i$, $\delta_j$, and $\varphi_k$ with $|\varepsilon_i|\leq\gamma$, $0\leq\delta_0\leq\gamma$, $|\delta_j|\leq\gamma$ for $j=1,\dots,n$, and $\|\varphi_k\|_{\sup}\leq\gamma$.
\end{claim}

Next, we show

\begin{claim}\label{clm:abbig_cone2-2}
Suppose that $\nu_0(\mathcal{V})>0$.
There exists a $\gamma$ such that $0<\gamma<\nu_0(\mathcal{V})$ and
\[
 \left(\overline{D};\mathcal{V}-\delta_0[\nu_0]\right)\in\aBVBig_{\RR,\RR}(X|\nu_0)
\]
for every $\delta_0$ with $0\leq\delta_0\leq\gamma$.
\end{claim}

\begin{proof}[Proof of Claim~\ref{clm:abbig_cone2-2}]
Put $\mathcal{V}^{\circ}:=\mathcal{V}-\nu_0(\mathcal{V})[\nu_0]$.
By Claim~\ref{clm:abbig_cone2-1}, there exists a $\gamma_0>0$ such that
\[
 \left(\overline{D}+\delta\overline{D};\mathcal{V}+\delta\mathcal{V}^{\circ}\right)\in\aBVBig_{\RR,\RR}(X|\nu_0)
\]
for every $\delta$ with $0\leq\delta\leq\gamma_0$.
So
\[
 \left(\overline{D};\mathcal{V}-\frac{\delta\nu_0(\mathcal{V})}{1+\delta}[\nu_0]\right)\in\aBVBig_{\RR,\RR}(X|\nu_0)
\]
for every $\delta$ with $0\leq\delta\leq\gamma_0$.
Thus the assertion holds for $\gamma:=\frac{\gamma_0}{1+\gamma_0}\nu_0(\mathcal{V})$.
\end{proof}

If $\nu_0(\mathcal{V})=0$, then the assertion is nothing but Claim~\ref{clm:abbig_cone2-1}, so that we can assume $\nu_0(\mathcal{V})>0$.
By Claims~\ref{clm:abbig_cone2-1} and \ref{clm:abbig_cone2-2}, there exists a $\gamma$ with $0<\gamma<\nu_0(\mathcal{V})$ such that
\[
 \left(\overline{D}+\sum_{i=1}^m\varepsilon_i\overline{D}_i+\sum_{k=1}^l(0,\varphi_k[v_k]);\mathcal{V}+\delta_0[\nu_0]+\sum_{j=1}^n\delta_j\mathcal{V}_j\right)\in\aBVBig_{\RR,\RR}(X|\nu_0)
\]
and
\[
 \left(\overline{D};\mathcal{V}-\delta_0[\nu_0]\right)\in\aBVBig_{\RR,\RR}(X|\nu_0)
\]
for every $\varepsilon_i$, $\delta_j$, and $\varphi_k$ with $|\varepsilon_i|\leq\gamma$, $0\leq\delta_0\leq\gamma$, $|\delta_j|\leq\gamma$ for $j=1,\dots,n$, and $\|\varphi_k\|_{\sup}\leq\gamma$.
So the assertion holds for $\varepsilon:=\gamma/2$.
\end{proof}

\subsection{Theory of Okounkov bodies}\label{subsec:Boucksom_Chen}

Let $X$ be a normal, geometrically irreducible, and projective $K$-variety, and let $D$ be an $\RR$-Cartier divisor on $X$.
A graded linear series
\[
 V_{\sbullet}\subset\bigoplus_{m\geq 0}H^0(mD)
\]
is said to \emph{contain an ample series} if $V_m\neq \{0\}$ for every $m\gg 1$ and there exists an ample $\RR$-Cartier divisor $A$ such that
\begin{itemize}
\item $A\leq D$ and
\item $H^0(mA)\subset V_m$ for every sufficiently divisible $m$
\end{itemize}
(see \cite[Definition~1.1]{Boucksom_Chen} and \cite[page 1388]{MoriwakiBase}).
Note that, here, one can change $A$ with an ample $\QQ$-Cartier divisor $A'\leq A$ having the same properties (see Lemma~\ref{lem:w-ample}(\ref{enum:w-ampleQ})).

Let $\bm{\nu}:\Rat(X)^{\times}\to\Lambda_{\bm{\nu}}$ be a valuation of $\Rat(X)$ with rational rank $\dim X$ (see section \ref{subsec:Base_Cond}).
Denote by $\pr_1:\ZZ\times\Lambda_{\bm{\nu}}\to\ZZ$ and $\pr_2:\ZZ\times\Lambda_{\bm{\nu}}\to\Lambda_{\bm{\nu}}$ the first and the second projection, respectively.
We endow the $\RR$-vector space $\Lambda_{\bm{\nu}}\otimes_{\ZZ}\RR$ with the Lebesgue measure $\vol$ normalized by the lattice $\Lambda_{\bm{\nu}}$.

\begin{definition}
Let $(\overline{D};\mathcal{V})\in\aBVBig_{\RR,\RR}(X)$ such that $\mathcal{V}\geq 0$.
For $\KK=\text{a blank}$, $\QQ$, or $\RR$, we define
\begin{align}
 &H^0_{\KK}(D;\mathcal{V}) \\
 &\quad :=\left\{\phi\in H^0_{\KK}(D)\setminus\{0\}\,:\,\text{$\nu_X(D+(\phi))\geq\nu(\mathcal{V})$, $\forall\nu\in\DV(\Rat(X))$}\right\}\cup\{0\} \nonumber
\end{align}
as a subset of $H^0_{\KK}(D)$.
Set
\begin{align}
 &S_{\bm{\nu}}(D;\mathcal{V}):=\left\{(m,\bm{\nu}(\phi))\,:\,\phi\in H^0(mD;m\mathcal{V}),\,m\geq 1\right\},\\
 &S_{\bm{\nu}}(D;\mathcal{V})_m:=\pr_2\left(S_{\bm{\nu}}(D;\mathcal{V})\cap\pr_1^{-1}(m)\right)
\end{align}
for each $m\geq 1$, and
\begin{equation}
 \Delta_{\bm{\nu}}(D;\mathcal{V}):=\overline{\bigcup_{m\geq 1}\frac{1}{m}S_{\bm{\nu}}(D;\mathcal{V})_m}.
\end{equation}

For each $m\geq 1$, we define an $\RR$-filtration on $H^0(mD;m\mathcal{V})\subset H^0(mD)$ by
\begin{equation}
 F^{mt}(m\overline{D};m\mathcal{V}):=\langle\aHz(m(\overline{D}-(0,2t[\infty]));m\mathcal{V})\rangle_K
\end{equation}
for $t\in\RR$ and set
\begin{equation}
 R^t(\overline{D};\mathcal{V}):=\bigoplus_{m\geq 0}F^{mt}(m\overline{D};m\mathcal{V}).
\end{equation}
Moreover, we set
\begin{align}
 &S_{\bm{\nu}}^t(\overline{D};\mathcal{V}):=\left\{(m,\bm{\nu}(\phi))\,:\,\phi\in F^{mt}(m\overline{D};m\mathcal{V})\setminus\{0\},\,m\geq 1\right\}, \\
 &S_{\bm{\nu}}^t(\overline{D};\mathcal{V})_m:=\pr_2\left(S_{\bm{\nu}}^t(\overline{D};\mathcal{V})\cap\pr_1^{-1}(m)\right)
\end{align}
for each $m\geq 1$, and
\begin{equation}
 \Delta_{\bm{\nu}}^t(\overline{D};\mathcal{V}):=\overline{\bigcup_{m\geq 1}\frac{1}{m}S_{\bm{\nu}}^t(\overline{D};\mathcal{V})_m}.
\end{equation}

We define the \emph{concave transform} $G_{\bm{\nu}}^{(\overline{D};\mathcal{V})}:\Delta_{\bm{\nu}}(D;\mathcal{V})\to\RR\cup\{-\infty\}$ as
\begin{equation}
 G_{\bm{\nu}}^{(\overline{D};\mathcal{V})}(w):=\sup\left\{t\in\RR\,:\,w\in\Delta_{\bm{\nu}}^t(\overline{D};\mathcal{V})\right\}
\end{equation}
for $w\in\Delta(D;\mathcal{V})$, and define the \emph{arithmetic Okounkov body} of $(\overline{D};\mathcal{V})$ as
\begin{equation}
 \aDelta_{\bm{\nu}}(\overline{D};\mathcal{V}):=\left\{(w,t)\in(\Lambda_{\bm{\nu}}\otimes_{\ZZ}\RR)\times\RR_{\geq 0}\,:\,0\leq t\leq G_{\bm{\nu}}^{(\overline{D};\mathcal{V})}(w)\right\}.
\end{equation}
\end{definition}

The theory of Boucksom--Chen \cite[section~2]{Boucksom_Chen} was generalized to the case of normed graded linear series that contain ample series and belong to arithmetic $\RR$-Cartier divisors over generically smooth, normal, and projective arithmetic varieties in \cite[section~1]{MoriwakiBase}.
By using the same arguments, we can easily generalize it to the case of adelically normed graded linear series that contain ample series and belong to adelic $\RR$-Cartier divisors over normal projective algebraic varieties.
(Here, the adelic norms are induced from the supremum norms).

We apply this theory to our graded linear series
\[
 \bigoplus_{m\geq 0}H^0(mD;m\mathcal{V})\subset\bigoplus_{m\geq 0} H^0(mD)
\]
endowed with the subspace adelic norms induced from $\overline{D}$, and obtain the following (see \cite[Theorem~2.8]{Boucksom_Chen}).

\begin{theorem}\label{thm:Boucksom_Chen}
\begin{enumerate}
\item For a $(\overline{D};\mathcal{V})\in\aBVBig_{\RR,\RR}(X)$, one has
\begin{align*}
 \avol(\overline{D};\mathcal{V}) &=\lim_{\substack{m\in\NN, \\ m\to+\infty}}\frac{\log\sharp\aHz(m\overline{D};m\mathcal{V})}{m^{\dim X+1}/(\dim X+1)!} \\
 &=(\dim X+1)![K:\QQ]\cdot\vol\left(\aDelta_{\bm{\nu}}(\overline{D};\mathcal{V})\right).
\end{align*}
\item For $a\in\RR_{>0}$ and $(\overline{D};\mathcal{V})\in\aBVBig_{\RR,\RR}(X)$, one has
\[
 \avol(a\overline{D};a\mathcal{V})=a^{\dim X+1}\avol(\overline{D};\mathcal{V}).
\]
\item If $(\overline{D};\mathcal{V}),(\overline{D}';\mathcal{V}')\in\aBVBig_{\RR,\RR}(X)$, then
\begin{align*}
 &\avol(\overline{D}+\overline{D}';\mathcal{V}+\mathcal{V}')^{1/(\dim X+1)} \\
 &\qquad\qquad \geq\avol(\overline{D};\mathcal{V})^{1/(\dim X+1)}+\avol(\overline{D}';\mathcal{V}')^{1/(\dim X+1)}.
\end{align*}
\end{enumerate}
\end{theorem}

\begin{proof}
We only show the assertion (2).
If $a$ is a positive integer, then the assertion (2) results from the assertion (1).
Hence it also holds for every $a\in\QQ_{>0}$.
In general, we take two sequences of positive rational numbers $(b_i)_{i=1}^{\infty}$ and $(c_i)_{i=1}^{\infty}$ such that $b_i\leq a\leq c_i$ and $c_i-b_i\to 0$ as $i\to\infty$.
Then
\begin{align*}
 b_i^{\dim X+1}\avol(\overline{D};\mathcal{V}) &=\avol(b_i\overline{D};b_i\mathcal{V}) \\
 &\leq\avol(a\overline{D};a\mathcal{V})\leq\avol(c_i\overline{D};c_i\mathcal{V})=c_i^{\dim X+1}\avol(\overline{D};\mathcal{V}).
\end{align*}
By taking $i\to\infty$, we have the assertion.
\end{proof}

Theorems~\ref{thm:Boucksom_Chen} and \ref{thm:abbig_cone} combined with the standard argument (see \cite[Theorem~6.3.4]{DudleyBook})  imply the following.

\begin{corollary}\label{cor:cont}
The function
\[
 \avol:\aBVBig_{\RR,\RR}(X)\to\RR
\]
is continuous, that is, if $(\overline{D};\mathcal{V})\in\aBVBig_{\RR,\RR}(X)$, $\overline{D}_1,\dots,\overline{D}_m\in\aDiv_{\RR}(X)$, $\mathcal{V}_1,\dots,\mathcal{V}_n\in\VDiv_{\RR}(X)$, $v_1,\dots,v_l\in M_K\cup\{\infty\}$, and $\varphi_1\in C^0_{v_1}(X),\dots,\varphi_l\in C^0_{v_l}(X)$, then
\[
 \lim_{\varepsilon_i,\delta_j,\|\varphi_k\|_{\sup}\to 0}\avol\left(\overline{D}+\sum_{i=1}^m\varepsilon_i\overline{D}_i+\sum_{k=1}^l(0,\varphi_k[v_k]);\mathcal{V}+\sum_{j=1}^n\delta_j\mathcal{V}_j\right)=\avol(\overline{D};\mathcal{V}).
\]
(Here the base conditions may not be effective.)
\end{corollary}

\section{Approximation of pairs}\label{sec:Approx_pairs}

\subsection{Construction of adelic metrics}\label{subsec:Adelic_Metrics}

Let $X$ be a normal projective $K$-variety, and let
\[
 \overline{D}_1=\left(D_1,\sum_{v\in M_K\cup\{\infty\}}g_v^{\overline{D}_i}\right),\,\dots,\,\overline{D}_r=\left(D_r,\sum_{v\in M_K\cup\{\infty\}}g_v^{\overline{D}_i}\right)
\]
be adelic Cartier divisors on $X$ such that $D_1,\dots,D_r$ are all effective.
This subsection is devoted to constructing (after blowing up) a coefficient-wise minimum of $\overline{D}_1,\dots,\overline{D}_r$ (see Definition~\ref{defn:minimum_adelic_divisors}).
Several special cases are already treated by many authors \cite{MoriwakiAdelic,Chen11,Yuan_Zhang12,IkomaCon,IkomaRem}.
Put
\begin{equation}\label{eqn:normalized_blowup_along_I}
 \mathcal{I}:=\sum_{i=1}^r\mathcal{O}_X(-D_i).
\end{equation}
Let $\varphi:Y\to X$ be a birational and projective morphism such that $Y$ is normal and $\mathcal{I}\mathcal{O}_Y$ is Cartier.
Let $M$ be an effective Cartier divisor on $Y$ such that $\mathcal{O}_Y(-M)=\mathcal{I}\mathcal{O}_Y$, and let
\begin{equation}
 g_v^{\overline{M}}(x):=\min_{1\leq i\leq r}\left\{g_v^{\overline{D}_i}(\varphi_v^{\rm an}(x))\right\}
\end{equation}
for $v\in M_K\cup\{\infty\}$ and $x\in Y_v^{\rm an}$.

\begin{lemma}\label{lem:endow_with_metric}
For every $v\in M_K\cup\{\infty\}$, $g_v^{\overline{M}}$ is a $M$-Green function on $Y_v^{\rm an}$.
\end{lemma}

\begin{proof}
This follows from the same arguments as in \cite[Proposition~4.7]{IkomaRem}.
\end{proof}

Choose a nonempty open subset $U$ of $\Spec(O_K)$ over which models of definition for $\overline{D}_1,\dots,\overline{D}_r$ exist.
Given any $\varepsilon>0$, there exist $O_K$-models $(\mathscr{X}_{\varepsilon},\mathscr{D}_{i,\varepsilon})$ of $(X,D_i)$ such that $\mathscr{X}_{\varepsilon}$ is normal, $\mathscr{D}_{i,\varepsilon}$ are $\QQ$-Cartier divisors on $\mathscr{X}_{\varepsilon}$, $(\mathscr{X}_{\varepsilon},\mathscr{D}_{i,\varepsilon})|_U$ gives a $U$-model of definition for $\overline{D}_i$, and
\begin{equation}\label{eqn:const_approx_D_i}
 \|g_v^{(\mathscr{X}_{\varepsilon},\mathscr{D}_{i,\varepsilon})}-g_v^{\overline{D}_i}\|_{\sup}\leq\varepsilon
\end{equation}
for every $v\in\Spec(O_K)\setminus U$.
Let $\mathscr{F}\in\Div_{\QQ}(\mathscr{X}_{\varepsilon})$ be a suitable positive combination of vertical fibers on $\mathscr{X}_{\varepsilon}$ such that
\[
 \mathscr{D}_{1,\varepsilon}+\mathscr{F},\dots,\mathscr{D}_{r,\varepsilon}+\mathscr{F}
\]
are all effective.
Let $n\geq 1$ be an integer such that
\[
 n(\mathscr{D}_{1,\varepsilon}+\mathscr{F}),\,\dots,\,n(\mathscr{D}_{r,\varepsilon}+\mathscr{F})
\]
belong to $\Div(\mathscr{X}_{\varepsilon})$, and put
\begin{equation}\label{eqn:defn_Jen}
 \mathcal{J}_{\varepsilon,n}:=\sum_{i=1}^r\mathcal{O}_{\mathscr{X}_{\varepsilon}}(-n(\mathscr{D}_{i,\varepsilon}+\mathscr{F})).
\end{equation}
The restriction of $\mathcal{J}_{\varepsilon,n}$ to the generic fiber is given by $\mathcal{J}_{\varepsilon,n}|_X=\sum_{i=1}^r\mathcal{O}_X(-nD_i)$.

\begin{lemma}\label{lem:integral_closure}
For every $k\geq 1$, we have
\[
 \left(\sum_{i=1}^r\mathcal{O}_X(-kD_i)\right)\mathcal{O}_Y=\left(\sum_{i=1}^r\mathcal{O}_X(-D_i)\right)^k\mathcal{O}_Y.
\]
\end{lemma}

\begin{proof}
We can assume that $X$ is affine and that each of $D_i$ is principal with defining equation $f_i$.

\begin{claim}\label{clm:integral_closure}
Set $I:=(f_1,\dots,f_r)$ and $J_k:=(f_1^k,\dots,f_r^k)$.
As ideals, we have
\[
 I^{kr}=J_k\cdot I^{k(r-1)}.
\]
\end{claim}

\begin{proof}[Proof of Claim~\ref{clm:integral_closure}]
Given any non-negative integers $a_1,\dots,a_r$ with $a_1+\dots+a_r=kr$, there exists at least one $j$ with $a_j\geq k$.
So, $f_1^{a_1}\cdots f_j^{a_j-k}\cdots f_r^{a_r}\in I^{k(r-1)}$.
\end{proof}

By \cite[Corollary~1.2.5]{Swanson_Huneke06}, the above claim implies an equality
\begin{equation}\label{eqn:reduction}
 \overline{J_k}=\overline{(I^k)}
\end{equation}
between the integral closures of ideals.
(We refer to \cite[section~1]{Swanson_Huneke06} for the theory of integral closures.)
So by \cite[Propositions~5.2.4 and 8.1.7]{Swanson_Huneke06} the normalization of the blowup $\Bl_{I^k}(X)$ is the same as the normalization of the blowup $\Bl_{J_k}(X)$.
By persistence of integral closures \cite[Remark~1.1.3(7)]{Swanson_Huneke06} and \cite[Proposition~1.5.2]{Swanson_Huneke06}, we have
\[
 I^k\mathcal{O}_Y\subset\overline{(I^k)}\mathcal{O}_Y\subset\overline{I^k\mathcal{O}_Y}=I^k\mathcal{O}_Y
\]
and, similarly, $J_k\mathcal{O}_Y=\overline{J_k}\mathcal{O}_Y$.
Hence the assertion results from (\ref{eqn:reduction}).
\end{proof}

By Lemma~\ref{lem:integral_closure}, $\varphi$ factorizes through the normalized blowup of $X$ along $\mathcal{J}_{\varepsilon,n}|_X$.
So there exists a birational and projective $O_K$-morphism
\begin{equation}
 \varphi_{\varepsilon}:\mathscr{Y}_{\varepsilon}\to\mathscr{X}_{\varepsilon}
\end{equation}
of projective arithmetic varieties such that $\mathscr{Y}_{\varepsilon}$ is a normal $O_K$-model of $Y$, $\varphi_{\varepsilon}$ extends $\varphi$, and $\mathcal{J}_{\varepsilon,n}\mathcal{O}_{\mathscr{Y}_{\varepsilon}}$ is Cartier.
If we set $\mathscr{M}_{\varepsilon}$ as a $\QQ$-Cartier divisor on $\mathscr{Y}_{\varepsilon}$ such that
\begin{equation}\label{eqn:defn_Me}
 \mathcal{O}_{\mathscr{Y}_{\varepsilon}}(-n(\mathscr{M}_{\varepsilon}+\varphi_{\varepsilon}^*\mathscr{F}))=\mathcal{J}_{\varepsilon,n}\mathcal{O}_{\mathscr{Y}_{\varepsilon}},
\end{equation}
then $(\mathscr{Y}_{\varepsilon},\mathscr{M}_{\varepsilon})$ is an $O_K$-model of $(Y,M)$.

\begin{lemma}\label{lem:construction_adelic_metrics}
\begin{enumerate}
\item For $v\in M_K$ and $x\in Y_v^{\rm an}$,
\[
 g_v^{(\mathscr{Y}_{\varepsilon},\mathscr{M}_{\varepsilon})}(x)=\min_{1\leq i\leq r}\left\{g_v^{(\mathscr{X}_{\varepsilon},\mathscr{D}_{i,\varepsilon})}(\varphi_v^{\rm an}(x))\right\}.
\]
\item For $v\in M_K\cap U$, $g_v^{(\mathscr{Y}_{\varepsilon},\mathscr{M}_{\varepsilon})}=g_v^{\overline{M}}$.
\item For $v\in M_K\setminus U$, $\|g_v^{(\mathscr{Y}_{\varepsilon},\mathscr{M}_{\varepsilon})}-g_v^{\overline{M}}\|_{\sup}\leq\varepsilon$.
\end{enumerate}
\end{lemma}

\begin{proof}
For each $x\in (Y\setminus\Supp(M))_v^{\rm an}$, we have
\begin{align*}
 &g_v^{(\mathscr{Y}_{\varepsilon},n\mathscr{M}_{\varepsilon})}(x) \\
 &\quad =-\log\max\left\{|f|_v^2(x)\,:\,f\in\mathcal{O}_{\mathscr{Y}_{\varepsilon}}(-n\mathscr{M}_{\varepsilon})_{r^{\mathscr{X}_{\varepsilon}}_v(x)}\right\} \\
 &\quad =-\log\max\left\{|h|_v^2(\varphi_v^{\rm an}(x))\,:\,h\in\mathcal{J}_{\varepsilon,n,r^{\mathscr{X}_{\varepsilon}}_v(\varphi_v^{\rm an}(x))}\right\}-g_v^{(\mathscr{X}_{\varepsilon},n\mathscr{F})}(\varphi_v^{\rm an}(x)) \\
 &\quad =\min_{1\leq i\leq r}\left\{g_v^{(\mathscr{X}_{\varepsilon},n\mathscr{D}_{i,\varepsilon})}(\varphi_v^{\rm an}(x))\right\}.
\end{align*}
The assertions (2) and (3) result from the assertion (1) and the relation (\ref{eqn:const_approx_D_i}).
\end{proof}

As a consequence of Lemma~\ref{lem:construction_adelic_metrics}, we can make the following definition.

\begin{definition}\label{defn:minimum_adelic_divisors}
The couple
\[
 \left(M,\sum_{v\in M_K\cup\{\infty\}}g_v^{\overline{M}}[v]\right)
\]
is an adelic Cartier divisor on $Y$.
We denote it by
\[
 \min_{1\leq i\leq r}^{\varphi}\left\{\overline{D}_i\right\}=\left(\min_{1\leq i\leq r}^{\varphi}\left\{D_i\right\},\sum_{v\in M_K\cup\{\infty\}}\min_{1\leq i\leq r}^{\varphi}\left\{g_v^{\overline{D}_i}\right\}[v]\right).
\]
\end{definition}

\begin{remark}\label{rem:minimum_adelic_div}
\begin{enumerate}
\item If $\overline{D}_i\geq 0$ for every $i$, then $\min_i^{\varphi}\left\{\overline{D}_i\right\}\geq 0$.
\item Let $\varphi':Y'\to X$ be another birational and projective morphism such that $Y'$ is normal and $\varphi'$ factorizes as $\varphi'=\psi\circ\varphi$.
Then
\[
 \min_{1\leq i\leq r}^{\varphi'}\left\{\overline{D}_i\right\}=\psi^*\min_{1\leq i\leq r}^{\varphi}\left\{\overline{D}_i\right\}.
\]
\item For each integer $k\geq 1$, let
\[
 \mathcal{J}_k:=\sum_{i=1}^r\mathcal{O}_X(-kD_i)
\]
and let $\varphi_k:Y_k\to X$ be the normalized blowup along $\mathcal{J}_k$.
Then, by Lemma~\ref{lem:integral_closure}, $\varphi_k$ factorizes through $\varphi_1$, and
\[
 \min_{1\leq i\leq r}^{\varphi_k}\left\{k\overline{D}_i\right\}=k\min_{1\leq i\leq r}^{\varphi_k}\left\{\overline{D}_i\right\}.
\]
In particular, we can define the minimum adelic $\QQ$-Cartier divisor $\min_{1\leq i\leq r}^{\varphi}\left\{\overline{D}_i\right\}$ for every $\overline{D}_1,\dots,\overline{D}_r\in\aDiv_{\QQ}(X)$ such that $D_1,\dots,D_r$ are effective.
\end{enumerate}
\end{remark}

\begin{lemma}\label{lem:minimum_of_valuation}
We keep the same notation as above.
For every $\nu\in\DV(\Rat(X))$, one has
\[
 \nu_Y\left(\min_{1\leq i\leq r}^{\varphi}\left\{D_i\right\}\right)=\min_{1\leq i\leq r}\left\{\nu_X(D_i)\right\}.
\]
\end{lemma}

\begin{proof}
Note that, for any effective Cartier divisor $D$ on $X$ with defining ideal sheaf $\mathcal{O}_X(-D)$, one has
\begin{equation}
 \nu_X(D)=\min\left\{\nu(\phi)\,:\,\phi\in\mathcal{O}_{X}(-D)_{c_X(\nu)}\setminus\{0\}\right\}.
\end{equation}
In fact, we take a local equation $f$ defining $D$ around $c_X(\nu)$.
Each $\phi\in\mathcal{O}_{X}(-D)_{c_X(\nu)}\setminus\{0\}$ can be written as $fg$ with $g\in\mathcal{O}_{X,c_X(\nu)}\setminus\{0\}$.
So
\[
 \nu(\phi)=\nu(fg)=\nu(f)+\nu(g)\geq \nu(f).
\]

Let $D':=\min_{1\leq i\leq r}^{\varphi}\left\{D_i\right\}$, let
\[
 \mathcal{I}:=\sum_{i=1}^r\mathcal{O}_X(-D_i)
\]
as in (\ref{eqn:normalized_blowup_along_I}), and let $f_i$ be a local equation defining $D_i$ around $c_X(\nu)$.
Since $\mathcal{O}_Y(-D')=\mathcal{I}\mathcal{O}_Y$, any element in $\mathcal{O}_{Y}(-D')_{c_X(\nu)}\setminus\{0\}$ can be written as
\begin{equation}\label{eqn:minimum_of_valuation1}
 f_1g_1+\dots+f_rg_r
\end{equation}
with $g_i\in\mathcal{O}_{Y,c_Y(\nu)}$.
We remove zeros in (\ref{eqn:minimum_of_valuation1}), and assume that any partial sum of (\ref{eqn:minimum_of_valuation1}) is nonzero.
Then
\[
 \nu(f_1g_1+\dots+f_rg_r)\geq\min_{g_i\neq 0}\left\{\nu(f_i)\right\}.
\]
So, we conclude.
\end{proof}

\begin{lemma}
We keep the same notations as above.
Let $\overline{D}_0$ be another adelic Cartier divisor on $X$ such that $D_0$ is effective and let $\mathcal{I}_0:=\sum_{i=1}^r\mathcal{O}_X(-D_i-D_0)$.
Then $\mathcal{I}_0\mathcal{O}_Y$ is Cartier and
\[
 \min_{1\leq i\leq r}^{\varphi}\left\{\overline{D}_i+\overline{D}_0\right\}=\min_{1\leq i\leq r}^{\varphi}\left\{\overline{D}_i\right\}+\varphi^*\overline{D}_0.
\]
\end{lemma}

\begin{proof}
Since
\[
 \bigoplus_{1\leq i\leq r}\mathcal{O}_X(-D_i-D_0)=\left(\bigoplus_{1\leq i\leq r}\mathcal{O}_X(-D_i)\right)\otimes_{\mathcal{O}_X}\mathcal{O}_X(-D_0),
\]
we have $\mathcal{I}_0=\mathcal{I}\mathcal{O}_X(-D_0)$.
It infers the assertion.
\end{proof}

Let $(\overline{D};\mathcal{V})\in\aBVDiv_{\ZZ,\RR}(X)$ such that $\aHzq{\QQ}(\overline{D};\mathcal{V})\neq\{0\}$.
For each integer $m\geq 1$ and $\phi\in\aHz(m\overline{D};m\mathcal{V})\setminus\{0\}$, we put
\[
 \lambda_{\phi}:=\essinf_{x\in X_{\infty}^{\rm an}}\log|\phi|(x)\exp\left(\frac{1}{2}g_{\infty}^{m\overline{D}}(x)\right),
\]
and consider the minimum adelic Cartier divisor of the finite family
\begin{equation}
 \left\{m\overline{D}+\widehat{(\phi)}-(0,2\lambda_{\phi}[\infty])\,:\,\phi\in\aHz(m\overline{D};m\mathcal{V})\setminus\{0\}\right\}.
\end{equation}
Let $\varphi_m:X_m\to X$ be the normalized blowup along
\[
 \mathcal{I}_m:=\sum_{\phi\in\aHz(m\overline{D};m\mathcal{V})\setminus\{0\}}\mathcal{O}_X(-(mD+(\phi))),
\]
and set
\begin{align}\label{eqn:defn_Mm}
 &\overline{M}(m\overline{D};m\mathcal{V}) \\
 &\qquad\quad :=\varphi_m^*(m\overline{D})-\min_{\phi\in\aHz(m\overline{D};m\mathcal{V})\setminus\{0\}}^{\varphi_m}\left\{m\overline{D}+\widehat{(\phi)}-(0,2\lambda_{\phi}[\infty])\right\}. \nonumber
\end{align}

\begin{lemma}\label{lem:approx_Zariski}
\begin{enumerate}
\item For each $m\geq 1$ with $\aHz(m\overline{D};m\mathcal{V})\neq\{0\}$, the morphism
\[
 \aSpan{K}{\aHz(m\overline{D};m\mathcal{V})}\otimes_K\mathcal{O}_{X_m}\to\mathcal{O}_{X_m}(M(m\overline{D};m\mathcal{V}))
\]
is surjective,
\[
 \aHz(\overline{M}(m\overline{D};m\mathcal{V}))\overset{\varphi_m^*}{=}\aHz(\varphi_m^*(m\overline{D});(m\mathcal{V})^{\varphi_m}),
\]
and $\overline{M}(m\overline{D};m\mathcal{V})\leq(\varphi_m^*(m\overline{D});(m\mathcal{V})^{\varphi_m})$.
\item If $(\overline{D};\mathcal{V})$ is big, then $(\varphi_m:X_m\to X,\overline{M}(m\overline{D};m\mathcal{V})/m)\in\aTheta(\overline{D};\mathcal{V})$ for every sufficiently divisible $m$.
\end{enumerate}
\end{lemma}

\begin{proof}
This is a version for pairs of \cite[Proposition~4.7]{IkomaRem}.

(1): Since the homomorphism $\aSpan{K}{\aHz(m\overline{D};m\mathcal{V})}\otimes_K\mathcal{O}_{X}(-mD)\to\mathcal{I}_m$ is surjective, so is
\[
 \iota:\aSpan{K}{\aHz(\varphi_m^*(m\overline{D});(m\mathcal{V})^{\varphi_m})}\otimes_K\mathcal{O}_{X_m}\to\mathcal{O}_{X_m}(M(m\overline{D};m\mathcal{V})).
\]
For each $s\in\aHz(\overline{M}(m\overline{D};m\mathcal{V}))\subset\aHz(\varphi_m^*(m\overline{D});(m\mathcal{V})^{\varphi_m})$, we have $\iota(s)=s$, so $\aHz(\overline{M}(m\overline{D};m\mathcal{V}))$ is contained in the image of $\aHz(\varphi_m^*(m\overline{D});(m\mathcal{V})^{\varphi_m})$ via $\iota$.
By \cite[Claim~4.9]{IkomaRem}, we have
\[
 \aHz(\overline{M}(m\overline{D};m\mathcal{V}))=\iota\left(\aHz(\varphi_m^*(m\overline{D});(m\mathcal{V})^{\varphi_m})\right).
\]

(2): By Lemma~\ref{lem:minimum_of_valuation} and the definition (\ref{eqn:defn_Mm}), we have
\[
 (\varphi_m^*(m\overline{D});(m\mathcal{V})^{\varphi_m})\geq\overline{M}(m\overline{D};m\mathcal{V}).
\]
The rest of the assertion is obvious.
\end{proof}

\subsection{Arithmetic positive intersection numbers}\label{subsec:aPosIntNum}

\begin{definition}
Let $X$ be a normal projective $K$-variety.
An \emph{approximation} of a pair $(\overline{D};\mathcal{V})\in\aBVBig_{\RR,\RR}(X)$ is a couple $(\mu:X'\to X,\overline{M})$ of a projective and birational morphism $\mu:X'\to X$ of normal projective varieties and an $\overline{M}\in\aNef_{\RR}(X')\cap\aBigCone_{\RR}(X')$ such that $\overline{M}\preceq\mu_*^{-1}(\overline{D};\mathcal{V})$.
We denote by $\aTheta(\overline{D};\mathcal{V})$ the set of all the approximations of $(\overline{D};\mathcal{V})$, and set
\[
 \aTheta_{\rm amp}(\overline{D};\mathcal{V}):=\left\{(\mu,\overline{M})\,:\,\text{$\overline{M}\in\aNef_{\QQ}(X)$ is ample and $\overline{M}\leq\mu_*^{-1}(\overline{D};\mathcal{V})$}\right\}.
\]
\begin{description}
\item[(admissible)] Let $U$ be a nonempty open subset of $\Spec(O_K)$, and let $\delta>0$.
Let $\aTheta_{U,\delta}(\overline{D})$ be the set of all the normal $O_K$-models $(\mathscr{X},\mathscr{D})$ of $(X,D)$ such that
\begin{itemize}
\item $g_v^{\overline{D}}-\delta\leq g_v^{(\mathscr{X},\mathscr{D})}\leq g_v^{\overline{D}}$ for all $v\in M_K\setminus U$ and
\item $g_v^{(\mathscr{X},\mathscr{D})}=g_v^{\overline{D}}$ for all but finitely many $v\in M_K\cap U$.
\end{itemize}
Put
\[
 \aTheta_{\rm mod}(\overline{D}):=\bigcup_{\substack{U\subset\Spec(O_K), \\ \delta>0}}\aTheta_{U,\delta}(\overline{D}),
\]
where $U$ runs over all the nonempty open subsets of $\Spec(O_K)$.
Given an $(\mathscr{X},\mathscr{D})\in\aTheta_{\rm mod}(\overline{D})$, we set $\overline{\mathscr{D}}:=(\mathscr{D},g_{\infty}^{\overline{D}})$.

An \emph{admissible} approximation of $(\overline{\mathscr{D}};\mathcal{V})$ is a couple $(\widetilde{\mu}:\mathscr{X}'\to\mathscr{X},\overline{\mathscr{M}})$ of a projective and birational $O_K$-morphism $\widetilde{\mu}:\mathscr{X}'\to\mathscr{X}$ of normal projective arithmetic varieties and a nef and big arithmetic $\RR$-Cartier divisor $\overline{\mathscr{M}}$ on $\mathscr{X}'$ such that $\mathscr{X}_K'$ is smooth, such that $\widetilde{\mu}^*\overline{\mathscr{D}}-\overline{\mathscr{M}}$ is an effective arithmetic $\QQ$-Cartier divisor on $\mathscr{X}'$, and such that $\nu_{\mathscr{X}_K'}(\widetilde{\mu}^*\mathscr{D}_K-\mathscr{M}_K)\geq \nu(\mathcal{V})$ for every $\nu\in\DV(\Rat(X))$.
We denote by $\aTheta_{\rm ad}(\overline{\mathscr{D}};\mathcal{V})$ the set of all the admissible approximations of $(\overline{\mathscr{D}};\mathcal{V})$.

Given two admissible approximations $(\widetilde{\mu}_1:\mathscr{X}_1'\to\mathscr{X},\overline{\mathscr{M}}_1)$ and $(\widetilde{\mu}_2:\mathscr{X}_2'\to\mathscr{X},\overline{\mathscr{M}}_2)$ of $(\overline{\mathscr{D}};\mathcal{V})$, we write
\[
 (\widetilde{\mu}_1:\mathscr{X}_1'\to\mathscr{X},\overline{\mathscr{M}}_1)\leq (\widetilde{\mu}_2:\mathscr{X}_2'\to\mathscr{X},\overline{\mathscr{M}}_2)
\]
if there exists a birational morphism $\widetilde{\mu}:\mathscr{X}'\to\mathscr{X}$ of normal projective arithmetic varieties such that $\widetilde{\mu}$ can be factorized into $\mathscr{X}'\xrightarrow{\widetilde{\mu}_1'}\mathscr{X}_1'\xrightarrow{\widetilde{\mu}_1}\mathscr{X}$ and $\mathscr{X}'\xrightarrow{\widetilde{\mu}_2'}\mathscr{X}_2'\xrightarrow{\widetilde{\mu}_2}\mathscr{X}$, respectively, and
\[
 \widetilde{\mu}_1^{\prime*}\overline{\mathscr{M}}_1^{\rm ad}\leq\widetilde{\mu}_2^{\prime*}\overline{\mathscr{M}}_2^{\rm ad}
\]
holds.
The set $\aTheta_{\rm ad}(\overline{\mathscr{D}};\mathcal{V})$ is partially ordered with respect to this order.
(See \cite[section~3]{IkomaCon}.)
\end{description}
\end{definition}

The approximations of arithmetic $\RR$-Cartier divisors is already treated in \cite[section~3]{IkomaCon}.
By using the approximation theorem (see for example \cite[Theorem~4.1.3]{MoriwakiAdelic}), we can easily reduce our case to the case of \cite[section~3]{IkomaCon} (see Proposition~\ref{prop:approx_by_models} below).

\begin{proposition}\label{prop:approx_by_models}
Let $(\overline{D};\mathcal{V})\in\aBVBig_{\RR,\RR}(X)$, and let $U$ be a nonempty open subset of $\Spec(O_K)$ over which a model of definition for $\overline{D}$ exists.
\begin{enumerate}
\item Given any $\delta>0$, $\aTheta_{U,\delta}(\overline{D})$ is nonempty.
\item For each $(\mathscr{X},\mathscr{D})\in\aTheta_{\rm mod}(\overline{D})$, the ordered set $\aTheta_{\rm ad}(\overline{\mathscr{D}};\mathcal{V})$ is filtered.
\item Let $(\varphi:X'\to X,\overline{M})\in\aTheta(\overline{D};\mathcal{V})$ and let $U'$ be a nonempty open subset of $\Spec(O_K)$ over which a model of definition for $\overline{M}$ exists.
Given any $\delta$ with $0<\delta<1$, there exists an ample adelic $\QQ$-Cartier divisor $\overline{H}$ on $X'$ such that
\begin{itemize}
\item $\overline{H}-(1-\delta)\overline{M}$ is nef and w-ample,
\item $(\overline{D}-\overline{H};\mathcal{V})$ is big and strictly effective, and
\item $\overline{H}$ has a $U'$-model of definition.
\end{itemize}
\item Let $(\mathscr{X},\mathscr{D})\in\aTheta_{\rm mod}(\overline{D})$, let $(\varphi:X'\to X,\overline{M})\in\aTheta(\overline{\mathscr{D}}^{\rm ad};\mathcal{V})$, and let $U'$ be a nonempty open subset of $\Spec(O_K)$ over which a model of definition for $\overline{M}$ exists.
Given any $\delta$ with $0<\delta<1$, there exists a $(\widetilde{\varphi}:\mathscr{X}'\to\mathscr{X},\overline{\mathscr{H}})\in\aTheta_{\rm ad}(\overline{\mathscr{D}};\mathcal{V})$ such that
\begin{itemize}
\item $\widetilde{\varphi}:\mathscr{X}'\to\mathscr{X}$ is a birational and projective morphism of normal $O_K$-models extending $\varphi$,
\item $\overline{\mathscr{H}}^{\rm ad}$ is ample on $X'$, $\mathscr{H}$ is relatively ample on $\mathscr{X}'$, and
\item $(1-\delta)\overline{M}-\delta\sum_{v\in M_K\setminus U'}(0,[v])\preceq\overline{\mathscr{H}}^{\rm ad}$.
\end{itemize}
\end{enumerate}
\end{proposition}

\begin{proof}
The assertion (1) is nothing but \cite[Theorem~4.1.3]{MoriwakiAdelic}.

(2): Let $(\widetilde{\varphi}_1:\mathscr{X}_1'\to\mathscr{X},\overline{\mathscr{M}}_1)$ and $(\widetilde{\varphi}_2:\mathscr{X}_2'\to\mathscr{X},\overline{\mathscr{M}}_2)$ in $\aTheta_{\rm ad}(\overline{\mathscr{D}};\mathcal{V})$.
By taking a modification dominating both $\widetilde{\varphi}_1$ and $\widetilde{\varphi}_2$, one can assume $\widetilde{\varphi}_1=\widetilde{\varphi}_2$ and $\mathscr{X}_1'=\mathscr{X}_2'$.
Let
\[
 \overline{\mathscr{F}}_i:=\widetilde{\varphi}_1^*\overline{\mathscr{D}}-\overline{\mathscr{M}}_i
\]
for $i=1,2$.
Let $n\geq 1$ be an integer such that both of $n\mathscr{F}_1$ and $n\mathscr{F}_2$ belong to $\Div(\mathscr{X}_1')$, let
\[
 \mathcal{J}_n:=\mathcal{O}_{\mathscr{X}_1'}(-n\mathscr{F}_1)+\mathcal{O}_{\mathscr{X}_1'}(-n\mathscr{F}_2),
\]
and let $\widetilde{\psi}_1:\mathscr{X}'\to\mathscr{X}_1'$ be the normalized blowup along $\mathcal{J}_n$.
Put
\[
 \overline{\mathscr{F}}:=\min^{\widetilde{\psi}_1}\left\{\overline{\mathscr{F}}_1,\overline{\mathscr{F}}_2\right\}
\]
as in \cite[Proposition~3.2]{IkomaCon}, $\widetilde{\varphi}:=\widetilde{\psi}_1\circ\widetilde{\varphi}_1$, and $\overline{\mathscr{M}}:=\widetilde{\varphi}^*\overline{\mathscr{D}}-\overline{\mathscr{F}}$.
Then, by Lemma~\ref{lem:minimum_of_valuation} and \cite[Lemma~3.4]{IkomaCon}, $(\widetilde{\varphi}:\mathscr{X}'\to\mathscr{X},\overline{\mathscr{M}})\in\aTheta_{\rm ad}(\overline{\mathscr{D}};\mathcal{V})$ and $(\widetilde{\varphi},\overline{\mathscr{M}})\geq(\widetilde{\varphi}_i,\overline{\mathscr{M}}_i)$ for $i=1,2$.

(3): By use of Theorem~\ref{thm:abbig_cone}(1), one finds a nef and w-ample adelic $\RR$-Cartier divisor $\overline{A}$ such that $(\overline{D}-\overline{A};\mathcal{V})$ is still big, and such that $\overline{A}$ has a $U'$-model of definition (see \cite[Lemma~2.6]{IkomaRem}).
Fix nef and w-ample adelic Cartier divisors $\overline{A}_1,\dots\overline{A}_l$ such that $A_1,\dots,A_l$ form a basis for a rational $\RR$-subspace of $\Div_{\RR}(X)$ containing both $A$ and $M$, and such that $\overline{A}_1,\dots,\overline{A}_l$ have $U'$-models of definition.
By using Theorem~\ref{thm:abbig_cone}(1) again, one can find an $\alpha>0$ such that
\begin{equation}
 (\overline{D}-\overline{A};\mathcal{V})-\alpha_1\overline{A}_1-\dots-\alpha_l\overline{A}_l
\end{equation}
is big for every $\alpha_k$ with $0\leq\alpha_k\leq\alpha$.

Put $\overline{H}':=\delta\overline{A}+(1-\delta)\overline{M}$, which is ample by Theorem~\ref{thm:ample}(2), and choose real numbers $\beta_1,\dots,\beta_l$ such that $0\leq\beta_k\leq\alpha\delta$ and such that
\[
 \overline{H}'':=\overline{H}'+\beta_1\overline{A}_1+\dots+\beta_l\overline{A}_l
\]
is rational.
By Lemma~\ref{lem:abbig_char}(2),
\begin{align*}
 &(\overline{D}-\overline{H}'';\mathcal{V}) \\
 &\qquad =\left(\delta(\overline{D}-\overline{A};\mathcal{V})-\beta_1\overline{A}_1-\dots-\beta_l\overline{A}_l\right)+(1-\delta)(\overline{D}-\overline{M};\mathcal{V})
\end{align*}
is big, so that we can take a $\phi\in\aHzq{\QQ}(\overline{D}-\overline{H}'';\mathcal{V})\setminus\{0\}$ (see Remark~\ref{rem:ad_hoc}(2)).
Then $\overline{H}:=\overline{H}''-\widehat{(\phi)}$ has the required properties.

(4): By the assertion (3), we can find a $(\varphi:X'\to X,\overline{H})\in\aTheta_{\rm amp}(\overline{\mathscr{D}}^{\rm ad};\mathcal{V})$ such that
\begin{itemize}
\item $(1-\delta)\overline{M}\preceq\overline{H}$,
\item $\varphi_*^{-1}(\overline{\mathscr{D}}^{\rm ad};\mathcal{V})-\overline{H}$ is big, and
\item $\overline{H}$ has a $U'$-model of definition.
\end{itemize}
By reducing $\delta$ if necessary, we can assume that $\overline{H}-\delta\sum_{v\in M_K\setminus U'}(0,[v])$ is ample and that $\varphi_*^{-1}(\overline{\mathscr{D}};\mathcal{V})-\overline{H}-\delta\sum_{v\in M_K\setminus U}(0,[v])$ is big (see Theorems~\ref{thm:ample} and \ref{thm:abbig_cone}).
We can find a normal $O_K$-model $(\mathscr{X}',\mathscr{H}')$ of $(X',H)$ such that
\begin{itemize}
\item there is a birational and projective morphism $\widetilde{\varphi}:\mathscr{X}'\to\mathscr{X}$ extending $\varphi$,
\item $\mathscr{H}'$ is relatively nef,
\item $\overline{H}-\delta\sum_{v\in M_K\setminus U'}(0,[v])\leq\overline{\mathscr{H}}^{\prime{\rm ad}}\leq\overline{H}$, and
\item $\varphi_*^{-1}(\overline{\mathscr{D}}^{\rm ad};\mathcal{V})-\overline{\mathscr{H}}^{\prime\rm ad}$ is big,
\end{itemize}
where $\overline{\mathscr{H}}':=(\mathscr{H}',g_{\infty}^{\overline{H}})$.

We fix ample arithmetic Cartier divisors $\overline{\mathscr{A}}_1,\dots\overline{\mathscr{A}}_l$ such that $\mathscr{A}_1,\dots,\mathscr{A}_l$ are relatively ample and form a basis for a rational $\RR$-subspace of $\Div_{\RR}(\mathscr{X}')$ containing both $\widetilde{\varphi}^*\mathscr{D}$ and $\mathscr{H}'$.
We choose non-negative real numbers $\beta_1,\dots,\beta_l$ such that
\[
 \varphi_*^{-1}(\overline{\mathscr{D}}^{\rm ad};\mathcal{V})-\overline{\mathscr{H}}^{\prime{\rm ad}}-\beta_1\overline{\mathscr{A}}_1^{\rm ad}-\dots-\beta_l\overline{\mathscr{A}}_l^{\rm ad}
\]
is still big and such that $\widetilde{\varphi}^*\mathscr{D}-\mathscr{H}'-\beta_1\mathscr{A}_1-\dots-\beta_l\mathscr{A}_l$ belongs to $\Div_{\QQ}(\mathscr{X}')$.

Take a $\phi\in\aHzq{\QQ}(\widetilde{\varphi}^*\overline{\mathscr{D}}^{\rm ad}-\overline{\mathscr{H}}^{\prime{\rm ad}}-\beta_1\overline{\mathscr{A}}_1^{\rm ad}-\dots-\beta_l\overline{\mathscr{A}}_l^{\rm ad};\mathcal{V}^{\varphi})\setminus\{0\}$ (see Remark~\ref{rem:ad_hoc}(2)), and set
\[
 \overline{\mathscr{H}}:=\overline{\mathscr{H}}'+\beta_1\overline{\mathscr{A}}_1+\dots+\beta_l\overline{\mathscr{A}}_l-\widehat{(\phi)}.
\]
Then $(\widetilde{\varphi}:\mathscr{X}'\to\mathscr{X},\overline{\mathscr{H}})\in\aTheta_{\rm ad}(\overline{\mathscr{D}};\mathcal{V})$ has the required properties.
\end{proof}

\begin{definition}
Let $n$ be an integer such that $0\leq n\leq \dim X+1$, let $(\overline{D}_1;\mathcal{V}_1),\dots,(\overline{D}_n;\mathcal{V}_n)\in\aBVBig_{\RR,\RR}(X)$, and let $\overline{D}_{n+1},\dots,\overline{D}_{\dim X+1}\in\aNef_{\RR}(X)\cap\aBigCone_{\RR}(X)$.
We define
\begin{align}\label{eqn:aPosInt_first_version}
 &\langle(\overline{D}_1;\mathcal{V}_1)\cdots (\overline{D}_n;\mathcal{V}_n)\rangle\cdot\overline{D}_{n+1}\cdots\overline{D}_{\dim X+1}\\
 &\qquad\qquad :=\sup_{(\mu,\overline{M}_i)\in\aTheta(\overline{D}_i;\mathcal{V}_i)}\adeg\left(\overline{M}_1\cdots\overline{M}_n\cdot\mu^*\overline{D}_{n+1}\cdots\mu^*\overline{D}_{\dim X+1}\right),\nonumber
\end{align}
which is a positive real number.
\end{definition}

\begin{remark}\label{rem:aPosInt_prop}
Under the notations as above, one can easily see the following.
\begin{enumerate}
\item The map $\aBVBig_{\RR,\RR}(X)^{\times n}\times(\aNef_{\RR}(X)\cap\aBigCone_{\RR}(X))^{\times(\dim X+1-n)}\to\RR_{>0}$,
\begin{align*}
 &((\overline{D}_1;\mathcal{V}_1),\dots, (\overline{D}_n;\mathcal{V}_n);\overline{D}_{n+1},\dots,\overline{D}_{\dim X+1}) \\
 &\qquad\qquad\qquad \mapsto\langle(\overline{D}_1;\mathcal{V}_1)\cdots (\overline{D}_n;\mathcal{V}_n)\rangle\cdot\overline{D}_{n+1}\cdots\overline{D}_{\dim X+1},
\end{align*}
is symmetric in $(\overline{D}_1;\mathcal{V}_1),\dots,(\overline{D}_n;\mathcal{V}_n)$ (respectively, in $\overline{D}_{n+1},\dots,\overline{D}_{\dim X+1}$).
\item If $(\overline{D}_1';\mathcal{V}_1')\in\aBVBig_{\RR,\RR}(X)$ satisfies $(\overline{D}_1;\mathcal{V}_1)\preceq (\overline{D}_1';\mathcal{V}_1')$, then
\begin{align*}
 &\langle(\overline{D}_1;\mathcal{V}_1)\cdots (\overline{D}_n;\mathcal{V}_n)\rangle\cdot\overline{D}_{n+1}\cdots\overline{D}_{\dim X+1}\\
 &\qquad\qquad\qquad \leq\langle(\overline{D}_1';\mathcal{V}_1')\cdot(\overline{D}_2;\mathcal{V}_2)\cdots (\overline{D}_n;\mathcal{V}_n)\rangle\cdot\overline{D}_{n+1}\cdots\overline{D}_{\dim X+1}.
\end{align*}
\item\label{num:BC_homog} For every $a\in \RR_{>0}$,
\begin{align*}
 &\langle(a\overline{D}_1;a\mathcal{V}_1)\cdot(\overline{D}_2;\mathcal{V}_2)\cdots (\overline{D}_n;\mathcal{V}_n)\rangle\cdot\overline{D}_{n+1}\cdots\overline{D}_{\dim X+1}\\
 &\qquad\qquad\qquad =a\langle(\overline{D}_1;\mathcal{V}_1)\cdots (\overline{D}_n;\mathcal{V}_n)\rangle\cdot\overline{D}_{n+1}\cdots\overline{D}_{\dim X+1}.
\end{align*}
\item If $(\overline{D}_1';\mathcal{V}_1')\in\aBVBig_{\RR,\RR}(X)$, then
\begin{align*}
 &\langle(\overline{D}_1+\overline{D}_1';\mathcal{V}_1+\mathcal{V}_1')\cdot(\overline{D}_2;\mathcal{V}_2)\cdots (\overline{D}_n;\mathcal{V}_n)\rangle\cdot\overline{D}_{n+1}\cdots\overline{D}_{\dim X+1}\\
 &\qquad\qquad\qquad \geq\langle(\overline{D}_1;\mathcal{V}_1)\cdots (\overline{D}_n;\mathcal{V}_n)\rangle\cdot\overline{D}_{n+1}\cdots\overline{D}_{\dim X+1}\\
 &\qquad\qquad\qquad\qquad\qquad +\langle(\overline{D}_1;\mathcal{V}_1)\cdots (\overline{D}_n;\mathcal{V}_n)\rangle\cdot\overline{D}_{n+1}\cdots\overline{D}_{\dim X+1}.
\end{align*}
Note that the base conditions may not be effective.
\end{enumerate}
\end{remark}

\begin{proposition}\label{prop:aPosInt_continuous}
We use the notations as above.
Let $\overline{E}_{ab}\in\aDiv_{\RR}(X)$, $\mathcal{W}_{cd}\in\VDiv_{\RR}(X)$, $v_{ef}\in M_K\cup\{\infty\}$, and $\varphi_{ef}\in C^0_{v_{ef}}(X)$.
One has
\begin{align*}
 &\lim_{\varepsilon_{ab},\delta_{cd},\|\varphi_{ef}\|_{\sup}\to 0}\Biggl\langle\Biggl(\overline{D}_1+\sum_{b=1}^{p_1}\varepsilon_{1b}\overline{E}_{1b}+\sum_{f=1}^{r_1}(0,\varphi_{1f}[v_{1f}]);\mathcal{V}_1+\sum_{d=1}^{q_1}\delta_{1d}\mathcal{W}_{1d}\Biggr) \\
 &\qquad \cdots \Biggl(\overline{D}_n+\sum_{b=1}^{p_n}\varepsilon_{nb}\overline{E}_{nb}+\sum_{f=1}^{r_n}(0,\varphi_{nf}[v_{nf}]);\mathcal{V}_n+\sum_{d=1}^{q_n}\delta_{nd}\mathcal{W}_{nd}\Biggr)\Biggr\rangle \cdot\overline{D}_{n+1}\cdots\overline{D}_{\dim X+1} \\
 &\qquad\qquad\qquad\qquad\qquad\qquad\qquad =\langle(\overline{D}_1;\mathcal{V}_1)\cdots (\overline{D}_n;\mathcal{V}_n)\rangle\cdot\overline{D}_{n+1}\cdots\overline{D}_{\dim X+1},
\end{align*}
where the base conditions may not be effective.
\end{proposition}

\begin{proof}
By Remark~\ref{rem:aPosInt_prop}(2), we can assume $r_1=\dots=r_n=0$.
By Theorem~\ref{thm:abbig_cone}(1), there exists a $\gamma_0>0$ such that
\[
 (\overline{D}_i;\mathcal{V}_i)+\Biggl(\sum_{b=1}^{p_i}\varepsilon_{ib}\overline{E}_{ib};\sum_{d=1}^{q_i}\delta_{id}\mathcal{W}_{id}\Biggr)\in\aBVBig_{\RR,\RR}(X)
\]
for $i=1,\dots,n$, $\varepsilon_{ab}$, and $\delta_{cd}$ with $\max\{|\varepsilon_{ab}|,|\delta_{cd}|\}\leq\gamma_0$.
So,
\[
 (1-\gamma)(\overline{D}_i;\mathcal{V}_i)\preceq (\overline{D}_i;\mathcal{V}_i)+\Biggl(\sum_{b=1}^{p_i}\varepsilon_{ib}\overline{E}_{ib};\sum_{d=1}^{q_i}\delta_{id}\mathcal{W}_{id}\Biggr)\preceq (1+\gamma)(\overline{D}_i;\mathcal{V}_i)
\]
for every $\varepsilon_{ab},\delta_{cd}$ with $\max\{|\varepsilon_{ab}|,|\delta_{cd}|\}\leq\gamma_0\gamma$.

By Remark~\ref{rem:aPosInt_prop}(2),(3), we have
\begin{align*}
 &(1-\gamma)^n\langle(\overline{D}_1;\mathcal{V}_1)\cdots (\overline{D}_n;\mathcal{V}_n)\rangle\cdot\overline{D}_{n+1}\cdots\overline{D}_{\dim X+1} \\
 &\quad \leq \Biggl\langle\Biggl(\overline{D}_1+\sum_{b=1}^{p_1}\varepsilon_{1b}\overline{E}_{1b}+\sum_{f=1}^{r_1}(0,\varphi_{1f}[v_{1f}]);\mathcal{V}_1+\sum_{d=1}^{q_1}\delta_{1d}\mathcal{W}_{1d}\Biggr) \\
 &\qquad \cdots \Biggl(\overline{D}_n+\sum_{b=1}^{p_n}\varepsilon_{nb}\overline{E}_{nb}+\sum_{f=1}^{r_n}(0,\varphi_{nf}[v_{nf}]);\mathcal{V}_n+\sum_{d=1}^{q_n}\delta_{nd}\mathcal{W}_{nd}\Biggr)\Biggr\rangle \cdot\overline{D}_{n+1}\cdots\overline{D}_{\dim X+1} \\
 &\quad \leq (1+\gamma)^n\langle(\overline{D}_1;\mathcal{V}_1)\cdots (\overline{D}_n;\mathcal{V}_n)\rangle\cdot\overline{D}_{n+1}\cdots\overline{D}_{\dim X+1} 
\end{align*}
for every $\varepsilon_{ab},\delta_{cd}$ with $\max\{|\varepsilon_{ab}|,|\delta_{cd}|\}\leq\gamma_0\gamma$.
Hence the middle converges to $\langle(\overline{D}_1;\mathcal{V}_1)\cdots (\overline{D}_n;\mathcal{V}_n)\rangle\cdot\overline{D}_{n+1}\cdots\overline{D}_{\dim X+1}$ as $\max\{|\varepsilon_{ab}|,|\delta_{cd}|\}\to 0$.
\end{proof}

\begin{proposition}\label{prop:aPosInt_additive}
Let $n$ be an integer with $0\leq n\leq\dim X+1$, let $(\overline{D}_1;\mathcal{V}_1),\dots,(\overline{D}_n;\mathcal{V}_n)\in\aBVBig_{\RR,\RR}(X)$, and let $\overline{D}_{n+1},\dots,\overline{D}_{\dim X+1}\in\aNef_{\RR}(X)\cap\aBigCone_{\RR}(X)$.
\begin{enumerate}
\item Let $U$ be a nonempty open subset of $\Spec(O_K)$ over which a model of definition for $\overline{D}_1,\dots,\overline{D}_n$ exist, and let $\delta>0$.
One has
\begin{align*}
 &\langle(\overline{D}_1;\mathcal{V}_1)\cdots (\overline{D}_n;\mathcal{V}_n)\rangle\cdot\overline{D}_{n+1}\cdots\overline{D}_{\dim X+1} \\
 = &\sup_{(\mu,\overline{M}_i)\in\aTheta_{\rm amp}(\overline{D}_i,\mathcal{V}_i)}\adeg\left(\overline{M}_1\cdots\overline{M}_n\cdot\mu^*\overline{D}_{n+1}\cdots\mu^*\overline{D}_{\dim X+1}\right) \\
 = &\sup_{(\mathscr{X},\mathscr{D}_i)\in\aTheta_{U,\delta}(\overline{D}_i)}\sup_{(\widetilde{\mu},\overline{\mathscr{M}}_i)\in\aTheta_{\rm ad}(\overline{\mathscr{D}}_i^{\rm ad};\mathcal{V}_i)}\adeg\left(\overline{\mathscr{M}}_1^{\rm ad}\cdots\overline{\mathscr{M}}_n^{\rm ad}\cdot\widetilde{\mu}^*\overline{D}_{n+1}\cdots\widetilde{\mu}^*\overline{D}_{\dim X+1}\right).
\end{align*}
\item For $\overline{D}_{n+1}^{(1)},\overline{D}_{n+1}^{(2)}\in\aNef_{\RR}(X)\cap\aBigCone_{\RR}(X)$ one has
\begin{align*}
 &\langle(\overline{D}_1;\mathcal{V}_1)\cdots (\overline{D}_n;\mathcal{V}_n)\rangle\cdot(\overline{D}_{n+1}^{(1)}+\overline{D}_{n+1}^{(2)})\cdot\overline{D}_{n+2}\cdots\overline{D}_{\dim X+1} \\
 &\qquad\quad =\langle(\overline{D}_1;\mathcal{V}_1)\cdots (\overline{D}_n;\mathcal{V}_n)\rangle\cdot\overline{D}_{n+1}^{(1)}\cdots\overline{D}_{\dim X+1} \\
 &\qquad\qquad\qquad\qquad\qquad\quad +\langle(\overline{D}_1;\mathcal{V}_1)\cdots (\overline{D}_n;\mathcal{V}_n)\rangle\cdot\overline{D}_{n+1}^{(2)}\cdots\overline{D}_{\dim X+1}.
\end{align*}
\item  Let $e_1,\dots,e_r$ be integers such that $e_1+\dots+e_r=n$.
Then
\begin{align*}
 &\langle(\overline{D}_1;\mathcal{V}_1)^{\cdot e_1}\cdots (\overline{D}_r;\mathcal{V}_r)^{\cdot e_r}\rangle\cdot\overline{D}_{n+1}\cdots\overline{D}_{\dim X+1} \\
 &\qquad\quad =\sup_{(\mu,\overline{M}_i)\in\aTheta(\overline{D}_i;\mathcal{V}_i)}\adeg\left(\overline{M}_1^{\cdot e_1}\cdots\overline{M}_r^{\cdot e_r}\cdot\mu^*\overline{D}_{n+1}\cdots\mu^*\overline{D}_{\dim X+1}\right),
\end{align*}
where $\langle(\overline{D}_1;\mathcal{V}_1)^{\cdot e_1}\cdots (\overline{D}_r;\mathcal{V}_r)^{\cdot e_r}\rangle\cdot\overline{D}_{n+1}\cdots\overline{D}_{\dim X+1}$ is an abbreviation for
\[
 \langle\overbrace{(\overline{D}_1;\mathcal{V}_1)\cdots(\overline{D}_1;\mathcal{V}_1)}^{e_1}\cdots \overbrace{(\overline{D}_r;\mathcal{V}_r)\cdots(\overline{D}_r;\mathcal{V}_r)}^{e_r}\rangle\cdot\overline{D}_{n+1}\cdots\overline{D}_{\dim X+1}.
\]
\end{enumerate}
\end{proposition}

\begin{proof}
(1): We show the second equality.
The inequality $\geq$ is clear.
By Proposition~\ref{prop:aPosInt_continuous}, there exist $(\mathscr{X},\mathscr{D}_1)\in\aTheta_{U,\delta}(\overline{D}_1),\dots,(\mathscr{X},\mathscr{D}_n)\in\aTheta_{U,\delta}(\overline{D}_n)$ such that
\begin{align*}
 &\langle(\overline{D}_1;\mathcal{V}_1)\cdots (\overline{D}_n;\mathcal{V}_n)\rangle\cdot\overline{D}_{n+1}\cdots\overline{D}_{\dim X+1} \\
 &\qquad\qquad\qquad \leq\langle(\overline{\mathscr{D}}_1^{\rm ad};\mathcal{V}_1)\cdots (\overline{\mathscr{D}}_n^{\rm ad};\mathcal{V}_n)\rangle\cdot\overline{D}_{n+1}\cdots\overline{D}_{\dim X+1}+\varepsilon.
\end{align*}
Take $(\mu,\overline{M}_1)\in\aTheta(\overline{\mathscr{D}}_1^{\rm ad};\mathcal{V}_1),\dots,(\mu,\overline{M}_n)\in\aTheta(\overline{\mathscr{D}}_n^{\rm ad};\mathcal{V}_n)$ such that
\begin{align*}
 &\langle(\overline{\mathscr{D}}_1^{\rm ad};\mathcal{V}_1)\cdots (\overline{\mathscr{D}}_n^{\rm ad};\mathcal{V}_n)\rangle\cdot\overline{D}_{n+1}\cdots\overline{D}_{\dim X+1} \\
 &\qquad\qquad\qquad \leq\adeg\left(\overline{M}_1\cdots\overline{M}_n\cdot\mu^*\overline{D}_{n+1}\cdots\mu^*\overline{D}_{\dim X+1}\right)+\varepsilon.
\end{align*}
Let $U'$ be a nonempty open subset of $\Spec(O_K)$ over which models of definition for $\overline{M}_1,\dots,\overline{M}_n$ exist.
We can choose a sufficiently small $\delta>0$ such that
\begin{align*}
 &\adeg\left(\overline{M}_1\cdots\overline{M}_r\cdot\mu^*\overline{D}_{n+1}\cdots\mu^*\overline{D}_{\dim X+1}\right) \\
 &\qquad \leq\adeg\BIGl(\BIGl((1-\delta)\overline{M}_1-\delta\sum_{v\in M_K\setminus U'}(0,[v])\BIGr)\cdots\BIGl((1-\delta)\overline{M}_r-\delta\sum_{v\in M_K\setminus U'}(0,[v])\BIGr) \\
 &\qquad\qquad\qquad\qquad\qquad\qquad\qquad\qquad\qquad\qquad \cdot\mu^*\overline{D}_{n+1}\cdots\mu^*\overline{D}_{\dim X+1}\BIGr)+\varepsilon.
\end{align*}
All in all, we can find, by Proposition~\ref{prop:approx_by_models}, $(\widetilde{\mu},\overline{\mathscr{H}}_1)\in\aTheta_{\rm ad}(\overline{\mathscr{D}}_1;\mathcal{V}_1),\dots,(\widetilde{\mu},\overline{\mathscr{H}}_n)\in\aTheta_{\rm ad}(\overline{\mathscr{D}}_n;\mathcal{V}_n)$ such that
\begin{align*}
 &\langle(\overline{D}_1;\mathcal{V}_1)\cdots (\overline{D}_n;\mathcal{V}_n)\rangle\cdot\overline{D}_{n+1}\cdots\overline{D}_{\dim X+1} \\
 &\qquad \leq\adeg\left(\overline{M}_1\cdots\overline{M}_r\cdot\mu^*\overline{D}_{n+1}\cdots\mu^*\overline{D}_{\dim X+1}\right)+2\varepsilon \\
 &\qquad \leq\adeg\left((\overline{\mathscr{H}}_1^{\rm ad})\cdots(\overline{\mathscr{H}}_n^{\rm ad})\cdot\widetilde{\mu}^*\overline{D}_{n+1}\cdots\widetilde{\mu}^*\overline{D}_{\dim X+1}\right)+3\varepsilon \\
 &\qquad \leq\sup_{(\widetilde{\mu},\overline{\mathscr{M}}_i)\in\aTheta_{\rm ad}(\overline{\mathscr{D}}_i^{\rm ad};\mathcal{V}_i)}\adeg\left(\overline{\mathscr{M}}_1^{\rm ad}\cdots\overline{\mathscr{M}}_n^{\rm ad}\cdot\widetilde{\mu}^*\overline{D}_{n+1}\cdots\widetilde{\mu}^*\overline{D}_{\dim X+1}\right)+3\varepsilon.
\end{align*}
So we conclude the proof.

(2): The inequality $\leq$ is clear.
Let $U$ be a nonempty open subset of $\Spec(O_K)$ over which models of definition for $\overline{D}_1,\dots,\overline{D}_n$ exist.
Given any $\varepsilon>0$, there exists a $\delta>0$ such that
\begin{align*}
 &\langle(\overline{D}_1;\mathcal{V}_1)\cdots (\overline{D}_n;\mathcal{V}_n)\rangle\cdot\overline{D}_{n+1}^{(i)}\cdots\overline{D}_{\dim X+1} \\
 &\qquad\quad \leq \left\langle\left(\overline{D}_1-\delta\sum_{v\in M_K\setminus U}(0,[v]);\mathcal{V}_1\right)\cdots \left(\overline{D}_n-\delta\sum_{v\in M_K\setminus U}(0,[v]);\mathcal{V}_n\right)\right\rangle \\
 &\qquad\qquad\qquad\qquad\qquad\qquad\qquad\qquad\qquad\qquad\qquad \cdot\overline{D}_{n+1}^{(i)}\cdots\overline{D}_{\dim X+1}+\varepsilon.
\end{align*}
for $i=1,2$ (see Proposition~\ref{prop:aPosInt_continuous}).
By Proposition~\ref{prop:approx_by_models}(1) and Remark~\ref{rem:aPosInt_prop}(2), there exist $(\mathscr{X},\mathscr{D}_1)\in\aTheta_{U,\delta}(\overline{D}_1),\dots,(\mathscr{X},\mathscr{D}_n)\in\aTheta_{U,\delta}(\overline{D}_n)$ such that
\begin{align*}
 &\langle(\overline{D}_1;\mathcal{V}_1)\cdots (\overline{D}_n;\mathcal{V}_n)\rangle\cdot\overline{D}_{n+1}^{(i)}\cdots\overline{D}_{\dim X+1} \\
 &\qquad\qquad\qquad \leq\langle(\overline{\mathscr{D}}_1^{\rm ad};\mathcal{V}_1)\cdots (\overline{\mathscr{D}}_n^{\rm ad};\mathcal{V}_n)\rangle\cdot\overline{D}_{n+1}^{(i)}\cdots\overline{D}_{\dim X+1}+\varepsilon
\end{align*}
for $i=1,2$.
Hence, by Proposition~\ref{prop:approx_by_models}(2),(4),
\begin{align*}
 &\langle(\overline{D}_1;\mathcal{V}_1)\cdots (\overline{D}_n;\mathcal{V}_n)\rangle\cdot\overline{D}_{n+1}^{(1)}\cdots\overline{D}_{\dim X+1} \\
 &\qquad\qquad\qquad\qquad\qquad +\langle(\overline{D}_1;\mathcal{V}_1)\cdots (\overline{D}_n;\mathcal{V}_n)\rangle\cdot\overline{D}_{n+1}^{(2)}\cdots\overline{D}_{\dim X+1} \\
 &\qquad\qquad \leq\langle(\overline{\mathscr{D}}_1^{\rm ad};\mathcal{V}_1)\cdots (\overline{\mathscr{D}}_n^{\rm ad};\mathcal{V}_n)\rangle\cdot\overline{D}_{n+1}^{(1)}\cdots\overline{D}_{\dim X+1} \\
 &\qquad\qquad\qquad\qquad\qquad\qquad\qquad +\langle(\overline{\mathscr{D}}_1^{\rm ad};\mathcal{V}_1)\cdots (\overline{\mathscr{D}}_n^{\rm ad};\mathcal{V}_n)\rangle\cdot\overline{D}_{n+1}^{(2)}\cdots\overline{D}_{\dim X+1}+2\varepsilon \\
 &\qquad\qquad =\langle(\overline{\mathscr{D}}_1^{\rm ad};\mathcal{V}_1)\cdots (\overline{\mathscr{D}}_n^{\rm ad};\mathcal{V}_n)\rangle\cdot(\overline{D}_{n+1}^{(1)}+\overline{D}_{n+1}^{(2)})\cdots\overline{D}_{\dim X+1}+2\varepsilon \\
 &\qquad\qquad \leq\langle(\overline{D}_1;\mathcal{V}_1)\cdots (\overline{D}_n;\mathcal{V}_n)\rangle\cdot(\overline{D}_{n+1}^{(1)}+\overline{D}_{n+1}^{(2)})\cdots\overline{D}_{\dim X+1}+2\varepsilon
\end{align*}
for every $\varepsilon>0$.

(3): The inequality $\geq$ is clear.
Let $U$ be a nonempty open subset of $\Spec(O_K)$ over which models of definition for $\overline{D}_1,\dots,\overline{D}_r$ exist.
Given any $\varepsilon>0$, there exists a $\delta>0$ such that
\begin{align*}
 &\langle(\overline{D}_1;\mathcal{V}_1)^{\cdot e_1}\cdots (\overline{D}_r;\mathcal{V}_r)^{\cdot e_r}\rangle\cdot\overline{D}_{n+1}\cdots\overline{D}_{\dim X+1} \\
 &\qquad \leq \left\langle\left(\overline{D}_1-\delta\sum_{v\in M_K\setminus U}(0,[v]);\mathcal{V}_1\right)^{\cdot e_r}\cdots \left(\overline{D}_r-\delta\sum_{v\in M_K\setminus U}(0,[v]);\mathcal{V}_r\right)^{\cdot e_r}\right\rangle \\
 &\qquad\qquad\qquad\qquad\qquad\qquad\qquad\qquad\qquad\qquad\qquad \cdot\overline{D}_{n+1}\cdots\overline{D}_{\dim X+1}+\varepsilon.
\end{align*}
(see Proposition~\ref{prop:aPosInt_continuous}).
By Proposition~\ref{prop:approx_by_models}(1) and Remark~\ref{rem:aPosInt_prop}(2), there exist $(\mathscr{X},\mathscr{D}_1)\in\aTheta_{U,\delta}(\overline{D}_1),\dots,(\mathscr{X},\mathscr{D}_n)\in\aTheta_{U,\delta}(\overline{D}_n)$ such that
\begin{align*}
 &\langle(\overline{D}_1;\mathcal{V}_1)^{\cdot e_1}\cdots (\overline{D}_r;\mathcal{V}_r)^{\cdot e_r}\rangle\cdot\overline{D}_{n+1}\cdots\overline{D}_{\dim X+1} \\
 &\qquad\qquad\qquad \leq\langle(\overline{\mathscr{D}}_1^{\rm ad};\mathcal{V}_1)^{\cdot e_1}\cdots (\overline{\mathscr{D}}_r^{\rm ad};\mathcal{V}_r)^{\cdot e_r}\rangle\cdot\overline{D}_{n+1}\cdots\overline{D}_{\dim X+1}+\varepsilon.
\end{align*}
Hence, by Proposition~\ref{prop:approx_by_models}(2),(4),
\begin{align*}
 &\langle(\overline{D}_1;\mathcal{V}_1)^{\cdot e_1}\cdots (\overline{D}_r;\mathcal{V}_r)^{\cdot e_r}\rangle\cdot\overline{D}_{n+1}\cdots\overline{D}_{\dim X+1} \\
 &\qquad \leq\langle(\overline{\mathscr{D}}_1^{\rm ad};\mathcal{V}_1)^{\cdot e_1}\cdots (\overline{\mathscr{D}}_r^{\rm ad};\mathcal{V}_r)^{\cdot e_r}\rangle\cdot\overline{D}_{n+1}\cdots\overline{D}_{\dim X+1}+\varepsilon \\
 &\qquad =\sup_{(\widetilde{\mu},\overline{\mathscr{M}}_i)\in\aTheta_{\rm ad}(\overline{\mathscr{D}}_i;\mathcal{V}_i)}\adeg\left((\overline{\mathscr{M}}_1^{\rm ad})^{\cdot e_1}\cdots(\overline{\mathscr{M}}_r^{\rm ad})^{\cdot e_r}\cdot\widetilde{\mu}^*\overline{D}_{n+1}\cdots\widetilde{\mu}^*\overline{D}_{\dim X+1}\right)+\varepsilon \\
 &\qquad \leq\sup_{(\mu,\overline{M}_i)\in\aTheta(\overline{D}_i;\mathcal{V}_i)}\adeg\left(\overline{M}_1^{\cdot e_1}\cdots\overline{M}_r^{\cdot e_r}\cdot\mu^*\overline{D}_{n+1}\cdots\mu^*\overline{D}_{\dim X+1}\right)+\varepsilon
\end{align*}
for every $\varepsilon>0$.
\end{proof}

\begin{definition}\label{defn:aPosInt}
By Proposition~\ref{prop:aPosInt_additive}(2), we can extend the map (\ref{eqn:aPosInt_first_version}) to $\aBVBig_{\RR,\RR}(X)^{\times n}\times\aInt_{\RR}(X)^{\times(\dim X+1-n)} \to \RR$,
\begin{align*}
 &((\overline{D}_1;\mathcal{V}_1),\dots, (\overline{D}_n;\mathcal{V}_n);\overline{D}_{n+1},\dots,\overline{D}_{\dim X+1}) \\
 &\qquad\qquad\qquad\qquad\qquad \mapsto\langle(\overline{D}_1;\mathcal{V}_1)\cdots (\overline{D}_n;\mathcal{V}_n)\rangle\cdot\overline{D}_{n+1}\cdots\overline{D}_{\dim X+1}.
\end{align*}

If $n=\dim X$, then, by using the same arguments as in \cite[Proposition~3.10(3)]{IkomaCon}, we have a map $\aBVBig_{\RR,\RR}(X)^{\times \dim X}\times\aDiv_{\RR}(X) \to \RR$,
\begin{align*}
 &((\overline{D}_1;\mathcal{V}_1),\dots, (\overline{D}_{\dim X};\mathcal{V}_{\dim X});\overline{D}_{\dim X+1}) \\
 &\qquad\qquad\qquad\qquad\qquad \mapsto\langle(\overline{D}_1;\mathcal{V}_1)\cdots (\overline{D}_{\dim X};\mathcal{V}_{\dim X})\rangle\cdot\overline{D}_{\dim X+1}.
\end{align*}
\end{definition}

\begin{theorem}\label{thm:approx1}
For every $(\overline{D};\mathcal{V})\in\aBVBig_{\RR,\RR}(X)$, one has
\[
 \avol(\overline{D};\mathcal{V})=\langle(\overline{D};\mathcal{V})^{\cdot(\dim X+1)}\rangle.
\]
\end{theorem}

\begin{proof}
The inequality $\geq$ results from Corollary~\ref{cor:cont}.
First, we assume $(\overline{D};\mathcal{V})\in\aBVBig_{\ZZ,\RR}(X)$.
For any $\varepsilon>0$, there exists an $m_1\geq 1$ such that
\begin{equation}
 \avol(\overline{D};\mathcal{V})\leq\frac{\log\sharp\aHz(m\overline{D};m\mathcal{V})}{m^{\dim X+1}/(\dim X+1)!}+\varepsilon
\end{equation}
for every $m\geq m_1$ (see Theorem~\ref{thm:Boucksom_Chen}).

By Lemma~\ref{lem:approx_Zariski}(2), there exists an $m_2\geq m_1$ such that
\[
 (\varphi_{m_2m},\overline{M}(m_2m\overline{D};m_2m\mathcal{V})/(m_2m))\in\aTheta(\overline{D};\mathcal{V})
\]
for every $m\geq 1$.
By Lemma~\ref{lem:approx_Zariski}(1) and Theorem~\ref{thm:Boucksom_Chen},
\begin{equation}
 \frac{\log\sharp\aHz(mm_2\overline{D};mm_2\mathcal{V})}{(mm_2)^{\dim X+1}/(\dim X+1)!}=\frac{\log\sharp\aHz(\overline{M}(mm_2\overline{D};mm_2\mathcal{V}))}{(mm_2)^{\dim X+1}/(\dim X+1)!}
\end{equation}
for every $m\geq 1$ and
\begin{equation}
 \frac{\log\sharp\aHz(\overline{M}(mm_2\overline{D};mm_2\mathcal{V}))}{(mm_2)^{\dim X+1}/(\dim X+1)!}\leq\frac{\avol(\overline{M}(m_2\overline{D};m_2\mathcal{V}))}{m_2^{\dim X+1}}+\varepsilon
\end{equation}
for every $m\gg 1$.

All in all, we have
\[
 \avol(\overline{D};\mathcal{V})\leq\frac{\avol(\overline{M}(m_2\overline{D};m_2\mathcal{V}))}{m_2^{\dim X+1}}+2\varepsilon\leq\langle(\overline{D};\mathcal{V})^{\cdot(\dim X+1)}\rangle+2\varepsilon
\]
for every $\varepsilon>0$.

In general, by homogeneity and continuity (see Proposition~\ref{prop:aPosInt_continuous} and Theorem~\ref{thm:Boucksom_Chen}), the assertion holds for every $(\overline{D};\mathcal{V})\in\aBVBig_{\RR,\RR}(X)$.
\end{proof}

\section{Differentiability of the arithmetic volumes}\label{sec:Diff}

\subsection{Proof of Theorem~A}\label{subsec:Main}

We recall the arithmetic Siu inequality of Yuan and its consequence (see \cite[Theorem~2.2]{Yuan07} and \cite[Proposition~6.1]{IkomaCon}).

\begin{proposition}\label{prop:Yuan}
\begin{enumerate}
\item Let $\overline{M},\overline{N}$ be nef adelic $\RR$-Cartier divisors on $X$.
Then
\[
 \avol(\overline{M}-\overline{N})\geq\adeg\left(\overline{M}^{\cdot(\dim X+1)}\right)-(\dim X+1)\adeg\left(\overline{M}^{\cdot\dim X}\cdot\overline{N}\right).
\]
\item Let $\overline{M},\overline{D}'\in\aDiv_{\RR}(X)$.
Suppose that $\overline{M}$ is nef, and there exists a nef and big adelic $\RR$-Cartier divisor $\overline{A}$ such that $\overline{A}\pm\overline{D}'$ is nef and $\overline{A}-\overline{M}$ is pseudo-effective.
Then for all $r\in\RR$
\begin{align*}
 &\avol(\overline{M}+r\overline{D}')-\avol(\overline{M}) \\
 &\qquad\qquad\quad \geq (\dim X+1)\adeg(\overline{M}^{\cdot\dim X}\cdot\overline{D}')\cdot r-C(|r|)\avol(\overline{A})\cdot r^2,
\end{align*}
where $C(|r|):=2\dim X(\dim X+1)(1+|r|)^{\dim X-1}$.
\end{enumerate}
\end{proposition}

The main purpose of this paper is the following.

\begin{theorem}\label{thm:diff_along_arith}
Let $(\overline{D};\mathcal{V})\in\aBVDiv_{\RR,\RR}(X)$ and $\overline{D}'\in\aDiv_{\RR}(X)$.
If $(\overline{D};\mathcal{V})$ is big, then the function
\[
 \RR\ni r\mapsto\avol(\overline{D}+r\overline{D}';\mathcal{V})\in\RR
\]
is two-sided differentiable at $r=0$, and
\[
 \lim_{r\to 0}\frac{\avol(\overline{D}+r\overline{D}';\mathcal{V})-\avol(\overline{D};\mathcal{V})}{r}=(\dim X+1)\langle(\overline{D};\mathcal{V})^{\cdot\dim X}\rangle\cdot\overline{D}'.
\]
\end{theorem}

\begin{proof}
First, we assume that $\overline{D}'$ is integrable, and fix a nef and big adelic $\RR$-Cartier divisor $\overline{A}$ such that $\overline{A}\pm\overline{D}'$ is nef and $\overline{A}-\overline{D}$ is pseudo-effective.
By Proposition~\ref{prop:Yuan}(1), for every $r\in\RR$ with $|r|\leq 1$ and $(\varphi,\overline{M})\in\aTheta(\overline{D};\mathcal{V})$,
\begin{align*}
 \avol(\overline{D}+r\overline{D}';\mathcal{V}) &\geq\avol(\overline{M}+r\varphi^*\overline{D}') \\
 &\geq \avol(\overline{M})+(\dim X+1)\adeg(\overline{M}^{\cdot\dim X}\cdot\varphi^*\overline{D}')\cdot r-C\avol(\overline{A})\cdot r^2
\end{align*}
and, for every $r\in\RR$ with $|r|\leq 1$ and $(\varphi_r,\overline{M}_r)\in\aTheta(\overline{D}+r\overline{D}';\mathcal{V})$,
\begin{align*}
 \avol(\overline{D};\mathcal{V}) &\geq\avol(\overline{M}_r-r\varphi_r^*\overline{D}') \\
 &\geq\avol(\overline{M}_r)-(\dim X+1)\adeg(\overline{M}_r^{\cdot\dim X}\cdot\varphi_r^*\overline{D}')\cdot r-C\avol(2\overline{A})\cdot r^2,
\end{align*}
where we set $C:=2^{\dim X}\dim X(\dim X+1)$.
Note that $(\overline{D}+r\overline{D}';\mathcal{V})\in\aBVBig_{\RR}(X)$ for every $r$ with $|r|$ sufficiently small (Theorem~\ref{thm:abbig_cone}(1)).
Hence, by Theorem~\ref{thm:approx1},
\[
 \avol(\overline{D}+r\overline{D}';\mathcal{V})-\avol(\overline{D};\mathcal{V})\geq (\dim X+1)r\langle(\overline{D};\mathcal{V})^{\cdot\dim X}\rangle\cdot\overline{D}'-Cr^2\avol(\overline{A})
\]
and
\begin{align*}
 &\avol(\overline{D};\mathcal{V})-\avol(\overline{D}+r\overline{D}';\mathcal{V}) \\
 &\qquad\qquad\qquad\qquad \geq -(\dim X+1)r\langle(\overline{D}+r\overline{D}';\mathcal{V})^{\cdot\dim X}\rangle\cdot\overline{D}'-Cr^2\avol(2\overline{A})
\end{align*}
hold for all $r$ with $|r|\ll 1$.
Thus, by Remark~\ref{rem:aPosInt_prop}(4), we conclude.

Next, in general, one can find, by the Stone-Weierstrass theorem, a sequence of continuous functions $(f_n)_{n\geq 1}$ such that $\overline{D}'+(0,2f_n[\infty])$ is of $C^{\infty}$-type and $\|f_n\|_{\sup}\to 0$ as $n\to\infty$.
Since
\begin{align*}
 &\left|\frac{\avol(\overline{D}+r\overline{D}';\mathcal{V})-\avol(\overline{D};\mathcal{V})}{r}-\frac{\avol(\overline{D}+r(\overline{D}'+(0,2f_n[\infty]));\mathcal{V})-\avol(\overline{D};\mathcal{V})}{r}\right|\\
 &\qquad\qquad\qquad\qquad\qquad\qquad\qquad\qquad \leq (\dim X+1)(\|f_n\|_{\sup})[K:\QQ]\vol(D+rD')
\end{align*}
for all $r\in\RR\setminus\{0\}$ and $n\geq 1$, we have, by Remark~\ref{rem:aPosInt_prop}(4),
\[
 \lim_{r\to 0}\frac{\avol(\overline{D}+r\overline{D}';\mathcal{V})-\avol(\overline{D};\mathcal{V})}{r}=(\dim X+1)\langle(\overline{D};\mathcal{V})^{\cdot\dim X}\rangle\cdot\overline{D}'.
\]
\end{proof}

\subsection{Arithmetic Bonnesen--Diskant inequalities}

By the same arguments as in \cite[section~2.3]{Yuan_Zhang13}, we can easily see that the arithmetic Hodge index theorem is also valid for adelic $\RR$-Cartier divisors.

\begin{theorem}[Arithmetic Hodge index theorem]\label{thm:aHodge}
Let $X$ be a projective variety over a number field.
Let $\overline{D}$ be an integrable adelic $\RR$-Cartier divisor on $X$, and let $\overline{H}_1,\dots,\overline{H}_{\dim X}$ be nef adelic $\RR$-Cartier divisors on $X$ such that $H_1,\dots,H_{\dim X-1}$ are big.
\begin{enumerate}
\item If $\deg(D\cdot H_2\cdots H_{\dim X})=0$, then $\adeg(\overline{D}^{\cdot 2}\cdot\overline{H}_2\cdots\overline{H}_{\dim X})\leq 0$.
\item If $\adeg(\overline{D}\cdot\overline{H}_1\cdots\overline{H}_{\dim X})=0$, then $\adeg(\overline{D}^{\cdot 2}\cdot\overline{H}_2\cdots\overline{H}_{\dim X})\leq 0$.
\end{enumerate}
\end{theorem}

\begin{proof}
We reproduce the proof for reader's convenience.

(1): Suppose that $\overline{H}_1,\dots\overline{H}_{\dim X-1}$ are nef adelic $\QQ$-Cartier divisors on $X$.
Choose $\overline{D}_1,\dots,\overline{D}_l\in\aDiv(X)$ and $a_1,\dots,a_l\in\RR$ such that $a_1,\dots,a_l$ are $\QQ$-linearly independent and $\overline{D}=a_1\overline{D}_1+\dots+a_l\overline{D}_l$.
Since
\[
 \sum_{i=1}^la_i\deg(D_i\cdot H_1\cdots H_{\dim X-1})=0
\]
and $\deg(D_i\cdot H_1\cdots H_{\dim X-1})\in\QQ$, we have $\deg(D_i\cdot H_1\cdots H_{\dim X-1})=0$ for all $i$.
So, by \cite[Theorem~1.3]{Yuan_Zhang13}, we have
\[
 \adeg((b_1\overline{D}_1+\dots+b_l\overline{D}_l)^{\cdot 2}\cdot\overline{H}_1\cdots\overline{H}_{d-1})\leq 0
\]
for all $b_1,\dots,b_l\in\QQ$, and $\adeg(\overline{D}^{\cdot 2}\cdot\overline{H}_1\cdots\overline{H}_{d-1})\leq 0$ by continuity.

Next, in general, we fix an ample adelic Cartier divisor $\overline{A}$ on $X$.
For each $i=2,\dots,\dim X$, we choose a sequence $(\overline{A}_i^{(j)})_{j=1}^{\infty}$ of nef adelic $\RR$-Cartier divisors, all of which are contained in a finite dimensional $\RR$-subspace of $\aDiv_{\RR}(X)$, having the properties that $\overline{A}_i^{(j)}\to 0$ as $j\to\infty$ and $\overline{H}_i^{(j)}:=\overline{H}_i+\overline{A}_i^{(j)}$ are all rational.
For each $j$, we set
\[
 \varepsilon_j:=-\frac{\deg\left(D\cdot H_2^{(j)}\cdots H_{\dim X}^{(j)}\right)}{\deg\left(A\cdot H_1^{(j)}\cdots H_{\dim X}^{(j)}\right)}\in\RR.
\]
Then
\[
 \deg\left((D+\varepsilon_jA)\cdot H_2^{(j)}\cdots H_{\dim X}^{(j)}\right)=0,
\]
$\overline{H}_i^{(j)}\in\aNef_{\QQ}(X)$, and $H_i^{(j)}$ are big, so, by the first case, we have
\[
 \adeg\left((\overline{D}+\varepsilon_j\overline{A})^{\cdot 2}\cdot\overline{H}_2^{(j)}\cdots\overline{H}_{\dim X}^{(j)}\right)\leq 0.
\]
As $j\to\infty$, we have $\overline{H}_i^{(j)}\to\overline{H}_i$ and
\[
 \varepsilon_j\to -\frac{\deg\left(D\cdot H_2\cdots H_{\dim X}\right)}{\deg\left(A\cdot H_2\cdots H_{\dim X}\right)}=0,
\]
where $\deg(A\cdot H_2\cdots H_{\dim X})>0$ since $H_2,\dots,H_{\dim X}$ are all big.
So we conclude by continuity.

(2): Set $t:=\deg\left(D\cdot H_2\cdots H_{\dim X}\right)/\deg\left(H_1\cdots H_{\dim X}\right)\in\RR$.
Since
\[
 \deg\left((D-tH_1)\cdot H_2\cdots H_{\dim X}\right)=0,
\]
we have by the assertion (1)
\begin{align*}
 &\adeg\left((\overline{D}-t\overline{H}_1)^{\cdot 2}\cdot\overline{H}_2\cdots\overline{H}_{\dim X}\right)\\
 &\qquad\qquad=\adeg\left(\overline{D}^{\cdot 2}\cdot\overline{H}_2\cdots\overline{H}_{\dim X}\right)+t^2\adeg\left(\overline{H}^{\cdot 2}\cdot\overline{H}_1\cdots\overline{H}_{d-1}\right)\leq 0,
\end{align*}
so $\adeg\left(\overline{D}^{\cdot 2}\cdot\overline{H}_2\cdots\overline{H}_{\dim X}\right)\leq 0$.
\end{proof}

\begin{corollary}\label{cor:aKT}
Let $\overline{D},\overline{E},\overline{H}_1,\dots,\overline{H}_{\dim X+1}$ be nef adelic $\RR$-Cartier divisors on $X$.
\begin{enumerate}
\item One has
\begin{align*}
 &\adeg\left(\overline{D}\cdot\overline{E}\cdot\overline{H}_3\cdots\overline{H}_{\dim X+1}\right)^2 \\
 &\qquad\quad \geq\adeg\left(\overline{D}^{\cdot 2}\cdot\overline{H}_3\cdots\overline{H}_{\dim X+1}\right)\cdot\adeg\left(\overline{E}^{\cdot 2}\cdot\overline{H}_3\cdots\overline{H}_{\dim X+1}\right).
\end{align*}
\item For every $n$ with $0\leq n\leq \dim X+1$ and for every $i$ with $0\leq i\leq n$, one has
\begin{align*}
 &\adeg\left(\overline{D}^{\cdot (n-i)}\cdot\overline{E}^{\cdot i}\cdot\overline{H}_{n+1}\cdots\overline{H}_{\dim X+1}\right)^n \\
 &\qquad \geq\adeg\left(\overline{D}^{\cdot n}\cdot\overline{H}_{n+1}\cdots\overline{H}_{\dim X+1})^{n-i}\cdot\adeg(\overline{E}^{\cdot n}\cdot\overline{H}_{n+1}\cdots\overline{H}_{\dim X+1}\right)^i.
\end{align*}
\item For every $n$ with $0\leq n\leq \dim X+1$, one has
\[
 \adeg\left(\overline{H}_1\cdots\overline{H}_{\dim X+1}\right)^n\geq\prod_{i=1}^n\adeg\left(\overline{H}_i^{\cdot n}\cdot\overline{H}_{n+1}\cdots\overline{H}_{\dim X+1}\right).
\]
\item For every $n$ with $1\leq n\leq \dim X$, one has
\begin{align*}
 &\adeg\left(\overline{D}^{\cdot n}\cdot\overline{E}^{\cdot(\dim X-n+1)}\right) \\
 &\qquad\qquad \geq\adeg\left(\overline{D}^{\cdot(n-1)}\cdot\overline{E}^{\cdot(\dim X-n+2)}\right)\cdot\adeg\left(\overline{D}^{\cdot(n+1)}\cdot\overline{E}^{\cdot(\dim X-n)}\right).
\end{align*}
\item For every $n$ with $1\leq n\leq \dim X+1$, one has
\begin{align*}
 &\adeg\left((\overline{D}+\overline{E})^{\cdot n}\cdot\overline{H}_{n+1}\cdots\overline{H}_{\dim X+1}\right)^{1/n} \\
 &\quad \geq\adeg\left(\overline{D}^{\cdot n}\cdot\overline{H}_{n+1}\cdots\overline{H}_{\dim X+1}\right)^{1/n}+\adeg\left(\overline{E}^{\cdot n}\cdot\overline{H}_{n+1}\cdots\overline{H}_{\dim X+1}\right)^{1/n}.
\end{align*}
\end{enumerate}
\end{corollary}

\begin{proof}
These result by formally transforming the inequality in Theorem~\ref{thm:aHodge}(2) as in the proof of \cite[Theorem~2.9 and Remark~2.10]{IkomaCon}.
\end{proof}

\begin{proposition}\label{prop:akt2}
Let $n$ be an integer with $0\leq n\leq\dim X+1$, let $(\overline{D}_1;\mathcal{V}_1),(\overline{D}_2;\mathcal{V}_2)\in\aBVBig_{\RR,\RR}(X)$, and let $\overline{H}_{n+1},\dots,\overline{H}_{\dim X+1}\in\aNef_{\RR}(X)\cap\aBigCone_{\RR}(X)$.
\begin{enumerate}
\item One has
\begin{align*}
 &\avol(\overline{D}_1+\overline{D}_2;\mathcal{V}_1+\mathcal{V}_2) \\
 &\qquad\qquad\quad \geq\sum_{i=0}^{\dim X+1}\binom{\dim X+1}{i}\langle(\overline{D}_1;\mathcal{V}_1)^{\cdot i}\cdot(\overline{D}_2;\mathcal{V}_2)^{\cdot (\dim X+1-i)}\rangle
\end{align*}
\item For every $n$ with $1\leq n\leq \dim X$, one has
\begin{align*}
 &\langle(\overline{D}_1;\mathcal{V}_1)^{\cdot n}\cdot(\overline{D}_2;\mathcal{V}_2)^{\cdot(\dim X-n+1)}\rangle^2 \\
 &\qquad\qquad\quad \geq\langle(\overline{D}_1;\mathcal{V}_1)^{\cdot(n-1)}\cdot(\overline{D}_2;\mathcal{V}_2)^{\cdot(\dim X-n+2)}\rangle \\
 &\qquad\qquad\qquad\qquad\qquad\qquad\quad \cdot\langle(\overline{D}_1;\mathcal{V}_1)^{\cdot(n+1)}\cdot(\overline{D}_2;\mathcal{V}_2)^{\cdot(\dim X-n)}\rangle.
\end{align*}
\item For every $n$ with $0\leq n\leq \dim X+1$, one has
\begin{align*}
 &\langle(\overline{D}_1;\mathcal{V}_1)^{\cdot n}\cdot(\overline{D}_2;\mathcal{V}_2)^{\cdot(\dim X-n+1)}\rangle^{\dim X+1} \\
 &\qquad\qquad\qquad\qquad\qquad\qquad \geq\avol(\overline{D}_1;\mathcal{V}_1)^n\cdot\avol(\overline{D}_2;\mathcal{V}_2)^{\dim X-n+1}.
\end{align*}
\item For every $n$ with $1\leq n\leq \dim X+1$, one has
\begin{align*}
 &\left(\langle(\overline{D}_1+\overline{D}_2;\mathcal{V}_1+\mathcal{V}_2)^{\cdot n}\rangle\cdot\overline{H}_{n+1}\cdots\overline{H}_{\dim X+1}\right)^{1/n} \\
 &\qquad\qquad\quad \geq\left(\langle(\overline{D}_1;\mathcal{V}_1)^{\cdot n}\rangle\cdot\overline{H}_{n+1}\cdots\overline{H}_{\dim X+1}\right)^{1/n} \\
 &\qquad\qquad\qquad\qquad\qquad\qquad +\left(\langle(\overline{D}_2;\mathcal{V}_2)^{\cdot n}\rangle\cdot\overline{H}_{n+1}\cdots\overline{H}_{\dim X+1}\right)^{1/n}.
\end{align*}
\end{enumerate}
\end{proposition}

\begin{proof}
(1): Given any $\varepsilon>0$, there exist, by Proposition~\ref{prop:aPosInt_continuous} and Proposition~\ref{prop:approx_by_models}, $(\mathscr{X},\mathscr{D}_1)\in\aTheta_{\rm mod}(\overline{D}_1)$ and $(\mathscr{X},\mathscr{D}_2)\in\aTheta_{\rm mod}(\overline{D}_2)$ such that
\begin{align*}
 &\sum_{i=0}^{\dim X+1}\binom{\dim X+1}{i}\langle(\overline{D}_1;\mathcal{V}_1)^{\cdot i}\cdot(\overline{D}_2;\mathcal{V}_2)^{\cdot (\dim X+1-i)}\rangle \\
 &\qquad\qquad \leq\sum_{i=0}^{\dim X+1}\binom{\dim X+1}{i}\langle(\overline{\mathscr{D}}_1^{\rm ad};\mathcal{V}_1)^{\cdot i}\cdot(\overline{\mathscr{D}}_2^{\rm ad};\mathcal{V}_2)^{\cdot (\dim X+1-i)}\rangle+\varepsilon.
\end{align*}
By Proposition~\ref{prop:approx_by_models}, there exist $(\widetilde{\varphi}:\mathscr{X}'\to\mathscr{X},\overline{\mathscr{M}}_1)\in\aTheta_{\rm ad}(\overline{\mathscr{D}}_1;\mathcal{V})$ and $(\widetilde{\varphi}:\mathscr{X}'\to\mathscr{X},\overline{\mathscr{M}}_2)\in\aTheta_{\rm ad}(\overline{\mathscr{D}}_2;\mathcal{V})$ such that
\begin{align*}
 &\sum_{i=0}^{\dim X+1}\binom{\dim X+1}{i}\langle(\overline{\mathscr{D}}_1^{\rm ad};\mathcal{V}_1)^{\cdot i}\cdot(\overline{\mathscr{D}}_2^{\rm ad};\mathcal{V}_2)^{\cdot (\dim X+1-i)}\rangle \\
 &\qquad\qquad\qquad \leq\sum_{i=0}^{\dim X+1}\binom{\dim X+1}{i}\adeg\left(\overline{\mathscr{M}}_1^{\cdot i}\cdot\overline{\mathscr{M}}_2^{\cdot (\dim X+1-i)}\right)+\varepsilon \\
 &\qquad\qquad\qquad =\adeg\left((\overline{\mathscr{M}}_1+\overline{\mathscr{M}}_2)^{\cdot(\dim X+1)}\right)+\varepsilon \\
 &\qquad\qquad\qquad \leq\avol(\overline{D}_1+\overline{D}_2;\mathcal{V}_1+\mathcal{V}_2) +\varepsilon.
\end{align*}

By the same arguments, one can also show the assertions (2), (3), and (4).
\end{proof}

Let $X$ be a normal projective $K$-variety and let $(\overline{D}_1;\mathcal{V}_1),(\overline{D}_2,\mathcal{V}_2)\in\aBVBig_{\RR,\RR}(X)$.
We define
\begin{equation}
 s_i:=\langle(\overline{D}_1;\mathcal{V}_1)^{\cdot i}\cdot(\overline{D}_2;\mathcal{V}_2)^{\cdot(\dim X+1-i)}\rangle
\end{equation}
for $i=0,\dots,\dim X+1$,
\begin{align}
 r &= r((\overline{D}_1;\mathcal{V}_1),(\overline{D}_2;\mathcal{V}_2)) \\
 &:= \inf_{(\mu,\overline{M})\in\aTheta(\overline{D}_2;\mathcal{V}_2)}\sup\left\{t\in\RR\,:\,(\mu^*\overline{D}_1-t\overline{M};\mathcal{V}_1^{\mu})\succeq 0\right\}\nonumber
\end{align}
(inradius), and
\begin{equation}
 R = R((\overline{D}_1;\mathcal{V}_1),(\overline{D}_2;\mathcal{V}_2)):=\frac{1}{r((\overline{D}_2;\mathcal{V}_2),(\overline{D}_1;\mathcal{V}_1))}
\end{equation}
(circumradius).

\begin{lemma}\label{lem:inradius_lemma1}
Let $(\overline{D}_1;\mathcal{V}_1),(\overline{D}_2;\mathcal{V}_2)\in\aBVBig_{\RR,\RR}(X)$.
\begin{enumerate}
\item Let $(\overline{D}_1',\mathcal{V}_1'),(\overline{D}_2',\mathcal{V}_2')\in\aBVBig_{\RR,\RR}(X)$.
If $(\overline{D}_1;\mathcal{V}_1)\preceq(\overline{D}_1';\mathcal{V}_1')$ and $(\overline{D}_2;\mathcal{V}_2)\succeq(\overline{D}_2';\mathcal{V}_2')$, then
\[
 r((\overline{D}_1;\mathcal{V}_1),(\overline{D}_2;\mathcal{V}_2))\leq r((\overline{D}_1';\mathcal{V}_1'),(\overline{D}_2';\mathcal{V}_2')).
\]
\item For $a>0$, one has
\[
 r(a(\overline{D}_1;\mathcal{V}_1),a(\overline{D}_2;\mathcal{V}_2))=ar((\overline{D}_1;\mathcal{V}_1),(\overline{D}_2;\mathcal{V}_2)).
\]
\end{enumerate}
\end{lemma}

\begin{proof}
(1): Obviously, one has
\[
 r((\overline{D}_1';\mathcal{V}_1'),(\overline{D}_2';\mathcal{V}_2'))\geq r((\overline{D}_1;\mathcal{V}_1),(\overline{D}_2';\mathcal{V}_2')).
\]
For any $(\mu,\overline{M})\in\aTheta(\overline{D}_2';\mathcal{V}_2')$, $(\mu,\overline{M})\in\aTheta(\overline{D}_2;\mathcal{V}_2)$; hence
\begin{align*}
 r((\overline{D}_1;\mathcal{V}_1),(\overline{D}_2';\mathcal{V}_2')) &=\inf_{(\mu,\overline{M})\in\aTheta(\overline{D}_2';\mathcal{V}_2')}\sup\left\{t\in\RR\,:\,(\mu^*\overline{D}_1-t\overline{M};\mathcal{V}_1^{\mu})\succeq 0\right\}\\
 &\geq\inf_{(\mu,\overline{M})\in\aTheta(\overline{D}_2;\mathcal{V}_2)}\sup\left\{t\in\RR\,:\,(\mu^*\overline{D}_1-t\overline{M};\mathcal{V}_1^{\mu})\succeq 0\right\} \\
 &=r((\overline{D}_1;\mathcal{V}_1),(\overline{D}_2;\mathcal{V}_2)).
\end{align*}

(2): Note that $(\mu,\overline{M})\in\aTheta(\overline{D}_2;\mathcal{V}_2)$ if and only if $(\mu,a\overline{M})\in\aTheta(a\overline{D}_2;a\mathcal{V}_2)$.
Hence
\begin{align*}
 &ar((\overline{D}_1;\mathcal{V}_1),(\overline{D}_2;\mathcal{V}_2)) \\
 &\qquad\qquad\quad =a\inf_{(\mu,\overline{M})\in\aTheta(\overline{D}_2;\mathcal{V}_2)}\sup\left\{t\in\RR\,:\,(\mu^*\overline{D}_1-t\overline{M};\mathcal{V}_1^{\mu})\succeq 0\right\}\\
 &\qquad\qquad\quad =a\inf_{(\mu,\overline{M})\in\aTheta(\overline{D}_2;\mathcal{V}_2)}\sup\left\{t\in\RR\,:\,(a\mu^*\overline{D}_1-at\overline{M};a\mathcal{V}_1^{\mu})\succeq 0\right\} \\
 &\qquad\qquad\quad =\inf_{(\mu,\overline{M})\in\aTheta(a\overline{D}_2;a\mathcal{V}_2)}\sup\left\{t\in\RR\,:\,(\mu^*(a\overline{D}_1)-t\overline{M};a\mathcal{V}_1^{\mu})\succeq 0\right\} \\
 &\qquad\qquad\quad =r(a(\overline{D}_1;\mathcal{V}_1),a(\overline{D}_2;\mathcal{V}_2)).
\end{align*}
\end{proof}

\begin{lemma}\label{lem:inradius_lemma2}
Let $(\overline{D};\mathcal{V}),(\overline{D}';\mathcal{V}')\in\aBVBig_{\RR,\RR}(X)$, and let $\overline{P}\in\aNef_{\RR}(X)\cap\aBigCone_{\RR}(X)$.
\begin{enumerate}
\item One has
\[
 r((\overline{D};\mathcal{V}),\overline{P})=\sup\left\{t\in\RR\,:\,(\overline{D}-t\overline{P};\mathcal{V})\succeq 0\right\}.
\]
\item One has
\[
 r((\overline{D}+\overline{D}';\mathcal{V}+\mathcal{V}'),\overline{P})\geq r((\overline{D};\mathcal{V}),\overline{P})+r((\overline{D}';\mathcal{V}'),\overline{P}).
\]
\item Let $\overline{E}_i\in\aDiv_{\RR}(X)$, $\mathcal{W}_j\in\VDiv_{\RR}(X)$, $v_k\in M_K\cup\{\infty\}$, and $\varphi_k\in C^0_{v_k}(X)$.
One has
\begin{align*}
 &\lim_{\varepsilon_i,\delta_j,\|\varphi_k\|_{\sup}\to 0}r\Biggl(\Biggl(\overline{D}+\sum_{i=1}^{m}\varepsilon_i\overline{E}_i+\sum_{k=1}^{l}(0,\varphi_k[v_k]);\mathcal{V}+\sum_{j=1}^{n}\delta_j\mathcal{W}_j\Biggr),\overline{P}\Biggr) \\
 &\qquad\qquad\qquad\qquad\qquad\qquad\qquad\qquad\qquad\qquad\qquad =r((\overline{D};\mathcal{V}),\overline{P}).
\end{align*}
\end{enumerate}
\end{lemma}

\begin{proof}
The assertion (1) is obvious.

(2): If $(\overline{D}-t\overline{P};\mathcal{V})\succeq 0$ and $(\overline{D}'-t'\overline{P};\mathcal{V}')\succeq 0$, then
\[
 (\overline{D}+\overline{D}'-(t+t')\overline{P};\mathcal{V}+\mathcal{V}')\succeq 0.
\]
So $r((\overline{D}+\overline{D}';\mathcal{V}+\mathcal{V}'),\overline{P})\geq t+t'$.

(3): The function $\aBVBig_{\RR,\RR}(X)\to\RR$,
\[
 (\overline{D};\mathcal{V})\mapsto r((\overline{D};\mathcal{V}),\overline{P}),
\]
is concave.
So the assertion results from Theorem~\ref{thm:abbig_cone}(1) and \cite[Theorem~6.3.4]{DudleyBook}.
\end{proof}

\begin{lemma}\label{lem:inradius_lemma3}
Let $(\overline{D}_1;\mathcal{V}_1),(\overline{D}_2;\mathcal{V}_2)\in\aBVBig_{\RR,\RR}(X)$, let $U$ be a nonempty open subset of $\Spec(O_K)$ over which a model of definition for $\overline{D}_2$ exists, and let $\delta>0$.
One has
\begin{align*}
 &r((\overline{D}_1;\mathcal{V}_1),(\overline{D}_2;\mathcal{V}_2)) \\
 &\qquad\quad =\inf_{(\mu,\overline{H})\in\aTheta_{\rm amp}(\overline{D}_2;\mathcal{V}_2)}\sup\left\{t\in\RR\,:\,(\mu^*\overline{D}_1-t\overline{H};\mathcal{V}_1^{\mu})\succeq 0\right\} \\
 &\qquad\quad =\inf_{(\mathscr{X},\mathscr{D}_2)\in\aTheta_{U,\delta}(\overline{D}_2)}\inf_{(\widetilde{\mu},\overline{\mathscr{H}})\in\aTheta_{\rm ad}(\overline{\mathscr{D}}_2;\mathcal{V})}\sup\left\{t\in\RR\,:\,(\mu^*\overline{D}_1-t\overline{\mathscr{H}}^{\rm ad};\mathcal{V}_1^{\mu})\succeq 0\right\}.
\end{align*}
\end{lemma}

\begin{proof}
First, we show the first equality.
The inequality $\leq$ is clear.
By definition, given any $\varepsilon>0$, there exists a $(\mu:X'\to X,\overline{M})\in\aTheta(\overline{D}_2;\mathcal{V}_2)$ such that
\begin{equation}
 r((\mu^*\overline{D}_1;\mathcal{V}_1^{\mu}),\overline{M})\leq r((\overline{D}_1;\mathcal{V}_1),(\overline{D}_2;\mathcal{V}_2))+\varepsilon.
\end{equation}
By Lemma~\ref{lem:inradius_lemma2}, there exists a sufficiently small $\delta>0$ such that
\begin{equation}
 r\left((1+\delta)(\mu^*\overline{D}_1;\mathcal{V}_1^{\mu}),\overline{M}\right)\leq r((\mu^*\overline{D}_1;\mathcal{V}_1^{\mu}),\overline{M})+\varepsilon.
\end{equation}
By the same arguments as in Proposition~\ref{prop:approx_by_models}(3), there exists an ample adelic $\RR$-Cartier divisor $\overline{H}$ on $X'$ such that $(\mu,\overline{H})\in\aTheta_{\rm amp}(\overline{D}_2;\mathcal{V}_2)$ and $(1+\delta)\overline{H}\succeq\overline{M}$, so, by Lemma~\ref{lem:inradius_lemma1}(1),(2),
\begin{align}
 r((\mu^*\overline{D}_1;\mathcal{V}_1^{\mu}),\overline{H}) &\leq (1+\delta)r((\mu^*\overline{D}_1;\mathcal{V}_1^{\mu}),\overline{H}) \\
 &=r\left((1+\delta)(\mu^*\overline{D}_1;\mathcal{V}_1^{\mu}),(1+\delta)\overline{H}\right) \nonumber\\
 &\leq r((1+\delta)(\mu^*\overline{D}_1;\mathcal{V}_1^{\mu}),\overline{M}). \nonumber
\end{align}
All in all, one has
\begin{align*}
 \inf_{(\mu,\overline{M})\in\aTheta_{\rm amp}(\overline{D}_2;\mathcal{V}_2)}r((\mu^*\overline{D}_1;\mathcal{V}_1^{\mu}),\overline{M}) &\leq r((\mu^*\overline{D}_1;\mathcal{V}_1^{\mu}),\overline{H}) \\
 &\leq r((\overline{D}_1;\mathcal{V}_1),(\overline{D}_2;\mathcal{V}_2))+2\varepsilon
\end{align*}
for every $\varepsilon>0$.

Next, we show the second equality.
The inequality $\leq$ is clear.
By the assertion (1), given any $\varepsilon>0$, there exists a $(\mu:X'\to X,\overline{H})\in\aTheta_{\rm amp}(\overline{D}_2;\mathcal{V}_2)$ such that $(\mu^*\overline{D}_2-\overline{H};\mathcal{V}_2^{\mu})$ is big and
\begin{equation}
 r((\mu^*\overline{D}_1;\mathcal{V}_1^{\mu}),\overline{H})\leq r((\overline{D}_1;\mathcal{V}_1),(\overline{D}_2;\mathcal{V}_2))+\varepsilon.
\end{equation}
By Lemma~\ref{lem:inradius_lemma2}(2), we can choose a $\delta_1>0$ such that
\begin{equation}
 r((1+\delta_1)(\mu^*\overline{D}_1;\mathcal{V}_1^{\mu}),\overline{H})\leq r((\overline{D}_1;\mathcal{V}_1),\overline{H})+\varepsilon.
\end{equation}

Let $U$ be a nonempty open subset of $\Spec(O_K)$ over which models of definition for both $\overline{D}_2$ and $\overline{H}$ exist.
By Theorem~\ref{thm:abbig_cone}(1), there exists a sufficiently small $\delta_2$ such that $0<\delta_2\leq\delta$ and
\[
 \frac{\delta_1}{1+\delta_1}\overline{H}-\delta_2\sum_{v\in M_K\setminus U}(0,[v])
\]
is big.
By Proposition~\ref{prop:approx_by_models}(4), there exist $(\mathscr{X},\mathscr{D}_2)\in\aTheta_{U,\delta_2}(\overline{D}_2)$ and $(\widetilde{\mu}:\mathscr{X}'\to\mathscr{X},\overline{\mathscr{H}})\in\aTheta_{\rm ad}(\overline{\mathscr{D}}_2;\mathcal{V}_2)$ such that
\[
 \frac{1}{1+\delta_1}\overline{H}\preceq\overline{H}-\delta_2\sum_{v\in M_K\setminus U}(0,[v])\preceq\overline{\mathscr{H}}^{\rm ad}\preceq\overline{H},
\]
so
\begin{align}
 r((\mu^*\overline{D}_1;\mathcal{V}_1^{\mu}),\overline{\mathscr{H}}^{\rm ad}) &\leq(1+\delta_1)r((\mu^*\overline{D}_1;\mathcal{V}_1^{\mu}),\overline{\mathscr{H}}^{\rm ad}) \\
 &\leq r((1+\delta_1)(\mu^*\overline{D}_1;\mathcal{V}_1^{\mu}),\overline{H}). \nonumber
\end{align}
All in all, we have
\begin{align*}
 &\inf_{(\mathscr{X},\mathscr{D}_2)\in\aTheta_{U,\delta}(\overline{D}_2)}\inf_{(\widetilde{\mu},\overline{\mathscr{H}})\in\aTheta_{\rm ad}(\overline{\mathscr{D}}_2;\mathcal{V}_2)}r((\mu^*\overline{D}_1;\mathcal{V}_1^{\mu}),\overline{H}^{\rm ad}) \\
 &\qquad\qquad\qquad\qquad\qquad \leq r((\mu^*\overline{D}_1;\mathcal{V}_1^{\mu}),\overline{\mathscr{H}}^{\rm ad}) \\
 &\qquad\qquad\qquad\qquad\qquad \leq r((\overline{D}_1;\mathcal{V}_1),(\overline{D}_2;\mathcal{V}_2))+2\varepsilon
\end{align*}
for every $\varepsilon>0$.
\end{proof}

The following gives a generalization of \cite[Theorem~7.1 and Corollary~7.3]{IkomaCon} (see also \cite{Bou_Fav_Mat06}, \cite[Theorems~6.9 and 6.10]{Cutkosky13}, \cite{TeissierBonn}).

\begin{theorem}\label{thm:Diskant}
We keep the notations as above.
Let $(\overline{D}_1;\mathcal{V}_1),(\overline{D}_2;\mathcal{V}_2)\in\aBVBig_{\RR,\RR}(X)$.
\begin{enumerate}
\item (An arithmetic Diskant inequality) One has
\[
 0\leq\left(s_{\dim X}^{\frac{1}{\dim X}}-rs_0^{\frac{1}{\dim X}}\right)^{\dim X+1}\leq s_{\dim X}^{1+\frac{1}{\dim X}}-s_{\dim X+1}\cdot s_0^{\frac{1}{\dim X}}.
\]
\item One has
\begin{multline*}
 \frac{s_{\dim X}^{\frac{1}{\dim X}}-\left(s_{\dim X}^{1+\frac{1}{\dim X}}-s_{\dim X+1}\cdot s_0^{\frac{1}{\dim X}}\right)^{\frac{1}{\dim X+1}}}{s_0^{\frac{1}{\dim X}}}\leq r \\
 \leq\frac{s_{\dim X+1}}{s_{\dim X}}\leq\dots\leq\frac{s_1}{s_0}\\
 \leq R\leq\frac{s_{\dim X+1}^{\frac{1}{\dim X}}}{s_1^{\frac{1}{\dim X}}-\left(s_1^{1+\frac{1}{\dim X}}-s_0\cdot s_{\dim X+1}^{\frac{1}{\dim X}}\right)^{\frac{1}{\dim X+1}}}
\end{multline*}
\item (An arithmetic Bonnesen inequality) If $\dim X=1$, then
\[
 \frac{s_0^2}{4}(R-r)^2\leq s_1^2-s_0s_2.
\]
\end{enumerate}
\end{theorem}

\begin{proof}
(1): We divide the proof into three steps.

\paragraph{Step 1.}
In this step, we assume that $(\overline{D}_2;\mathcal{V}_2)$ is given by $\overline{P}\in\aNef_{\RR}(X)\cap\aBigCone_{\RR}(X)$.
By \cite[Corollary~3.24]{IkomaCont}, one has
\[
 \avol(\overline{D}_1-t\overline{P};\mathcal{V}_1)\begin{cases} >0 & \text{if $t<s$,} \\ =0 & \text{if $t=s$.} \end{cases}
\]
Hence, by Theorem~\ref{thm:diff_along_arith},
\begin{equation}\label{eqn:discant1}
 \avol(\overline{D}_1;\mathcal{V}_1)=(\dim X+1)\int_{t=0}^s\langle(\overline{D}_1-t\overline{P};\mathcal{V}_1)^{\cdot \dim X}\rangle\cdot\overline{P}\,dt.
\end{equation}
On the other hand, one has, for $t<s$,
\begin{align}\label{eqn:discant2}
 0&\leq\langle(\overline{D}_1-t\overline{P};\mathcal{V}_1)^{\cdot \dim X}\rangle\cdot\overline{P} \\
 &\leq\left(\left(\langle(\overline{D}_1;\mathcal{V}_1)^{\cdot \dim X}\rangle\cdot\overline{P}\right)^{\frac{1}{\dim X}}-t\avol(\overline{P})^{\frac{1}{\dim X}}\right)^{\dim X} \nonumber
\end{align}
by Corollary~\ref{cor:aKT}(2).
All in all,
\begin{align*}
 &\avol(\overline{D}_1;\mathcal{V}_1)\avol(\overline{P})^{\frac{1}{\dim X}} \\
 &\qquad\quad \leq (\dim X+1)\avol(\overline{P})^{\frac{1}{\dim X}} \\
 &\qquad\qquad\quad \times\int_{t=0}^s\left(\left(\langle(\overline{D}_1;\mathcal{V}_1)^{\cdot\dim X}\rangle\cdot\overline{P}\right)^{\frac{1}{\dim X}}-t\avol(\overline{P})^{\frac{1}{\dim X}}\right)^{\dim X}\,dt\\
 &\qquad\quad =\left(\langle(\overline{D}_1;\mathcal{V}_1)^{\cdot \dim X}\rangle\cdot\overline{P}\right)^{1+\frac{1}{\dim X}} \\
 &\qquad\qquad\quad -\left(\left(\langle(\overline{D}_1;\mathcal{V}_1)^{\cdot \dim X}\rangle\cdot\overline{P}\right)^{\frac{1}{\dim X}}-s\avol(\overline{P})^{\frac{1}{\dim X}}\right)^{\dim X+1}.
\end{align*}

\paragraph{Step 2.}
In this step, we show the following claim.

\begin{claim}\label{clm:Diskant}
Let $(\overline{D}_1;\mathcal{V}_1),(\overline{D}_2;\mathcal{V}_2)\in\aBVBig_{\RR,\RR}(X)$.
For any $i,j\geq 0$ with $i+j\leq\dim X+1$, one has
\[
 r^js_i\leq s_{i+j}.
\]
\end{claim}

\begin{proof}[Proof of Claim~\ref{clm:Diskant}]
We can assume $i>0$.
Given any $\varepsilon$ with $0<\varepsilon< s_i/2$, there exists a $(\mu:X'\to X,\overline{M})\in\aTheta(\overline{D}_2;\mathcal{V}_2)$ such that
\begin{equation}
 \langle(\overline{D}_1;\mathcal{V}_1)^{\cdot i}\cdot\overline{M}^{\cdot j}\cdot(\overline{D}_2;\mathcal{V}_2)^{\cdot (\dim X+1-(i+j))}\rangle +\varepsilon\geq s_i.
\end{equation}
Set $r':=r((\overline{D}_1;\mathcal{V}_1),\overline{M})\geq r$.
Since $(\mu^*\overline{D}_1-r'\overline{M};\mathcal{V}_1^{\mu})\succeq 0$, one has, by Remark~\ref{rem:aPosInt_prop}(2),(3),
\begin{align}
 &{r'}^j\langle(\overline{D}_1;\mathcal{V}_1)^{\cdot i}\cdot\overline{M}^{\cdot j}\cdot(\overline{D}_2;\mathcal{V}_2)^{\cdot (\dim X+1-(i+j))}\rangle \\
 &\qquad\qquad\qquad\qquad\qquad =\langle(\overline{D}_1;\mathcal{V}_1)^{\cdot i}\cdot(r'\overline{M})^{\cdot j}\cdot(\overline{D}_2;\mathcal{V}_2)^{\cdot (\dim X+1-(i+j))}\rangle \nonumber\\
 &\qquad\qquad\qquad\qquad\qquad \leq s_{i+j} \nonumber
\end{align}
Hence,
\begin{align*}
 r^js_i &\leq {r'}^j\cdot\left(\langle(\overline{D}_1;\mathcal{V}_1)^{\cdot i}\cdot\overline{M}^{\cdot j}\cdot(\overline{D}_2;\mathcal{V}_2)^{\cdot (\dim X+1-(i+j))}\rangle +\varepsilon\right) \\
 &\leq s_{i+j}\left(1+\frac{2\varepsilon}{s_i}\right)
\end{align*}
for every sufficiently small $\varepsilon>0$.
\end{proof}

\paragraph{Step 3.}
In general, we take an arbitrary $\varepsilon>0$.
Let $U$ be a nonempty open subset of $\Spec(O_K)$ over which a model of definition for $\overline{D}_2$ exists.
By continuity (see Proposition~\ref{prop:aPosInt_continuous}), there exists a $\delta>0$ such that, for every $(\mathscr{X},\mathscr{D})\in\aTheta_{U,\delta}(\overline{D}_2)$,
\begin{align}
 &\frac{s_{\dim X}^{\frac{1}{\dim X}}-\left(s_{\dim X}^{1+\frac{1}{\dim X}}-s_{\dim X+1}\cdot s_0^{\frac{1}{\dim X}}\right)^{\frac{1}{\dim X+1}}}{s_0^{\frac{1}{\dim X}}} \\
 &\qquad\qquad\qquad \leq\frac{s_{\dim X}^{\prime\frac{1}{\dim X}}-\left(s_{\dim X}^{\prime\left(1+\frac{1}{\dim X}\right)}-s_{\dim X+1}^{\prime}\cdot s_0^{\prime\frac{1}{\dim X}}\right)^{\frac{1}{\dim X+1}}}{s_0^{\prime\frac{1}{\dim X}}}+\varepsilon, \nonumber
\end{align}
where $s_i':=\langle(\overline{D}_1;\mathcal{V}_1)^{\cdot i}\cdot(\overline{\mathscr{D}}^{\rm ad};\mathcal{V}_2)^{\cdot(\dim X+1-i)}\rangle$.

By Lemma~\ref{lem:inradius_lemma3}, there exist a $(\mathscr{X},\mathscr{D}_0)\in\aTheta_{U,\delta}(\overline{D}_2)$ and a $(\widetilde{\varphi}:\mathscr{X}'\to\mathscr{X},\overline{\mathscr{M}}_0)\in\aTheta_{\rm ad}(\overline{\mathscr{D}}_0;\mathcal{V})$ such that
\begin{equation}
 r(\widetilde{\varphi}_*^{-1}(\overline{D}_1;\mathcal{V}_1),\overline{\mathscr{M}}_0^{\rm ad})\leq r((\overline{D}_1;\mathcal{V}_1),(\overline{D}_2;\mathcal{V}_2))+\varepsilon.
\end{equation}

Moreover, by Proposition~\ref{prop:approx_by_models}(2), there exists a $(\widetilde{\mu}:\mathscr{X}''\to\mathscr{X},\overline{\mathscr{M}})\in\aTheta_{\rm ad}(\overline{\mathscr{D}}_0;\mathcal{V})$ such that $\widetilde{\mu}$ factorizes into $\mathscr{X}''\xrightarrow{\widetilde{\psi}}\mathscr{X}'\xrightarrow{\widetilde{\varphi}}\mathscr{X}$, $\overline{\mathscr{M}}\geq\widetilde{\psi}^*\overline{\mathscr{M}}_0$, and
\begin{align}
 &\frac{s_{\dim X}^{\prime\frac{1}{\dim X}}-\left(s_{\dim X}^{\prime\left(1+\frac{1}{\dim X}\right)}-s_{\dim X+1}^{\prime}\cdot s_0^{\prime\frac{1}{\dim X}}\right)^{\frac{1}{\dim X+1}}}{s_0^{\prime\frac{1}{\dim X}}} \\
 &\qquad\qquad\qquad \leq\frac{s_{\dim X}^{\prime\prime\frac{1}{\dim X}}-\left(s_{\dim X}^{\prime\prime\left(1+\frac{1}{\dim X}\right)}-s_{\dim X+1}^{\prime\prime}\cdot s_0^{\prime\prime\frac{1}{\dim X}}\right)^{\frac{1}{\dim X+1}}}{s_0^{\prime\prime\frac{1}{\dim X}}}+\varepsilon, \nonumber
\end{align}
where $s_i'':=\langle(\overline{D}_1;\mathcal{V}_1)^{\cdot i}\rangle\cdot(\overline{\mathscr{M}}^{\rm ad})^{\cdot(\dim X+1-i)}$.

By Step 1 and Lemma~\ref{lem:inradius_lemma1}(1), one has
\begin{align}
 &\frac{s_{\dim X}^{\prime\prime\frac{1}{\dim X}}-\left(s_{\dim X}^{\prime\prime\left(1+\frac{1}{\dim X}\right)}-s_{\dim X+1}^{\prime\prime}\cdot s_0^{\prime\prime\frac{1}{\dim X}}\right)^{\frac{1}{\dim X+1}}}{s_0^{\prime\prime\frac{1}{\dim X}}} \\
 &\qquad\qquad\qquad\qquad\qquad\qquad\qquad\qquad \leq r(\widetilde{\mu}_*^{-1}(\overline{D}_1;\mathcal{V}_1),\overline{\mathscr{M}}^{\rm ad}) \nonumber\\
 &\qquad\qquad\qquad\qquad\qquad\qquad\qquad\qquad \leq r(\widetilde{\varphi}_*^{-1}(\overline{D}_1;\mathcal{V}_1),\overline{\mathscr{M}}_0^{\rm ad}). \nonumber
\end{align}
All in all,
\[
 \frac{s_{\dim X}^{\frac{1}{\dim X}}-\left(s_{\dim X}^{1+\frac{1}{\dim X}}-s_{\dim X+1}\cdot s_0^{\frac{1}{\dim X}}\right)^{\frac{1}{\dim X+1}}}{s_0^{\frac{1}{\dim X}}}\leq r((\overline{D}_1;\mathcal{V}_1),(\overline{D}_2;\mathcal{V}_2))+3\varepsilon
\]
for every $\varepsilon>0$.

(2): By the assertion (1) and Claim~\ref{clm:Diskant}, one has
\[
 \frac{s_{\dim X}^{\frac{1}{\dim X}}-\left(s_{\dim X}^{1+\frac{1}{\dim X}}-s_{\dim X+1}\cdot s_0^{\frac{1}{\dim X}}\right)^{\frac{1}{\dim X+1}}}{s_0^{\frac{1}{\dim X}}}\leq r\leq\frac{s_{\dim X+1}}{s_{\dim X}}.
\]
By applying the above inequalities to $r((\overline{D}_2;\mathcal{V}_2),(\overline{D}_1;\mathcal{V}_1))=1/R$, one has
\[
 \frac{s_1^{\frac{1}{\dim X}}-\left(s_1^{1+\frac{1}{\dim X}}-s_0\cdot s_{\dim X+1}^{\frac{1}{\dim X}}\right)^{\frac{1}{\dim X+1}}}{s_{\dim X+1}^{\frac{1}{\dim X}}}\leq \frac{1}{R}\leq\frac{s_0}{s_1}.
\]
Moreover, by Proposition~\ref{prop:akt2}(2),
\[
 \frac{s_{\dim X+1}}{s_{\dim X}}\leq\frac{s_{\dim X}}{s_{\dim X-1}}\leq\dots\leq\frac{s_2}{s_1}\leq\frac{s_1}{s_0}.
\]
So we conclude.

(3): By the assertion (2), one has
\begin{align*}
 \frac{s_0^2}{4}(R-r)^2 &\leq\frac{s_0^2}{4}\left(\frac{s_2}{s_1-\sqrt{s_1^2-s_0\cdot s_2}}-\frac{s_1-\sqrt{s_1^2-s_0\cdot s_2}}{s_0}\right) \\
 &=s_1^2-s_0s_2.
\end{align*}
\end{proof}

\begin{remark}
\begin{enumerate}
\item If $(\overline{D}_1;\mathcal{V}_1),(\overline{D}_2;\mathcal{V}_2)\in\aBDBig_{\RR,\RR}(X)$ satisfies
\[
 \avol(\overline{D}_1+\overline{D}_2;\mathcal{V}_1+\mathcal{V}_2)^{\frac{1}{\dim X+1}}=\avol(\overline{D}_1;\mathcal{V}_1)^{\frac{1}{\dim X+1}}+\avol(\overline{D}_2;\mathcal{V}_2)^{\frac{1}{\dim X+1}},
\]
then $s_{\dim X}^{\dim X+1}=s_{\dim X+1}^{\dim X}\cdot s_0$, $s_1^{\dim X+1}=s_0^{\dim X}\cdot s_{\dim X}$, and
\[
 \left(\frac{s_{\dim X}}{s_0}\right)^{\frac{1}{\dim X}}=r=\frac{s_{\dim X+1}}{s_{\dim X}}=\dots =\frac{s_1}{s_0}=R=\left(\frac{s_{\dim X+1}}{s_1}\right)^{\frac{1}{\dim X}},
\]
but the converse may not be true.
\item Suppose that $X$ is a smooth curve.
Every discrete valuation is divisorial, so that we can naturally identify the three types of base conditions
\[
 \Div_{\RR}(X)=\WDiv_{\RR}(X)=\VDiv_{\RR}(X).
\]
For $(\overline{D};\mathcal{V})\in\aBVDiv_{\RR,\RR}(X)$,we put
\[
 \Upsilon(\overline{D};\mathcal{V}):=\left\{\overline{P}\,:\,\text{$\overline{P}$ is nef and $\overline{P}\leq (\overline{D};\mathcal{V})$}\right\}.
\]
If $\Upsilon(\overline{D};\mathcal{V})\neq\emptyset$, then it is known that $\Upsilon(\overline{D};\mathcal{V})$ admits a unique maximal element $\overline{P}(\overline{D};\mathcal{V})$ (see \cite[Theorem~6.2.3]{MoriwakiAdelic}).

For $(\overline{D}_1;\mathcal{V}_1),(\overline{D}_2;\mathcal{V}_2)\in\aBVBig_{\RR,\RR}(X)$, the following are equivalent.
\begin{enumerate}
\item $\avol(\overline{D}_1+\overline{D}_2;\mathcal{V}_1+\mathcal{V}_2)^{\frac{1}{2}}=\avol(\overline{D}_1;\mathcal{V}_1)^{\frac{1}{2}}+\avol(\overline{D}_2;\mathcal{V}_2)^{\frac{1}{2}}$.
\item $\overline{P}(\overline{D}_1;\mathcal{V}_1)/\avol(\overline{D}_1;\mathcal{V}_1)^{\frac{1}{2}}\sim_{\RR}\overline{P}(\overline{D}_2;\mathcal{V}_2)/\avol(\overline{D}_2;\mathcal{V}_2)^{\frac{1}{2}}$.
\end{enumerate}
\end{enumerate}
\end{remark}

\appendix
\section{An arithmetic Nakai--Moishezon criterion over curves}

In this appendix, we show an arithmetic Nakai--Moishezon criterion for adelic $\RR$-Cartier divisors on curves (Corollary~\ref{cor:aNM_for_curves}).

\begin{lemma}\label{lem:extend_ample_divisor}
Let $\mathscr{X}$ be a regular and geometrically connected arithmetic surface over $\Spec(O_K)$ with smooth generic fiber $X:=\mathscr{X}_K$.
Let $D$ be an ample $\RR$-Cartier divisor on $X$.
There exist a $\phi\in\Rat(X)^{\times}\otimes_{\ZZ}\RR$ and a relatively ample $\RR$-Cartier divisor $\mathscr{D}$ on $\mathscr{X}$ such that $\mathscr{D}|_X=D+(\phi)$.
\end{lemma}

\begin{proof}
Given any (general) closed point $x'$ in a closed fiber $\mathscr{X}_P$ over $P\in\Spec(O_K)$, one can find an $x\in X$ such that the Zariski closure $\overline{\{x\}}$ in $\mathscr{X}$ contains $x'$.
In fact, let $\varpi\in\mathcal{O}_{\mathscr{X},x'}$ be a local equation defining an irreducible component of $\mathscr{X}_P$ passing through $x'$, and choose an $f\in\mathcal{O}_{\mathscr{X},x'}$ such that $\varpi,f$ form a system of parameters for $\mathcal{O}_{\mathscr{X},x'}$.
Then $f\mathcal{O}_{\mathscr{X},x'}$ is a prime ideal of height one, and does not contain $\varpi$.
(See also \cite[Theorem~4.1]{Gubler_Kunnemann15}.)

There exist finitely many $P_1,\dots,P_l\in\Spec(O_K)$ such that $\mathscr{X}$ is geometrically irreducible over $\Spec(O_K)\setminus\{P_1,\dots,P_l\}$.
By the above argument, there exist $x_1,\dots,x_m\in X$ such that $\bigcup\overline{\{x_i\}}$ meets every irreducible component of $\mathscr{X}_{P_1},\dots,\mathscr{X}_{P_l}$.
We take a $\phi\in\Rat(X)^{\times}\otimes_{\ZZ}\RR$ such that $D+(\phi)$ can be written as a sum $\sum_{j=1}^na_jD_j$ such that $n\geq 1$, $a_j>0$, $D_j$ are prime Cartier divisors on $X$, and $\bigcup\Supp(D_j)\supset\{x_1,\dots,x_m\}$.
Let $\mathscr{D}_j$ be the Zariski closure of $D_j$ in $\mathscr{X}$.
Then
\[
 \mathscr{D}:=\sum_{j=1}^na_j\mathscr{D}_j
\]
is a relatively ample $\RR$-Cartier divisor on $\mathscr{X}$ extending $D+(\phi)$.
\end{proof}

\begin{lemma}\label{lem:aNM_for_curves}
Let $X$ be a smooth and geometrically irreducible $K$-curve and let $\overline{D}=\left(D,\sum_{v\in M_K\cup\{\infty\}}g_v^{\overline{D}}[v]\right)$ be an adelic $\RR$-Cartier divisor on $X$ such that $D$ is ample.
The following are equivalent.
\begin{enumerate}
\item For every $\varepsilon>0$, there exists an $O_K$-model $(\mathscr{X}_{\varepsilon},\mathscr{D}_{\varepsilon})$ of $(X,D)$ such that $(\mathscr{X}_{\varepsilon},\mathscr{D}_{\varepsilon})$ extends a fixed model of definition for $\overline{D}$, $\mathscr{D}_{\varepsilon}$ is a relatively nef $\RR$-Cartier divisor on $\mathscr{X}_{\varepsilon}$, and
\[
 \|g_v^{\overline{D}}-g_v^{(\mathscr{X}_{\varepsilon},\mathscr{D}_{\varepsilon})}\|_{\sup}\leq\varepsilon
\]
for every $v\in M_K$.
\item For every $\varepsilon>0$, there exists an $O_K$-model $(\mathscr{X}_{\varepsilon},\mathscr{D}_{\varepsilon})$ of $(X,D)$ such that $(\mathscr{X}_{\varepsilon},\mathscr{D}_{\varepsilon})$ extends a fixed model of definition for $\overline{D}$, $\mathscr{D}_{\varepsilon}$ is a relatively ample $\RR$-Cartier divisor on $\mathscr{X}_{\varepsilon}$, and
\[
 \|g_v^{\overline{D}}-g_v^{(\mathscr{X}_{\varepsilon},\mathscr{D}_{\varepsilon})}\|_{\sup}\leq\varepsilon
\]
for every $v\in M_K$.
\end{enumerate}
\end{lemma}

\begin{proof}
The implication (2) $\Rightarrow$ (1) is obvious, so that we are going to show the converse.
There exist a nonempty open subset $U$ of $\Spec(O_K)$ and a $U$-model of definition $(\mathscr{X}_U,\mathscr{D}_U)$ for $\overline{D}$ such that $\mathscr{X}_U\to U$ is smooth.

By the condition (1), for any $\varepsilon>0$, there exists an $O_K$-model $(\mathscr{X}_{\varepsilon},\mathscr{D}_{\varepsilon})$ such that $(\mathscr{X}_{\varepsilon},\mathscr{D}_{\varepsilon})$ extends $(\mathscr{X}_U,\mathscr{D}_U)$, $\mathscr{D}_{\varepsilon}$ is a relatively nef $\RR$-Cartier divisor on $\mathscr{X}_{\varepsilon}$, and
\[
 \|g_v^{\overline{D}}-g_v^{(\mathscr{X}_{\varepsilon},\mathscr{D}_{\varepsilon})}\|_{\sup}\leq\frac{\varepsilon}{2}
\]
for every $v\in M_K$.

By desingularization \cite[page 55]{Lipman78}, there exists a birational morphism $\pi_{\varepsilon}:\mathscr{X}_{\varepsilon}'\to\mathscr{X}_{\varepsilon}$ such that $\mathscr{X}_{\varepsilon}'$ is regular and $\pi_{\varepsilon}$ is isomorphic over $U$.
By Lemma~\ref{lem:extend_ample_divisor}, there exists a relatively ample $\RR$-Cartier divisor $\mathscr{D}_{\varepsilon}'$ on $\mathscr{X}_{\varepsilon}'$ that extends $D$.
We can choose a sufficiently small $\delta_{\varepsilon}>0$ such that
\[
 \left\|g_v^{\overline{D}}-\left((1-\delta_{\varepsilon})g_v^{(\mathscr{X}_{\varepsilon},\mathscr{D}_{\varepsilon})}+\delta_{\varepsilon} g_v^{(\mathscr{X}_{\varepsilon}',\mathscr{D}_{\varepsilon}')}\right)\right\|_{\sup}\leq\varepsilon
\]
for every $v\in M_K$.
So $(\mathscr{X}_{\varepsilon}',(1-\delta_{\varepsilon})\pi_{\varepsilon}^*\mathscr{D}_{\varepsilon}+\delta_{\varepsilon}\mathscr{D}_{\varepsilon}')$ is an $O_K$-model of $(X,D)$ having the required properties.
\end{proof}

The following is an arithmetic analogue of the theorem of Campana--Peternell \cite[Theorem~1.3]{Campana_Peternell90}.

\begin{theorem}\label{thm:aNM_for_R}
Let $X$ be a smooth projective $K$-variety and let $\overline{D}=\left(D,\sum_{v\in M_K\cup\{\infty\}}g_v^{\overline{D}}[v]\right)$ be an adelic $\RR$-Cartier divisor on $X$ such that the following condition ($\ast$) is satisfied.
\begin{enumerate}
\item[($\ast$)] For every $\varepsilon>0$, there exists an $O_K$-model $(\mathscr{X}_{\varepsilon},\mathscr{D}_{\varepsilon})$ of $(X,D)$ such that $(\mathscr{X}_{\varepsilon},\mathscr{D}_{\varepsilon})$ extends a fixed model of definition for $\overline{D}$, $\mathscr{D}_{\varepsilon}$ is a relatively ample $\RR$-Cartier divisor on $\mathscr{X}_{\varepsilon}$, and
\[
 \|g_v^{\overline{D}}-g_v^{(\mathscr{X}_{\varepsilon},\mathscr{D}_{\varepsilon})}\|_{\sup}\leq\varepsilon
\]
for every $v\in M_K$.
\end{enumerate}
Then the following are equivalent.
\begin{enumerate}
\item $\overline{D}$ is w-ample.
\item $\overline{D}$ is ample
\end{enumerate}
\end{theorem}

\begin{proof}
The implication (1) $\Rightarrow$ (2) is nothing but Theorem~\ref{thm:ample}(1), so that it suffices to show the converse.

(2) $\Rightarrow$ (1): If $\dim X=0$, then $\overline{D}$ can be written as $\left(0,\sum_{v\in M_K\cup\{\infty\}}\lambda_v[v]\right)$, where $\lambda_v=0$ for all but finitely many $v\in M_K$.
By the arithmetic Riemann-Roch formula, there exists a $\phi\in\Rat(X)^{\times}\otimes_{\ZZ}\RR$ such that $\overline{D}+\widehat{(\phi)}>0$.

We show the theorem by induction on dimension.
We can assume that $\overline{D}$ is associated to an arithmetic $\RR$-Cartier divisor $\overline{\mathscr{D}}$ on a normal $O_K$-model $\mathscr{X}$ such that $\mathscr{D}$ is relatively ample, such that $g_{\infty}^{\overline{\mathscr{D}}}$ is of $C^{\infty}$-type and positive pointwise, and such that
\begin{equation}
 \inf_{x\in X(\overline{K})}h_{\overline{\mathscr{D}}}(x)>0.
\end{equation}
In fact, by Theorem~\ref{thm:ample}(3), we have $\lambda:=\inf_{x\in X(\overline{K})}h_{\overline{D}}(x)>0$.
Let $U$ be an open subset of $\Spec(O_K)$ over which the fixed model of definition $(\mathscr{X}_U,\mathscr{D}_U)$ of $\overline{D}$ exists.
Let $\varepsilon$ be a positive real number such that
\[
 \varepsilon\left(\sharp(\Spec(O_K)\setminus U)+[K:\QQ]\right)<\lambda.
\]
By the condition ($\ast$) and the regularization theorem (see \cite[Theorem~4.6]{MoriwakiZar}), there exist a normal $O_K$-model $\mathscr{X}_{\varepsilon}$ and an arithmetic $\RR$-Cartier divisor $\overline{\mathscr{D}}_{\varepsilon}$ on $\mathscr{X}_{\varepsilon}$ such that $\mathscr{D}_{\varepsilon}$ is relatively ample, such that $g_{\infty}^{\overline{\mathscr{D}}_{\varepsilon}}$ is of $C^{\infty}$-type and positive pointwise, and such that
\[
 \overline{\mathscr{D}}_{\varepsilon}^{\rm ad}\leq\overline{D}\leq\overline{\mathscr{D}}_{\varepsilon}^{\rm ad}+2\varepsilon\sum_{v\in(M_K\setminus U)\cup\{\infty\}}(0,[v]).
\]
Note that
\[
 \inf_{x\in X(\overline{K})}h_{\overline{\mathscr{D}}_{\varepsilon}}(x)\geq\lambda-\left(\sharp(\Spec(O_K)\setminus U)+[K:\QQ]\right)\varepsilon>0,
\]
and that, if $\overline{\mathscr{D}}_{\varepsilon}^{\rm ad}$ is w-ample, then so is $\overline{D}$.

By induction hypothesis, $\overline{\mathscr{D}}|_{\mathscr{Y}}$ is w-ample for every horizontal arithmetic subvariety $\mathscr{Y}$ with $\dim\mathscr{Y}<\dim\mathscr{X}$.
Let $\overline{\mathscr{A}}_1,\dots,\overline{\mathscr{A}}_l$ be nef and w-ample arithmetic Cartier divisors of $C^{\infty}$-type on $\mathscr{X}$ such that $\mathscr{D}$ is contained in the rational $\RR$-subspace of $\Div_{\RR}(\mathscr{X})$ spanned by $\mathscr{A}_1,\dots,\mathscr{A}_l$.
There exist positive real numbers $\varepsilon_1,\dots,\varepsilon_l$ such that
\begin{equation}
 \overline{\mathscr{E}}:=\overline{\mathscr{D}}-\sum_{i=1}^l\varepsilon_i\overline{\mathscr{A}}_i
\end{equation}
is rational, $\mathscr{E}$ is ample, $g_{\infty}^{\overline{\mathscr{E}}}$ is positive pointwise, and
\[
 (\dim X+1)\cdot\sum_{i=1}^l\varepsilon_i\adeg\left(\overline{\mathscr{D}}^{\cdot\dim X}\cdot\overline{\mathscr{A}}_i\right)<\avol(\overline{\mathscr{D}}).
\]
Set $\overline{\mathscr{A}}:=\sum_{i=1}^l\varepsilon_i\overline{\mathscr{A}}_i$.
By the arithmetic Siu inequality (Proposition~\ref{prop:Yuan}(1)), we have
\[
 \avol(\overline{\mathscr{E}})\geq\avol(\overline{\mathscr{D}})-(\dim X+1)\cdot\adeg\left(\overline{\mathscr{D}}^{\cdot\dim X}\cdot\overline{\mathscr{A}}\right)>0,
\]
so that there exists an $s\in\aHzq{\QQ}(\overline{\mathscr{E}}^{\rm ad})\setminus\{0\}$.

Let $\mathscr{Y}_1,\dots,\mathscr{Y}_r$ be the reduced, irreducible, and horizontal components of $\CSupp(\mathscr{E}+(s))$.
Since $\overline{\mathscr{D}}|_{\mathscr{Y}_1},\dots,\overline{\mathscr{D}}|_{\mathscr{Y}_r}$ are all w-ample, one finds a sufficiently small $\delta>0$ such that $\mathscr{D}-\delta\mathscr{A}$ is relatively ample and
\[
 \overline{\mathscr{D}}|_{\mathscr{Y}_1}-\delta\overline{\mathscr{A}}|_{\mathscr{Y}_1},\dots,\overline{\mathscr{D}}|_{\mathscr{Y}_r}-\delta\overline{\mathscr{A}}|_{\mathscr{Y}_r}
\]
are all w-ample (see Lemma~\ref{lem:w-ample}(4)).

Set
\begin{equation}
 \overline{\mathscr{F}}:=\overline{\mathscr{D}}-\delta\overline{\mathscr{A}}=\overline{\mathscr{E}}+(1-\delta)\overline{\mathscr{A}}.
\end{equation}
Let $x\in X(\overline{K})$.
If $x\notin \bigcup_{j=1}^p\mathscr{Y}_j(\overline{K})$, then $h_{\overline{\mathscr{F}}}(x)\geq h_{\overline{\mathscr{E}}}(x)>0$.
If $x\in \mathscr{Y}_j(\overline{K})$ for a $j$, then $h_{\overline{\mathscr{F}}}(x)>0$ since $\overline{\mathscr{F}}|_{\mathscr{Y}_j}$ is w-ample.
So $\overline{\mathscr{F}}$ is nef and $\mathscr{F}$ is contained in the rational $\RR$-subspace spanned by $\mathscr{A}_1,\dots,\mathscr{A}_l$.
Hence, by Theorem~\ref{thm:ample}(5), we conclude that
\[
 \overline{\mathscr{D}}^{\rm ad}=\overline{\mathscr{F}}^{\rm ad}+\delta\sum_{i=1}^l\varepsilon_i\overline{\mathscr{A}}_i^{\rm ad}
\]
is w-ample.
\end{proof}

As a consequence of Theorem~\ref{thm:aNM_for_R} and Lemma~\ref{lem:aNM_for_curves}, we have the following.

\begin{corollary}\label{cor:aNM_for_curves}
Let $X$ be a smooth $K$-curve and let $\overline{D}$ be an adelic $\RR$-Cartier divisor on $X$.
The following are equivalent.
\begin{enumerate}
\item $\overline{D}$ is ample.
\item $\overline{D}$ is w-ample and relatively nef.
\end{enumerate}
\end{corollary}

\bibliography{ikoma}
\bibliographystyle{plain}

\end{document}